\documentclass[a4paper]{article}

\usepackage[hidelinks]{hyperref}

\usepackage{amsmath,amssymb,mathtools,amsthm}
\usepackage{xcolor,graphicx,float,enumerate}

\usepackage[font=small]{caption}
\usepackage{subcaption}

\usepackage[utf8]{inputenc}
\DeclareUnicodeCharacter{00A0}{ }
\usepackage[a4paper]{geometry}

\newcommand{\ee}{\varepsilon}
\newcommand{\G}{\mathbb{G}}
\newcommand{\Gnu}{D_\gamma \G}
\newcommand{\M}{\mathbb{M}}
\newcommand{\N}{\mathbb{N}}
\newcommand{\cM}{{\mathcal M}}
\newcommand{\WM}{\cM(\Omega, \delta^\gamma)}
\newcommand{\SFL}{\mathrm{SFL}}
\newcommand{\RFL}{\mathrm{RFL}}
\newcommand{\CFL}{\mathrm{CFL}}
\newcommand{\Ls}{\mathrm{L}}
\newcommand{\Green}{\mathcal{G}}
\newcommand{\Martin}{\mathcal{M}}

\newcommand{\pv}{\mathrm{p.v.}\!}

\DeclareMathOperator{\sign}{sign}
\DeclareMathOperator{\supp}{supp}

\newtheorem{proposition}{Proposition}[section]
\newtheorem{theorem}[proposition]{Theorem}
\newtheorem{corollary}[proposition]{Corollary}
\newtheorem{lemma}[proposition]{Lemma}
\theoremstyle{definition}\newtheorem{definition}[proposition]{Definition}
\newtheorem{example}[proposition]{Example}
\newtheorem{remark}[proposition]{Remark}

\numberwithin{equation}{section}

\begin{document}
\title{Singular boundary behaviour and large solutions for fractional elliptic equations}

\author{Nicola Abatangelo\thanks{Institut f\"ur Mathematik, Goethe-Universit\"at Frankfurt am Main. \url{abatange@math.uni-frankfurt.de}} \and David G\'omez-Castro\thanks{Instituto de Matem\'atica Interdisciplinar, Universidad Complutense de Madrid. \url{dgcastro@ucm.es}} \and Juan Luis V\'azquez\thanks{Departamento de Matem\'aticas, Universidad Aut\'onoma de Madrid. \url{juanluis.vazquez@uam.es}}}
\maketitle

\abstract{We show that the boundary behaviour of solutions to
nonlocal fractional equations posed in bounded domains strongly differs from the one of solutions
to elliptic problems modelled upon the Laplace-Poisson equation with zero boundary data.
In this classical case it is known that, at least in a suitable weak sense, solutions of non-homogeneous Dirichlet problem are unique and tend to zero at the boundary.
	Limits of these solutions then produce solutions of some non-homogeneous Dirichlet problem as the interior data concentrate suitably to the boundary.
	Here, we show that such results are false for equations driven by a wide class of nonlocal fractional operators, extending previous findings for some models of the fractional Laplacian operator. Actually, different blow-up phenomena may occur at the boundary of the domain. We describe such explosive behaviours and obtain precise quantitative estimates depending on simple parameters of the nonlocal operators. Our unifying technique is based on a careful study of the inverse operator in terms of the corresponding Green function.
	}

\tableofcontents

\newpage

\section{Introduction}
\normalsize
In recent years there have been many studies on boundary value problems
driven by nonlocal operators $\Ls$ obtained as fractional powers of uniformly elliptic operators,
such as the Laplacian. In this context, according to the ``degree of nonlocality''
of the leading operator in the differential equation,
additional values need to be prescribed either on the boundary of the underlying domain
or on its whole complement.
So, given a regular bounded domain $\Omega\subseteq\mathbb R^n$,
the simplest complete problems take the form of an equation
\begin{equation}
\label{eq:FDE}
\Ls u = f \qquad \text{in }\Omega,
\end{equation}
complemented by homogeneous values
\begin{equation}
\label{eq:BV}
u=0 \qquad \text{on }\partial\Omega,\text{ or in }\mathbb R^n\setminus\overline\Omega,
\end{equation}
the last choice depending on the nonlocal operator $\Ls$.
Usually,~\eqref{eq:FDE} is written in some weak form which also encodes~\eqref{eq:BV}.
In the standard elliptic theory,~\eqref{eq:BV} can be replaced by
$u=g$ on $\partial\Omega$,  for $g$  an $L^p$ function  and therefore a.e. finite on $\partial\Omega$.
In this paper we study solutions to equations of the form~\eqref{eq:FDE} that develop
an explosive behaviour at the boundary, \textit{i.e.}, solutions satisfying
\begin{align*}
u(x)\to+\infty,
\qquad\text{as }x\to x_0, \text{ for almost all }x_0\in\partial\Omega.
\end{align*}
They are usually called \textit{large solutions}
 and they account for
a new phenomenon, not appearing in the classical elliptic theory.
We will show that large solutions are tightly connected to the solutions
of the homogeneous problem via a natural limiting process.
Finally, they exhibit quite peculiar divergence rates that we will derive.
All of this will be done for a specific class of
nonlocal operators $\Ls$ that includes the usual examples and more.\medskip

The present research is motivated by two striking results involving singular behaviour near the boundary for the solutions of~\eqref{eq:FDE} in the case where $\Ls$ is the so-called restricted fractional Laplacian (for which one has to prescribe data on~$\mathbb R ^n \setminus \overline \Omega$). \medskip

 One of these striking results is the existence of nontrivial solutions of~\eqref{eq:FDE} such that $f = 0$ in $\Omega$ which, moreover, are positive everywhere and blow-up on the boundary. Explicit examples on the ball were constructed in \cite{MR1295711} (see also \cite{bogdan99representation,bogdan2009potential}). This is in contrast with the case of the standard Laplacian where no boundary blow-up solution exists.
The existence of this kind of solutions was systematised
independently in \cite{grubb15} (which also contains a thorough regularity theory,
see also \cite{grubb14} for related results and \cite{grubb18} for a review) and
in \cite{Abatangelo2015} for the restricted fractional Laplacian, and extended in \cite{Abatangelo2017a} for the spectral fractional Laplacian (which requires prescribed data at the boundary).\medskip

The second striking result, described in \cite{Abatangelo2015} when $\Ls$ is the restricted fractional Laplacian and in \cite{Abatangelo2017a} when $\Ls$ is the spectral fractional Laplacian, is that some admissible functions $f$ produce solutions $u$ blowing-up at the boundary, although they are limits of solutions with ``nice'' $f$ and zero boundary data. This is as well a new behaviour of the nonlocal problem, not present for the usual Laplacian. \medskip

For the case of the usual Laplacian, it is known that
the wider classes of weak or very weak solutions obtained as limits of the variational solutions satisfy the boundary condition either in the sense of traces or in a more generalised sense, described in \cite{Ponce2016a} as the average condition
\begin{equation}
	\label{eq:Ponce introduction}
	\eta^{-1} \int_{\{\mathrm{dist}(x,\partial\Omega) < \eta \}} |u| \to 0 \qquad \textrm{ as } \eta \downarrow 0.
\end{equation}

The aim of the present work is to show that these two blow-up phenomena occur for a large class of nonlocal operators of elliptic type. We  treat in a unified way the typical nonlocal elliptic equations, in particular the different fractional Laplacians on bounded domains. Our distinctive technique is based on the use of the Green kernel which gives a common roof to the several different cases. This approach extends previous work in \cite{bonforte+figalli+vazquez2018,GC+Vazquez2018}.
		
	We consider a general family of operators indexed on two parameters: one describing the interior point singularity of the Green kernel, the other one the kernel's boundary behaviour. This requires serious technical work, that justifies the extension of the paper.

First, we want to study and classify the explosive (or \textit{large}) solutions
whose singularity is, in some sense, generated by the right-hand side $f$.
	In particular, we compute explicitly the asymptotic boundary behaviour
	of $u$ for the family of power-like data $f \asymp \delta^\beta$ near the boundary, where $\delta(x) := \textrm{dist}(x,\partial \Omega)$. Here, we say that $f \asymp g$ on a set if there exists $C>0$ such that $C^{-1} g \le f \le C g$ on that set. Our main formula~\eqref{eq:boundary behaviour introduction} gives the behaviour of $u$ in terms of $f$ and the kernel of $\Ls$ in simple algebraic terms.
	The formula covers the whole range of behaviours, explosive or not.
	We also translate estimate~\eqref{eq:Ponce introduction} to our context by introducing a suitable weight, taking care of the singular profiles (see Lemma~\ref{lem:subritical traces}).
	We provide some careful numerical computations, to show the formation of the boundary singularity due to the right-hand side (see Figures \ref{fig:subcritical} and \ref{fig:towards Martin}). 

	Even if  the solution operator for the Dirichlet problem acting on a class of good functions $f$ produces solutions with Dirichlet boundary data, we show that the natural closure of that solution operator to its maximal domain of definition produces solutions which no longer satisfy the Dirichlet condition and could reach a range of boundary blow-up that we describe.
	In the case of problems in which the boundary condition is set (\textit{e.g.}, the spectral fractional Laplacian), this is counter-intuitive.
	The occurrence of boundary blow-up is a very important fact, that does not happen for the usual Laplacian.

Secondly, we remark how there is a different class of explosive solutions
whose singularity is not generated by any right-hand side.
	In fact, they can be chosen as ``$\Ls$-harmonic in $\Omega$'' in the sense $\Ls w = 0$.
This class relies on some hidden information in the form of singular behaviour that can be prescribed on the boundary. Moreover, this second class can be obtained as a limit of singular solutions of the previous class  as the support of $f$ concentrates at the boundary in a convenient way.  This means they cannot be disregarded in any complete theory of the problem. See the detailed results in Section~\ref{sec:Martin problem}.

We conclude this introduction with an important remark.
If a definition of solution of~\eqref{eq:FDE} is ``too weak'', then the combination of the two classes seems to pose a problem to uniqueness.
This highlights the importance of a suitable definition of weak solution of~\eqref{eq:FDE} preserving uniqueness and including the classical solutions.
We provide this definition in Section~\ref{sec:interior} under the name of \emph{weak-dual solution}, and show that the problem is then well-posed.
We also detect the optimal class of admissible data $f$.
To take care of the second class, we construct a ``singular boundary data'' problem. We give a well-posed notion of solution for this second problem:
uniqueness is the easy part.


\subsection{Main topics and results}

\paragraph{Existence and uniqueness results for~\eqref{eq:FDE}.}
To begin with, we need to produce a general existence theory for data $f$ in good classes, \textit{i.e.}, compactly supported and bounded. We want to treat a general class of operators
	such that the unique solution of~\eqref{eq:FDE} is given by the formula
 	\begin{equation}
	\label{eq:K0}
	\tag{K0}
	\Green (f) (x)= \int_ \Omega \G(x,y) f(y) \; dy
	\end{equation}
	with kernels $\G:\Omega\times\Omega\to\mathbb R$ such that
	\begin{align}
	\label{eq:K1}
	\tag{K1}
	\G(x,y) &= \G(y,x) \qquad \text{and} \\
	\label{eq:K2}
	\tag{K2}
	\G (x,y) &\asymp \frac{1}{|x-y|^{n-2s}} \left( \frac{\delta(x) \delta(y)}{|x-y|^2} \wedge 1  \right)^\gamma.
	\end{align}
	The two exponents $s$ and $\gamma$ take values
	\begin{align}\label{eq:K3}\tag{K3}
	s,\gamma \in (0,1],
	\qquad\text{with}\qquad 2s\leq n.
	\end{align}
Their relative values will play an important role in the results.

Throughout this note we use the notation $a \wedge b = \min \{a,b\}$, $a \vee b = \max \{a,b\}$.

	Some notable examples of this general class of operators are the three most known fractional Laplacian operators:
	\begin{enumerate}[\it i.]
		\item The restricted fractional Laplacian (RFL): in this case $\gamma = s\in(0,1)$;
		\item The spectral fractional Laplacian (SFL), for which $\gamma = 1$ and $s\in(0,1)$;
		\item The regional or censored fractional Laplacian (CFL) which has $\gamma = 2s-1$ and $s\in(1/2,1)$.
	\end{enumerate}
	These examples will be presented in some more detail in Section~\ref{sec:examples of operators} so that we can adapt to them the general results.

	In Section~\ref{sec:interior} we prove existence, uniqueness, a priori estimates, 	and some regularity for problem~\eqref{eq:FDE}.
	In Section~\ref{sec:breakdown boundary condition},
	we prove that the optimal class of data $f$ such that~\eqref{eq:K0} is well defined
	(meaning $\Green(|f|) \not \equiv +\infty$) is
	\begin{equation}
		\label{eq:admissible class introduction}
		f \in L^1 (\Omega, \delta^\gamma) = \{ f \textrm{ measurable in }  \Omega: f \delta^\gamma \in L^1 (\Omega) \}.
	\end{equation}

\paragraph{Boundary behaviour.} As we mentioned, for the standard Laplacian $-\Delta$, the zero boundary data are taken in the some sense even when $f$ is taken in the optimal class of data. The sense depends on how good is $f$, see the general results in \cite{Ponce2016a}. A quite novel property of the restricted fractional Laplacian on bounded domains shows that this is not true for admissible $f$ even if they are not so badly behaved. This is explained in \cite{Abatangelo2015} and we want to extend the analysis to our general class of operators and show the detailed relation between the operators, the boundary behaviour of $f$, and the singular boundary behaviour of the solution. The main information about the operators will be the values of $\gamma$ and $s$.
	
	 	In Theorem~\ref{thm:range of exponents} we establish the explicit estimate
	\begin{equation}\label{eq:boundary behaviour introduction}
		\Green(\delta^{\beta}) \asymp \delta^{\gamma \wedge ( \beta + 2s )} \qquad \text{whenever }\gamma + \beta > -1 \text{ and }  \beta \ne \gamma - 2s
	\end{equation}
	that needs a delicate computation using the properties of the kernel.
	This is depicted in Figure~\ref{fig:alpha-vs-beta}.
	Notice that $\gamma + \beta > -1$ is the condition so that
	$f = \delta^\beta$ belongs to the admissible class given by~\eqref{eq:admissible class introduction}.
	
	\begin{figure}[th]
		\centering
		\includegraphics{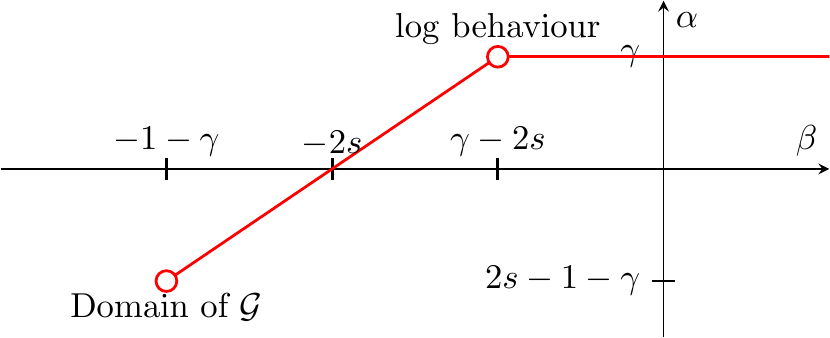}
		\caption{Relation of parameters $\alpha$ and $\beta$ such that $\Green(\delta^{\beta}) \asymp \delta^\alpha$.}
		\label{fig:alpha-vs-beta}
	\end{figure}

	In many cases, the existence of eigenfunctions is known, and their boundary behaviour is well understood.
	Under~\eqref{eq:K0},~\eqref{eq:K2}, and some extra assumptions on the operator $\Ls$, the authors in \cite{bonforte+figalli+vazquez2018}
	proved that the operator $\Green$ admits an eigendecomposition and its first eigenfunction $\Phi_1$ satisfies
		\begin{equation}
			\label{eq:estimate first eigenfunction}
			\Phi_1 \asymp \delta^\gamma \qquad \text{in }\Omega.
		\end{equation}
	The boundary behaviour is clear from the algebraic point of view, since $\gamma$ is the only exponent fixed by $\Green$.

\paragraph{Solutions with singular behaviour.}
	We observe that, according to  formula~\eqref{eq:boundary behaviour introduction}, there are values of $\beta$ for which the solution associated to datum $\delta^\beta$ is singular at the boundary: this happens whenever $\beta\in(-1-\gamma,-2s)$ is allowed, and therefore when $\gamma>2s-1$. In particular, it  comes out that if $ \gamma > 2s - 1 $ then there exist solutions of the Dirichlet problem not complying with the condition $u = 0$ on the boundary. This was known for the RFL \cite[Proposition 3]{Abatangelo2015} and the SFL \cite[Proposition 7]{dhifli2012subordinate}.

 	The behaviour $\delta^{\gamma\wedge(2s-\gamma-1)}$, corresponding to the limit case $\beta=-1-\gamma$, serves somehow as an upper bound for solutions.
	In Lemma~\ref{lem:subritical traces} we will prove that
	\begin{enumerate}[\rm a)]
	\item If $\gamma > s - 1/2$, then for any $f\in L^1(\Omega,\delta^\gamma)$
		\begin{equation*}
			\frac1\eta \int_{ \{  \delta < \eta  \} } \frac{\Green(f)}{\delta^{2s-\gamma-1}} \longrightarrow 0 \qquad\text{as }\eta\downarrow 0.
		\end{equation*}
	\item If $\gamma < s - 1/2$, then for any nonnegative $f\in L^1(\Omega,\delta^\gamma)$ and $\eta>0$
		\begin{equation*}
			\frac1\eta \int_{ \{  \delta < \eta  \} } \frac{\Green(f)}{\delta^\gamma} \asymp 1.
		\end{equation*}
	\end{enumerate}
	We also prove that, in the case $\gamma = s - 1/2$, there is a logarithmic correction.
	
	For the usual Laplacian, when $s=\gamma=1$, we have $0=2s- \gamma - 1<\gamma$: this reproduces~\eqref{eq:Ponce introduction}. This same fact holds for the CFL, because $\gamma=2s-1$. If $2s-\gamma - 1 > 0$ then all solutions tend to $0$ upon approaching the boundary.
	
	 The two conditions $\gamma > 2s-  1 $ and $\gamma > s - \frac 1 2$ allow us to split the parameter $s, \gamma$ in three regions as in Figure~\ref{fig:gamma-vs-s}.
		\begin{figure}[ht]
			\centering
			\includegraphics[width=.3\textwidth]{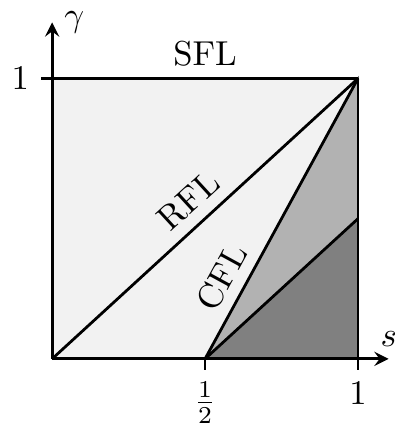}
			\caption{Different relations between $\gamma$ and $s$.}
			\label{fig:gamma-vs-s}
		\end{figure}

\paragraph{Normal derivatives.} A sharpest study of the boundary behaviour of solution with data $f \in L^\infty_c (\Omega)$ consists of the analysis of the limit
	\begin{equation*}
		D_\gamma u (z) := \lim_{\substack{x \to z \\ x \in \Omega}} \frac{u(x)}{\delta(x)^\gamma},
		\qquad z\in\partial\Omega.
	\end{equation*}
	We will call this limit $\gamma$-\textit{normal derivative}.
	We devote Section~\ref{sec:sharp boundary behaviour good data} to the study of these normal derivatives
	(see Theorem~\ref{thm:gamma normal derivative}).

\paragraph{Large solutions.} In \cite{bogdan2009potential} the authors introduce a surprising singular solution of the homogeneous problem $f=0$ that shows very precise asymptotics at the boundary. It is the type known as \emph{large solution} in other situations for nonlinear equations.
For example, the function
\begin{equation*}
u(x):=\begin{dcases}
(1-|x|^2)^{s-1} & \text{for }|x|<1 \\
0 & \text{for }|x|\geq 1
\end{dcases}
\end{equation*}
is known to satisfy ${(-\Delta)}^s_\RFL u(x)=0$ for $|x|<1$, see \cite[Example 1]{bogdan99representation} and \cite{MR1295711}.
In \cite{Abatangelo2015} there is a complete description of the singular boundary value problem for the RFL,
while in \cite{Abatangelo2017a} there is the analogue for the SFL.\medskip

We prove that this theory may be obtained as a limit of interior problems.
We construct one such particular large solution $u^\star$ which is $\Ls$-harmonic on the interior ($\Ls u^\star=0$ in $\Omega$).
In Section~\ref{sec:Martin problem}, we show that there exists a sequence of admissible functions ${(f_j)}_{j\in\mathbb N}$
(with $\textrm{dist}(\supp f_j, \partial \Omega ) < 2/j$) such that
	\begin{equation*}
		\Green(f_j) \rightharpoonup u^\star \qquad \textrm{ in } L^1_{loc} (\Omega),\text{ as }j\uparrow\infty.
	\end{equation*}
	This limit function has the boundary behaviour
	\begin{equation*}
		\label{eq:singular solution}
		u^\star \asymp \delta^{(2s-\gamma - 1) \wedge \gamma} \qquad\text{in }\Omega,
	\end{equation*}
	except in the case $\gamma=s-1/2$ when a logarithmic correction is in order.

	Notice that the exponent is the upper bound of the range in~\eqref{eq:boundary behaviour introduction}.
	We will prove that problems
	\begin{equation*}
		\int_ \Omega u \, \psi = \int_{\partial\Omega}  h \, D_\gamma [\Green(\psi)] \qquad \text{for any } \psi \in L^\infty_c (\Omega).
	\end{equation*}	
	have a unique solution $u$,
	which is comparable to $u^\star$ at the boundary.
	When $\gamma>s-\frac12$, this is even the unique weak-dual solution of problem
	\begin{equation}
		\label{eq:Martin}
		\begin{dcases}
			\Ls u = 0 & \text{in }\Omega, \\
			u = 0 & \text{in }\mathbb R^N\setminus\overline\Omega \textrm{ (if applicable)}, \\
			\frac{u}{u^\star} = h & \text{on }\partial \Omega.
		\end{dcases}
	\end{equation}

\paragraph{Comments.}
	Our presentation unifies in a single theory previous results
	for the RFL ($u^\star \asymp \delta^{s-1}$, see \cite{Abatangelo2015}),
	the SFL ($u^\star \asymp \delta^{2(s-1)}$, see \cite{Abatangelo2017a}),
	and the CFL ($u^\star \asymp 1$, see \cite{Chen2018}).\medskip

	The case
		$\gamma < s- \frac 1 2 $, which does not include any of the main known examples, is somewhat particular.	
	In this case, due to~\eqref{eq:singular solution}, $0 < u^\star (x) \to 0$ as $x \to \partial \Omega$ and it is a non-trivial solution of~\eqref{eq:FDE} with data $f = 0$. This yields some doubt about the uniqueness of solutions to~\eqref{eq:FDE}-\eqref{eq:BV}. Furthermore, if $\gamma \le s - \frac 1 2$ then $u^\star \asymp \delta^\gamma$, which in turn means that the critical solutions have the same boundary behaviour as the solutions for regular data $f$. This does not seem to be consistent with elliptic problems like~\eqref{eq:FDE}.

\subsection{Some examples}
\label{sec:examples of operators}

Large classes of operators $\Ls$ have Green operators $\Green$ given by
\eqref{eq:K0}--\eqref{eq:K3}: here are some notorious examples that are reviewed for instance in \cite{av-bari,bonforte+figalli+vazquez2018, MR3588125}.

\subsubsection*{The restricted fractional Laplacian (RFL)}

The restricted fractional Laplacian is defined as the singular integral operator
\begin{equation*}
	(-\Delta)^s_\RFL u (x)= \pv\int_{ \mathbb R^n } \frac{u(x) - u(y)}{|x-y|^{n+2s}} \; dy,
\end{equation*}
up to a multiplicative constant only depending on $n$ and $s$,
and corresponds to the $s$ power of the Laplacian operator defined in $\mathbb R^n$
(which can be equivalently defined via the Fourier transform).

The natural boundary conditions are given in $\mathbb R^n\setminus\overline\Omega$
\begin{equation*}
	\begin{dcases}
		(-\Delta)^s_\RFL u =  f & \text{in }\Omega \\
		u = 0 & \text{in }\mathbb R^n\setminus\overline\Omega
	\end{dcases}
\end{equation*}
and also have some stochastic interpretation corresponding to killing a Lévy flight upon leaving $\Omega$.


Here we can consider all $s\in (0,1)$ and we have the precise value
\begin{equation*}
	\gamma = s.
\end{equation*}
Details can be consulted in many references, see for instance  \cite{BonSireVaz2015,Ros-Oton2014}.

\subsubsection*{Perturbations of the RFL}

Using the above one, it is possible to build other examples.
Here are a couple of interesting operators which are included in our analysis
and the corresponding references:
\begin{itemize}
\item $(-\Delta)^s_\RFL+b\cdot\nabla$ for $s\in(1/2,1)$ and $b\in L^\infty(\Omega)$: in this case (see \cite{bogdan2012gradient})
\begin{align*}
\gamma=s.
\end{align*}
\item $(-\Delta)^{s_1}_\RFL+(-\Delta)^{s_2}_\RFL$ with $0<s_2<s_1\leq 1$: in this case
(see respectively \cite{chen2010frac} and \cite{chen2012lapl})
\begin{eqnarray*}
s=\gamma=s_1 & & \text{for }s_1<1 \text{ and } n>2s_1, \\
s=\gamma=1   & & \text{for }s_1=1 \text{ and } n\geq 3.
\end{eqnarray*}
\end{itemize}

\subsubsection*{The spectral fractional Laplacian (SFL)}
A different way of considering the $s$ power of the Laplacian consists in taking the power
of the Dirichlet Laplacian, \textit{i.e.}, the Laplacian coupled with homogeneous boundary conditions.
This approach typically makes use of an eigenbasis expansion. Let ${(\varphi_m)}_{m\in\mathbb{N}}$ be the eigenfunctions of the Laplacian linked to the nondecreasing sequence of eigenvalues $0 < \lambda_1 < \lambda_2 \le ...$ (repeated according to their multiplicity)
\begin{equation*}
	\begin{dcases}
		-\Delta \varphi_m = \lambda_m \varphi_m & \text{in }\Omega, \\
		\varphi_m = 0 & \text{on }\partial\Omega.
	\end{dcases}
\end{equation*}
Let $u \in H^2 \cap H^1_0 (\Omega)$. Letting  $\widehat u_m = \int_ \Omega u \varphi_m$, we have the representation
\begin{equation*}
	-\Delta u = \sum_{m=1}^{+\infty} \lambda_m \widehat u_m \varphi_m.
\end{equation*}

The spectral fractional Laplacian is the operator with eigenvalues $\lambda_m^s$ corresponding to eigenfunctions $\varphi_m$. Hence, we define
\begin{equation*}
	(-\Delta)^s_\SFL u = \sum_{m=1}^{+\infty} \lambda_m^s \widehat u_m \varphi_m .
\end{equation*}
Since this is an operator-wise definition we provide the boundary conditions given from the classical operator, and hence the problem is
\begin{equation*}
	\begin{dcases}
		(-\Delta)^s_\SFL u = f & \text{in }\Omega, \\
		u = 0 & \text{on }\partial \Omega.
	\end{dcases}
\end{equation*}
We underline how this is not the only possible representation and it is also possible to write it as
the singular integral operator
\begin{align*}
(-\Delta)^s_\SFL u(x)=\pv\int_\Omega \big(u(x)-u(y)\big)\,J(x,y)\;dt+\kappa(x)\,u(x),
\qquad x\in\Omega,
\end{align*}
for
\begin{align*}
J(x,y)&=\frac{s}{\Gamma(1-s)}\int_0^{+\infty}{p_\Omega(t,x,y)}\;\frac{dt}{t^{1+s}},
\\
\kappa(x)&=\frac{s}{\Gamma(1-s)}\int_0^{+\infty}\Big(1-\int_\Omega p_\Omega(t,x,y)\;dy\Big)\,\frac{dt}{t^{1+s}},
\end{align*}
and $p_\Omega$ the Dirichlet heat kernel on $\Omega$. It is possible to prove that,
when $\partial\Omega\in C^{1,1}$,
\begin{align}\label{spectral estimates}
J(x,y)\asymp\frac{1}{{|x-y|}^{n+2s}}\left( \frac{\delta(x)\,\delta(y)} {|x-y|^2} \wedge 1\right) ,
\qquad
\kappa(x)\asymp\delta(x)^{-2s},
\qquad
x,y\in\Omega,
\end{align}
see \cite[Theorem 5.92]{bogdan2009potential}.

Stochastically speaking, this operator generates a subordinate killed Brownian motion,
which is a Brownian motion killed upon hitting $\partial\Omega$ and which is then evaluated at random times
distributed as an increasing $\alpha$-stable process in $(0,\infty)$, see~\cite{song2003}.
The killing of the Brownian motion as it touches the boundary is encoded in the homogeneous boundary conditions.

Here again $s\in(0,1)$ and in this case
\begin{equation*}
	\gamma = 1.
\end{equation*}
Details can be consulted in many references, see for instance~\cite{BonSireVaz2015,CabreTan}.

\subsubsection*{An interpolation of the RFL and the SFL}
A family of ``intermediate'' operators between the RFL and the SFL
has been built in~\cite{vondra}. For $\sigma_1,\sigma_2\in(0,1]$,
one can consider the spectral $\sigma_2$ power of a RFL of exponent~$\sigma_1$
\begin{align*}
\Ls_{\sigma_1,\sigma_2} =
\big( \, (-\Delta)^{\sigma_1}_\RFL \, \big)^{\sigma_2}_\SFL.
\end{align*}
It is formally clear that
\begin{align*}
\Ls_{\sigma_1,1}=(-\Delta)^{\sigma_1}_\RFL
\qquad\text{and}\qquad
\Ls_{1,\sigma_2}=(-\Delta)^{\sigma_2}_\SFL.
\end{align*}
The Green function associated to this operator satisfies~\eqref{eq:K2} with
\begin{align*}
s=\sigma_1\sigma_2
\qquad\text{and}\qquad
\gamma=\sigma_1,
\end{align*}
see~\cite[Theorem 6.4]{vondra}.

\subsubsection*{The censored fractional Laplacian (CFL)}
This operator is defined as
\begin{equation*}
	(-\Delta)^s_\CFL u (x) = \pv\int_{ \Omega } \frac{u(x) - u(y)}{|x-y|^{n+2s}} \; dy,
\end{equation*}
so that we have identity
\begin{align*}
	(-\Delta)^s_\CFL u = (-\Delta)^s_\RFL u - u \, (-\Delta)^s_\RFL \chi_\Omega
\end{align*}
(recall that the RFL is evaluated only on functions satisfying $u=0$ in $\mathbb R^n\setminus\overline\Omega$).

This operator generates a censored stable process, introduced in~\cite{bogdan03censored},
a stable process which is confined in~$\Omega$ and
finally killed upon hitting~$\partial\Omega$.
For this reason, a suitable boundary condition is
\begin{align*}
u=0\qquad\text{on }\partial\Omega.
\end{align*}

Here $s\in(1/2,1)$ and
\begin{equation*}
	\gamma = 2s-1,
\end{equation*}
see~\cite{bogdan03censored,Chen2009}.

\section{Interior Dirichlet problem: existence, uniqueness and integrability}
\label{sec:interior}

\subsection{Functional properties of the Green operator}

\begin{theorem}\label{thm:functional}
Assume~\eqref{eq:K0}--\eqref{eq:K3}.
Then $\Green$ is a continuous operator
\begin{align}
L^\infty(\Omega)\ & \longrightarrow\ L^\infty(\Omega), \label{infty-infty}\\
L^\infty_c(\Omega)\ & \longrightarrow\ \delta^\gamma L^\infty(\Omega), \label{compact infty-almost derivative}\\
L^1(\Omega)\ & \longrightarrow\ L^1(\Omega), \label{l1-l1}\\
L^1(\Omega,\delta^\gamma)\ & \longrightarrow\ L^1_{loc}(\Omega). \label{weighted l1-local l1}
\end{align}
Moreover, for $f\in\ L^1_c(\Omega)$,
	\begin{equation}
	\label{eq:L 1 to L inf outside support}
	|\Green(f) (x)| \le C\, \mathrm{dist}\big(x,\supp (f)\big)^{2s-n-\gamma} \int_ \Omega |f| \delta^\gamma,
	\qquad x \in \Omega \setminus \supp (f).
	\end{equation}
In particular, when $f\in L^1_c(\Omega)$ then $\delta^{-\gamma}\Green(f)$ is bounded in a neighbourhood of the boundary.
\end{theorem}

\begin{proof}
We are going to extensively use assumptions~\eqref{eq:K0}--\eqref{eq:K3} without further notice.

As to~\eqref{infty-infty},
we simply estimate, for any $f\in L^\infty(\Omega)$ and $x\in\Omega$,
\begin{align*}
\big|\Green(f)(x)\big|\leq\|f\|_{L^\infty(\Omega)}\int_\Omega \G(x,y)\;dy
	\leq C \|f\|_{L^\infty(\Omega)}\int_\Omega{|x-y|}^{2s-n}\;dy
	\leq C \|f\|_{L^\infty(\Omega)}.
\end{align*}

Concerning~\eqref{compact infty-almost derivative},
for $x \in \Omega \setminus \supp (f)$, we deduce
	\begin{align*}
		|\Green(f) (x)| &\le \|f\|_{L^\infty(\Omega)}\int_{\supp(f)} \G (x,y) \; dy \\
		&\le C \|f\|_{L^\infty(\Omega)} \delta(x)^\gamma \int_{\supp(f)} |x-y|^{2s-n-2\gamma} \delta(y)^\gamma \; dy
	\end{align*}
	and this proves the result.
	
For~\eqref{l1-l1}
we estimate, for $f\in L^1(\Omega)$,
\begin{align}
\int_\Omega \big| \Green(f) \big| \leq \int_\Omega \int_\Omega \G(x,y) |f(y)| \; dy \; dx
=\int_\Omega |f| \, \Green(\chi_\Omega) \leq C \|f\|_{L^1(\Omega)}
\label{lem:L 1 to L inf outside support}
\end{align}
where we have used~\eqref{infty-infty} on $\Green(\chi_\Omega)$.
	
In order to prove~\eqref{weighted l1-local l1}, we notice that,
	for any $K\Subset\Omega$ and $f\in L^1(\Omega,\delta)$, we have
	\begin{align*}
		\int_K \big| \Green(f) \big| \leq \int_K \int_\Omega \G(x,y) |f(y)| \; dy \; dx
		=\int_\Omega |f| \, \Green(\chi_K)
		\leq C \int_\Omega |f| \, \delta^\gamma
		\end{align*}
	where we have used~\eqref{compact infty-almost derivative} on $\Green(\chi_K)$.
	
Finally, we prove~\eqref{eq:L 1 to L inf outside support}.
For $x \in \Omega \setminus \supp (f)$ we have
\begin{align*}
|\Green(f) (x)| &\le \int_{\Omega} \G (x,y) |f(y)| \; dy \\
&\le C \int_{ \Omega} |f(y)|\delta(y)^\gamma \, dy \; \sup_{y \in \supp (f)} |x-y|^{2s-n-\gamma}  \\
&\le C \|f \delta^\gamma \|_{L^1(\Omega)} \, \mathrm{dist}\big(x, \supp(f)\big) ^{2s-n-\gamma}
\end{align*}
and this proves the result.
\end{proof}

\begin{remark}
In Section~\ref{sec:sharp functional spaces} we will give a sharper characterization of the image
of map $\Green$ in terms of weighted $L^1$ spaces.
\end{remark}

\begin{remark}
	\label{rem:measure data}
	Formally, one could take $\mu\in\cM(\Omega)$ and estimate
	\begin{align}
		\label{eq:measure-l1 estimate}
		\int_\Omega \big| \Green(\mu) \big| \leq \int_\Omega \int_\Omega \G(x,y) \; d|\mu|(y) \; dx
		= \int_\Omega \Green(\chi_\Omega) \; d|\mu|
		\leq C |\mu|(\Omega)
	\end{align}
	where we have used~\eqref{infty-infty} on $\Green(\chi_\Omega)$ or take $\mu\in\cM(\Omega,\delta^\gamma)$ and estimate, for any $K\Subset\Omega$
	\begin{align}
		\label{eq:weighted measure-l1loc estimate}
	\int_K \big| \Green(\mu) \big| \leq \int_K \int_\Omega \G(x,y) \; d|\mu|(y) \; dx
		= \int_\Omega \Green(\chi_K) \; d|\mu|
		\leq C \int_\Omega \delta^\gamma d|\mu|
	\end{align}
	where we have used~\eqref{compact infty-almost derivative} on $\Green(\chi_K)$.
	This computation is justified in the typical examples where $\G$ is continuous. However, since we have made no continuity assumptions for $\G$, it is possible that integration against a measure is not defined. We will give more details on this case in Section~\ref{sec:measure data}.
\end{remark}

\subsection{Weak-dual formulation}

If $\Ls$ is self-adjoint, equations of type~\eqref{eq:FDE} are typically written in \emph{very weak} form as
\begin{equation}\label{very weak}
	\int_ \Omega u \, \Ls \varphi = \int_ \Omega f \varphi
\end{equation}
for all test functions $\varphi$ in some adequate space given by the operator and the boundary conditions. Since we want to tackle multiple types of operators and boundary conditions, we focus instead on the \emph{weak-dual} formulation (see, \textit{e.g.},~\cite{bonforte+figalli+vazquez2018}).
This is formulated instead in terms of the inverse operator $\Green$, which is taken as an \emph{a priori}.
This allows to avoid giving a meaning to $\Ls \varphi$.

\begin{definition}
Given $f\in L^1(\Omega,\delta^\gamma)$, a function $u\in L^1_{loc}(\Omega)$ is a weak-dual solution of problem~\eqref{eq:FDE} if
\begin{equation}\label{weak dual}
	\int_ \Omega u \psi = \int_ \Omega f \, \Green (\psi), \qquad \text{ for any } \psi \in L^\infty_c (\Omega).
\end{equation}
\end{definition}
Note that this weak-dual formulation is equivalent to take test functions $\varphi\in\Green (L^\infty_c(\Omega))$ in~\eqref{very weak}. Also, we underline how the integral in the right-hand side of~\eqref{weak dual} is finite in view of~\eqref{compact infty-almost derivative}.

\begin{theorem}
	\label{thm:L 1 to L 1 loc estimate}
	Assume~\eqref{eq:K0}--\eqref{eq:K3} and let $f \in L^1 (\Omega,\delta^\gamma)$. Then, there exists a unique function $u \in L^1_{loc}(\Omega)$ satisfying
	\begin{equation}
		\label{eq:vwf}
		\int_ \Omega u \psi = \int_ \Omega f \, \Green(\psi) \qquad \text{for any } \psi \in L^\infty_c (\Omega).
	\end{equation}
	This function is precisely $u=\Green(f)$
	and it satisfies
	\begin{equation*}
		\int_K |u| \le \| f \delta^\gamma \|_{L^1(\Omega)} \left\|  \frac {\Green (\chi_K)}{\delta^\gamma}  \right\|_{L^\infty(\Omega)} \qquad \text{for any } K \Subset \Omega.
	\end{equation*}
\end{theorem}

\begin{proof}
	Let us first notice that $u=\Green(f)\in L^1_{loc}(\Omega)$ in view of~\eqref{weighted l1-local l1}.
	It formally satisfies~\eqref{eq:vwf} as a consequence of~\eqref{eq:K1} by the Fubini's theorem. This formal bounds are indeed rigorous for $f \in L^\infty_c (\Omega)$. Furthermore, due to the bounds provided by Theorem~\ref{thm:functional} one can pass to the limit in approximations.

We now focus on uniqueness. Let $u_1, u_2$ be two solutions to~\eqref{eq:vwf}. Then
	\begin{equation*}
		\int_\Omega(u_1-u_2)\,\psi = 0, \qquad \text{for any } \psi \in L^\infty_c(\Omega).
	\end{equation*}
	Let $K \Subset \Omega$ and $\psi = \sign(u_1-u_2) \chi_K \in L^\infty_c (\Omega)$. Using this as a test function, we deduce
	\begin{equation*}
		\int_K |u_1 - u_2| =0.
	\end{equation*}
	Since this holds for every $K \Subset \Omega$, we have that $u_1 = u_2$ a.e. in $\Omega$.
	Also, we have that
	\begin{equation*}
	\int_ K |u| \leq \int_K \big| \Green(f) \big| \le \int_\Omega |f| \, \Green(\chi_K) \le \| f \delta^\gamma \|_{L^1(\Omega)} \left\|  \frac {\Green (\chi_K)}{\delta^\gamma}  \right\|_{L^\infty(\Omega)}
	\end{equation*}
	which is a nontrivial inequality thanks to~\eqref{compact infty-almost derivative}.
	
%
\end{proof}

\subsection{Optimal class of data and a lower Hopf estimate}

\begin{theorem}[Lower Hopf]
	\label{thm:lower Hopf}
	Assume~\eqref{eq:K0}--\eqref{eq:K3}. There exists $c>0$ such that, for all $f \ge 0$,
	\begin{equation*}
		\Green (f) (x) \ge c \delta(x)^ \gamma \int_{ \Omega } f(y) \, \delta (y)^\gamma \; dy,
		\qquad x\in\Omega.
	\end{equation*}
\end{theorem}

\begin{proof}
	By assumption~\eqref{eq:K0}, it is sufficient to prove that
	\begin{equation}
		\label{eq:Hopf lower bound of kernel}
		\G(x,y) \ge c \big(  {\delta(x) \delta(y)} \big)^\gamma,
		\qquad x,y\in\Omega.
	\end{equation}
	Assume, towards a contradiction, this is not true. Then, there exist sequences of points ${(x_j)}_{j\in\N},{(y_j)}_{j\in\N} \subseteq \Omega$ such that
	\begin{equation*}
		\frac{\G(x_j,y_j)}{\delta(x_j)^\gamma \delta (y_j)^\gamma} \to 0,
		\qquad\text{as }j\uparrow\infty.
	\end{equation*}
	By assumption~\eqref{eq:K2}, either
	\begin{equation*}
		|x_j-y_j|^{2s-n-2\gamma} \to 0,
		\qquad\text{as }j\uparrow\infty,
	\end{equation*}
	which is not possible since $\Omega$ is bounded and $2s-n-2\gamma\leq 0$ (\textit{cf.}~\eqref{eq:K3}), or
	\begin{equation*}
		\frac{|x_j-y_j|^{2s-n}}{\delta(x_j)^\gamma \delta(y_j)^\gamma} \to 0,
		\qquad\text{as }j\uparrow\infty.
	\end{equation*}
	Since $\Omega$ is bounded, $\delta$ is bounded, and hence we should have that $|x_j-y_j|^{2s-n} \to 0$ as $j\uparrow\infty$ (\textit{cf.}~\eqref{eq:K3}). Again, this is not possible. We arrive to a contradiction and~\eqref{eq:Hopf lower bound of kernel} is proven.
\end{proof}

\begin{corollary}
	\label{lem:estimate G of characteristic compact support}
	Assume~\eqref{eq:K0}--\eqref{eq:K3}
	and let $K \Subset \Omega$. Then
	\begin{equation*}
		\Green (\chi_K) (x) \asymp \delta(x)^\gamma,
		\qquad x\in\Omega.
	\end{equation*}
\end{corollary}

\begin{remark}
	\label{rem:no solution outside L1 weight}
	If $ 0 \le f \notin L^1 (\Omega,\delta^\gamma)$ and $f_k=f\wedge k,\ k\in\N,$ then, for every $x \in \Omega$,
	\begin{equation*}
		\Green(f_k) (x) \ge c \delta(x)^\gamma \int_\Omega f_k(y) \delta(y)^\gamma\;dy \to + \infty
	\end{equation*}
	due to the monotone convergence theorem.
	
\end{remark}

Due to Theorem~\ref{thm:lower Hopf} and Remark~\ref{rem:no solution outside L1 weight} we have shown that $L^1(\Omega,\delta^\gamma)$ is the \textit{optimal class of data}.

\subsection{Uniform integrability over compacts}

Let us show that $\Green$ maps $L^1$-bounded sequences into $L^1$-weakly pre-compact sequences.

\begin{lemma}
	\label{lem:uniform integrability over compacts}
	Assume~\eqref{eq:K0}-\eqref{eq:K3} and let $f \in L^1 (\Omega,\delta^\gamma)$, $K\Subset \Omega$. Then, for $A \subset K$,
	\begin{equation*}
		\int_A |\Green(f)| \le C_{K,\beta} |A|^\beta \| f \delta^\gamma \|_{L^1 (\Omega)}, \qquad \textrm{ for any } 0 < \beta < \frac {2s} {n}.
	\end{equation*}
	In particular, for any $K \Subset \Omega$, $\Green$ maps bounded sequences in $L^1 (\Omega, \delta^\gamma)$ into uniformly integrable sequences in $K$.
\end{lemma}
\begin{proof}
	We have that
	\begin{align*}
		\int_{A} |\Green(f)| &\le \int_A \int_\Omega \G(x,y) f(y) \; dy \; dx
		= \int_\Omega |f(y)| \left( \int_A \G(x,y) dx \right) dy \\
		&= \int_\Omega |f(y)| \delta(y)^\gamma \left( \int_A \frac{ \G(x,y) }{\delta(y)^\gamma} \; dx \right) dy
	\end{align*}
	We take $ 1 < p < \frac{n}{n-2s}$. Due to the Hölder's inequality
	\begin{equation*}
		\int_A \frac{\G(x,y)}{\delta(y)^\gamma} \; dx \le
		|A|^{\frac 1 {p'}} \left( \int_K \left| \frac{\G(x,y)}{\delta(y)^\gamma} \right|^p dx \right)^{\frac 1 p},\
		\qquad p'=\frac{p}{p-1}>\frac{n}{2s}.
	\end{equation*}
	We estimate this last integral to recover $C_K$. For any $y\in\Omega$ such that $\mathrm{dist}(y,K) < \mathrm{dist}(K,\partial \Omega) / 2$ we have that
	\[
	\delta(y) = \mathrm{dist}(y,\partial \Omega) > \mathrm{dist}(K,\partial \Omega) - \mathrm{dist}(y,K) > \frac 1 2 \mathrm{dist}(K,\partial \Omega)
	\]
	and hence
	\begin{multline*}
		 \int_K \left| \frac{\G(x,y)}{\delta(y)^\gamma} \right|^p dx  \le \left( \frac{ \mathrm{dist}(K,\partial \Omega)}2  \right)^{-\gamma p} \int_K \G(x,y)^p dy \\
		 \le\ C \left( \frac{\mathrm{dist}(K,\partial \Omega)}2  \right)^{-\gamma p} \int_K |x-y|^{-p(n-2s)} dx
		 \le C \left( \frac{\mathrm{dist}(K,\partial \Omega)}2  \right)^{-\gamma p}
	\end{multline*}
	since $p(n-2s) < n$, where $C$ depends on $p$ but does not depend on $K$.
	One the other hand, if $y$ is such that $\mathrm{dist}(y,K) \ge \mathrm{dist}(K, \partial \Omega)/2$, for $x \in K$ we have $|x-y| > \mathrm{dist}(y,K)$ and so we compute
	\begin{multline*}
		\int_K \left| \frac{\G(x,y)}{\delta(y)^\gamma} \right|^p dx
		\le C \int_K \frac{\delta(x)^{\gamma p}}{|x-y|^{p(n-2s+2\gamma)}} dx \\
		\le\ C \left( \frac{\mathrm{dist}(K,\partial \Omega)}{2}  \right)^{p(-n+2s-2\gamma)} \int_K \delta(x)^{\gamma p} dx
		\le C\left(\frac{\mathrm{dist}(K,\partial \Omega)}{2}\right)^{p(-n+2s-2\gamma)}
	\end{multline*}
	where $C$ depends on $p$ but does not depend on $K$.
	This completes the proof.
\end{proof}

\subsection{Measure data and continuous solutions}

\label{sec:measure data}

Under mild assumptions on the Green kernel $\G$, it is possible
to improve~\eqref{infty-infty} and~\eqref{compact infty-almost derivative} to
higher regularity of solutions. By duality, this allows more general data in~\eqref{eq:FDE}
and we are particularly interested in measure data. For this reason, let us assume that
\begin{align}\label{eq:K4}\tag{K4}
  \textrm{For any sequence }\Omega \ni x_j \to x \textrm{ as } j\uparrow \infty \textrm{ we have } \lim_{j\uparrow\infty}\G(x_j, \cdot) = \G(x, \cdot) \textrm{ a.e. in } \Omega.
\end{align}
\begin{theorem}
Assume~\eqref{eq:K0}--\eqref{eq:K4}. Then the operator $\Green$ maps
\begin{align*}
	L^\infty (\Omega) &\ \longrightarrow\ C(\overline \Omega) \\
	L^\infty_c (\Omega) &\ \longrightarrow\ \delta^\gamma C(\overline \Omega).
\end{align*}
\end{theorem}
\begin{proof}
	In view of~\eqref{infty-infty} and~\eqref{compact infty-almost derivative},
	we just need to justify the continuity claim. Let us consider $f\in L^\infty(\Omega)$. To prove continuity we select an $x\in\Omega$,
	and ${(x_j)}_{j\in\N}\subset\overline\Omega\text{ such that }x_j\to x\text{ as }j\uparrow\infty$.
	By assumption~\eqref{eq:K4} we know that $\G(x_j,\cdot)\to\G(x,\cdot)$ a.e. in $\Omega$.
	Moreover, let us notice that ${\big(\G(x_j,\cdot)\big)}_{j\in\N}\subset L^p(\Omega)$, $p\in[1,n/(n-2s))$ is uniformly bounded,
	since
	\begin{align*}
	\int_\Omega {\G(x_j,y)}^p\;dy \leq \int_\Omega\frac{dy}{{|x_j-y|}^{(n-2s)p}}
	\end{align*}
	and $\Omega$ is bounded.
	Therefore, so is ${\big( \G (x_j, \cdot) f \big)}_{j\in\mathbb N}$. Due to the weak compactness in reflexive spaces it is convergent. Applying~\eqref{eq:K4} we can compute the pointwise limit
	\begin{align*}
	\Green(f)(x_j)=\int_\Omega \G(x_j,y)\,f(y)\;dy\longrightarrow\int_\Omega \G(x,y)\,f(y)\;dy = \Green(f)(x)
	\qquad\text{as }j\uparrow\infty.
	\end{align*}
	This proves that $\Green(f)\in C(\overline{\Omega})$.
	
	Let $f \in L^\infty_c (\Omega)$. Since we have already proven that $\Green(f) \in C(\Omega)$, we only need to prove that that $\delta^{-\gamma} \Green(f)$ is continuous on some small neighbourhood of $\partial \Omega$.
	Consider $\ee > 0$ small enough so that  $K=\supp(f) \subset \{\delta\ge 2\varepsilon\}\Subset\Omega$. 	
	Let $U = \{\delta<\varepsilon\}$ be the small neighbourhood of $\partial \Omega$. We have that
	\begin{align*}
		\int_K \frac{{\G(x,y)}^p}{\delta(x_j)^{\gamma p}}\;dy \leq \int_K \frac{{\delta(y)}^{\gamma p}}{{|x-y|}^{(n-2s+2\gamma)p}} dy \le C_\ee, \qquad \forall x \in U.
	\end{align*}
	Select now an $x \in U$ and let $\Omega \in x_j \to x$ as $j \uparrow +\infty$. Since $U$ is open, then $x_j \in U$ for $j$ large enough.
	Again, by weak compactness,
	\begin{multline*}
	\delta(x_j)^{-\gamma}\Green(f)(x_j)=\int_K \frac{\G(x_j,y)}{{\delta(x_j)}^{\gamma}}\,f(y)\;dy \\
	\longrightarrow\int_K \frac{\G(x,y)}{{\delta(x)}^{\gamma}}\,f(y)\;dy = \delta(x)^{-\gamma}\Green(f)(x)
	\qquad\text{as }j\uparrow\infty.
	\end{multline*}
	This completes the proof.
\end{proof}

With this new machinery, we can justify the intuition given by Remark~\ref{rem:measure data}.
\begin{theorem}
	Assume~\eqref{eq:K0}--\eqref{eq:K4}. Then, $\Green$ maps
	\begin{align}
			\cM(\Omega)\ & \longrightarrow\ L^1(\Omega) \label{measure-l1}\\
		\cM(\Omega,\delta^\gamma)\ & \longrightarrow\ L^1_{loc}(\Omega).  \label{weighted measure-local l1}
	\end{align}
	Furthermore, for every $\mu \in \WM$, $u = \Green(\mu)$
	is the unique $u \in L^1_{loc} (\Omega)$
	such that
	\begin{equation}
			\label{eq:weak-dual formulation measures}
			\int_\Omega u \psi = \int_\Omega \Green(\psi) \; d \mu \qquad \text{for any } \psi \in L^\infty_c (\Omega).
	\end{equation}
	holds. Moreover, it satisfies
	\begin{equation*}
		\int_K |u| \le \left( \int_\Omega \delta^\gamma d|\mu| \right) \left\|  \frac {\Green (\chi_K)}{\delta^\gamma}  \right\|_{L^\infty(\Omega)} \qquad \text{for any } K \Subset \Omega.
	\end{equation*}
\end{theorem}
\begin{proof}
	Due to~\eqref{eq:K4}, $\Green (\mu)$ is now a well-defined integral. Now we can apply~\eqref{eq:measure-l1 estimate} and~\eqref{eq:weighted measure-l1loc estimate}.
	Moreover, also $\int_\Omega \Green(\mu)\,\psi$ is well-defined for any $\psi\in L^\infty_c(\Omega)$. Notice that, in view of~\eqref{compact infty-almost derivative}, we have
	\begin{align*}
	\int_\Omega \int_\Omega \G(x,y) \, |\psi(x)| \; dx \; d|\mu|(y) \leq C_\psi\int_\Omega \delta(y)^\gamma \; d|\mu|(y)<+\infty
	\end{align*}
	so that we can apply the Fubini's theorem and~\eqref{eq:K1} to deduce
	\begin{align*}
	\int_\Omega \Green(\mu)\,\psi=\int_\Omega\Green(\psi)\;d\mu
	\end{align*}
	which proves~\eqref{eq:weak-dual formulation measures}.
	We now show uniqueness. Let $u_1, u_2$ be two solutions to~\eqref{eq:weak-dual formulation measures}. Then
	\begin{equation*}
		\int_\Omega(u_1-u_2)\,\psi = 0, \qquad \text{for any } \psi \in L^\infty_c(\Omega).
	\end{equation*}
	Let $K \Subset \Omega$ and $\psi = \sign(u_1-u_2) \chi_K \in L^\infty_c (\Omega)$. Using this as a test function, we deduce
	\begin{equation*}
		\int_K |u_1 - u_2| =0.
	\end{equation*}
	Since this holds for every $K \Subset \Omega$, we have that $u_1 = u_2$ a.e. in $\Omega$.
	Also, we have that
	\begin{equation*}
		\int_ K |u| \leq \int_K \big| \Green(\mu) \big| \le \int_\Omega \Green(\chi_K) \; d|\mu|
		\le \left( \int_\Omega \delta^\gamma d|\mu| \right) \left\|  \frac {\Green (\chi_K)}{\delta^\gamma}  \right\|_{L^\infty(\Omega)}
	\end{equation*}
	which is a nontrivial inequality thanks to~\eqref{compact infty-almost derivative}.
\end{proof}

\section{Breakdown of the boundary condition in the interior problem}
\label{sec:breakdown boundary condition}
We address now the main questions of this paper,
which is the violation of the boundary data in the optimal theory for the interior problem.
We give precise answers of the anomalous boundary behaviour
in terms of the behaviour of the forcing data $f$.
\begin{figure}[p]
	\centering
    \begin{minipage}[t][12cm][b]{.32\textwidth}
                    \centering
                    \includegraphics[width=\textwidth]{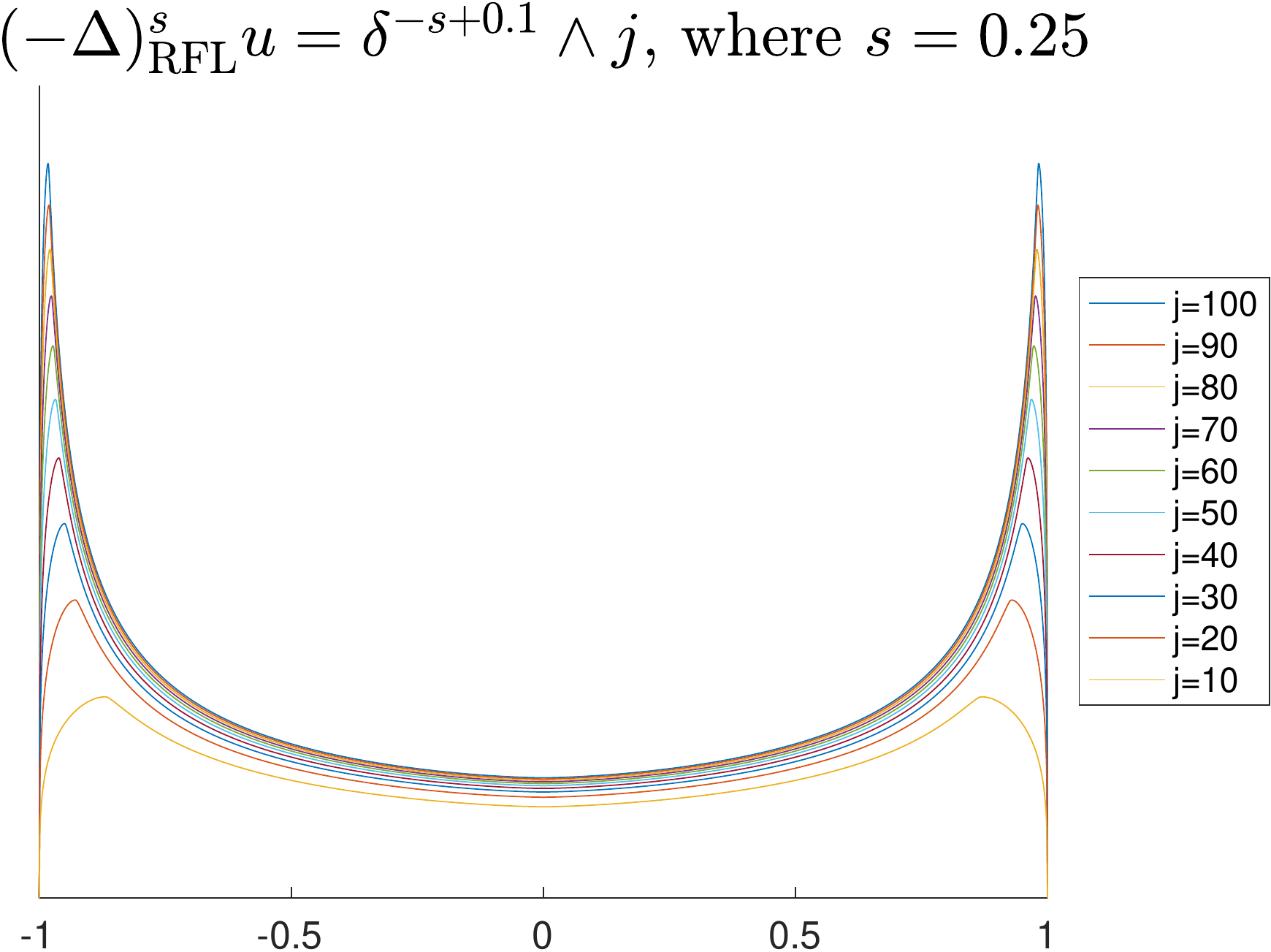}%
                    
                    \includegraphics[width=\textwidth]{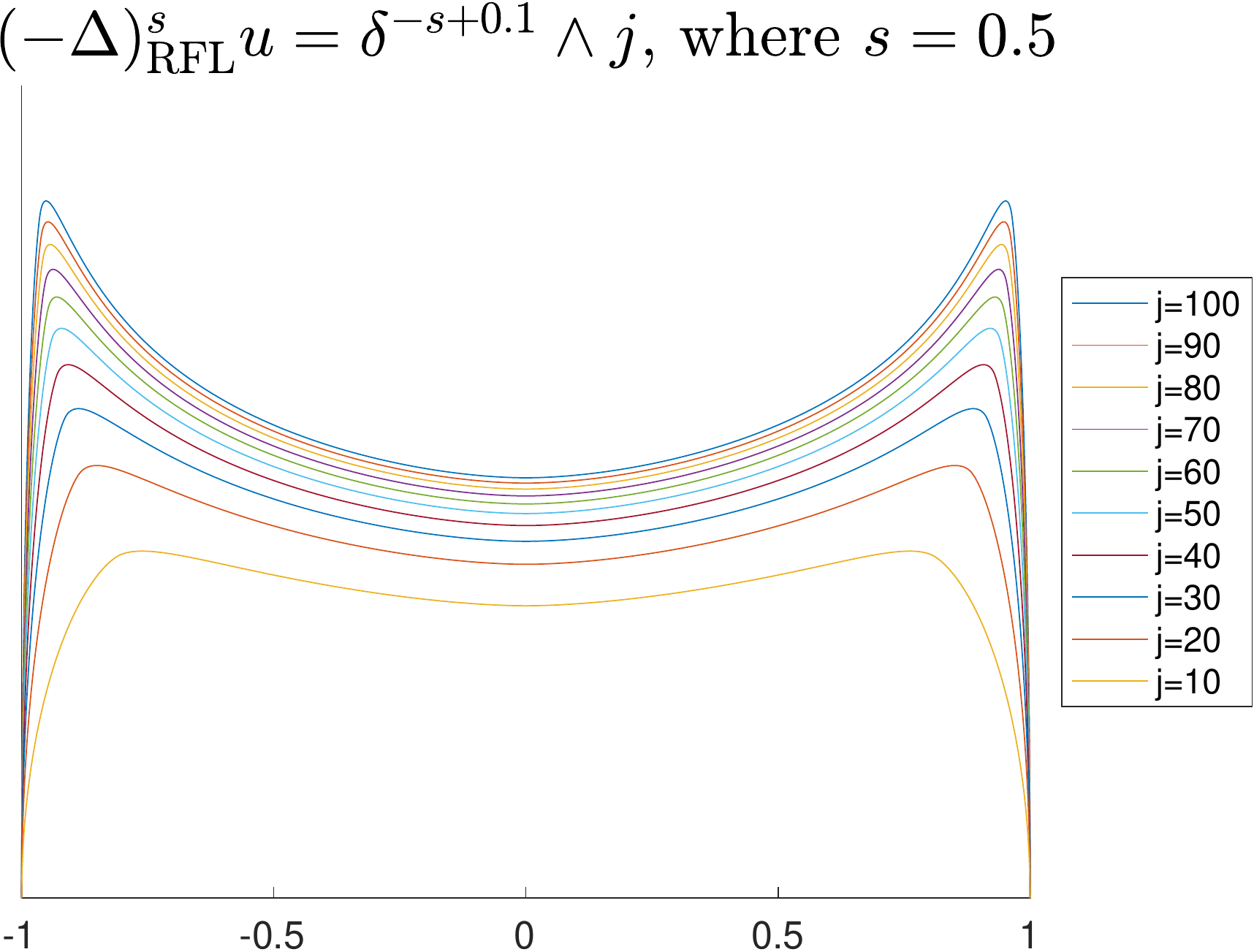}
                    
                    \includegraphics[width=\textwidth]{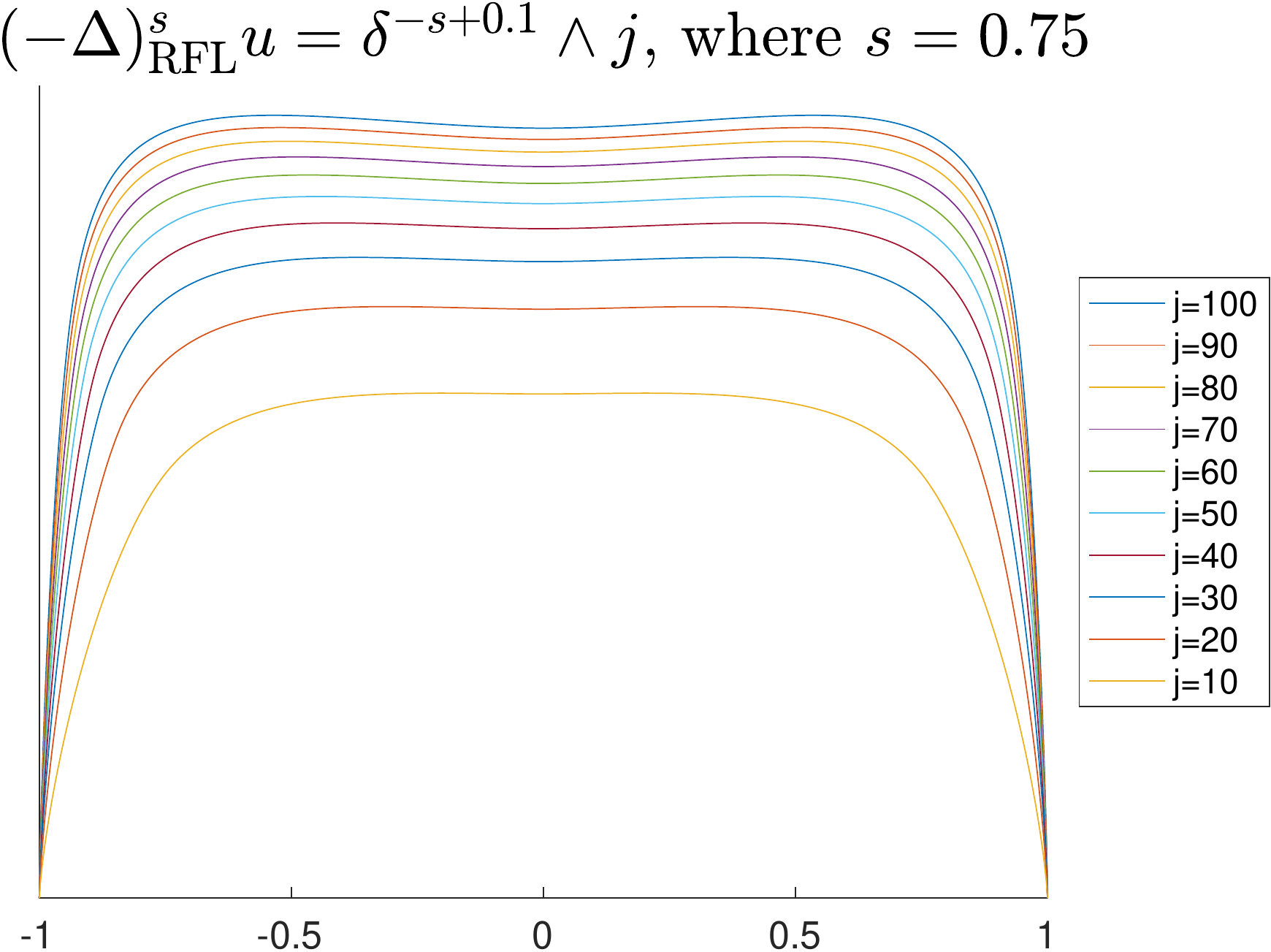}
                    
                    \includegraphics[width=\textwidth]{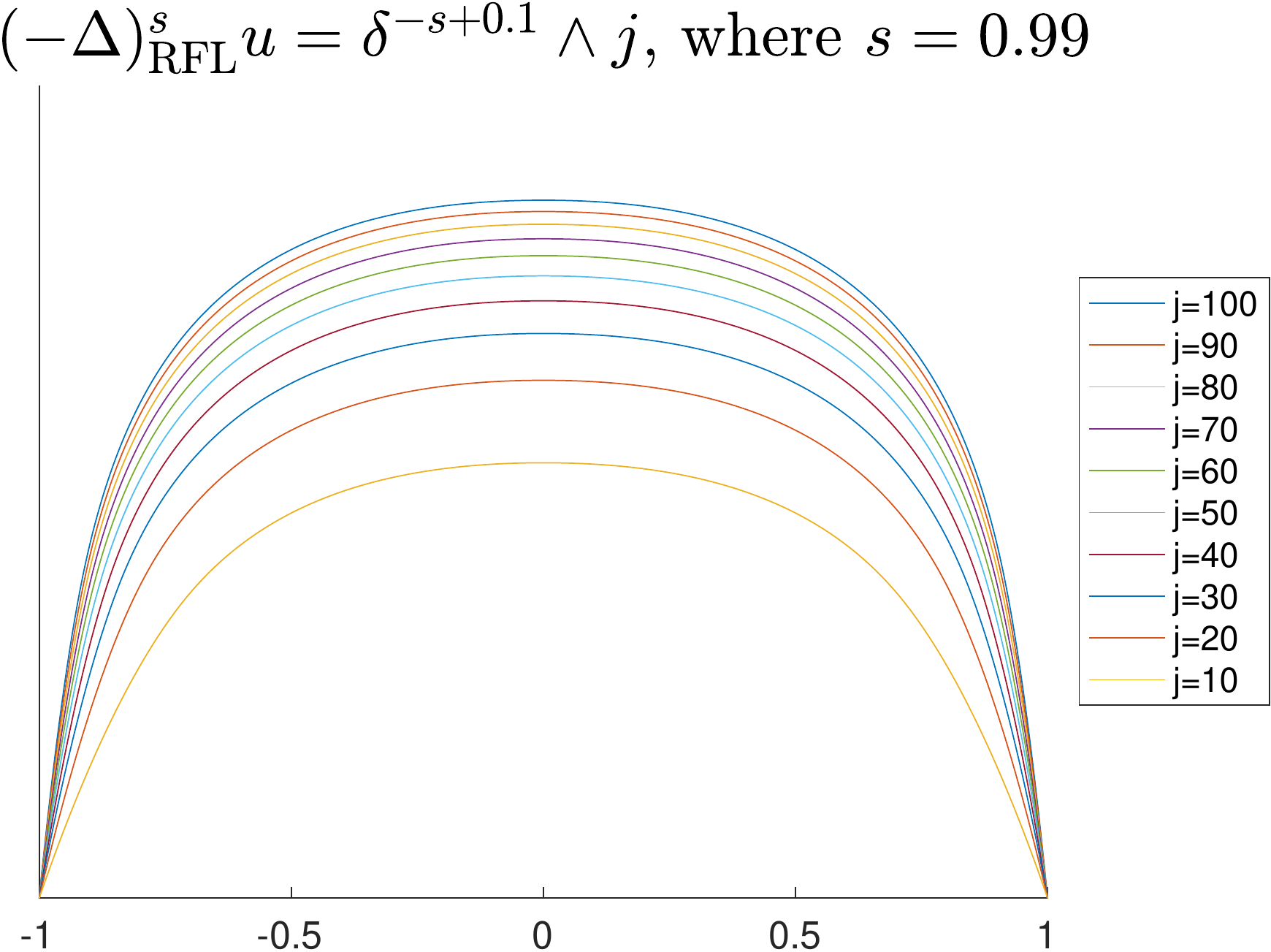}
                    
                    \subcaption{RFL}
                    
                    \label{fig:1b}
    \end{minipage}
    \begin{minipage}[t][12cm][b]{.32\textwidth}
                    \centering
                    \includegraphics[width=\textwidth]{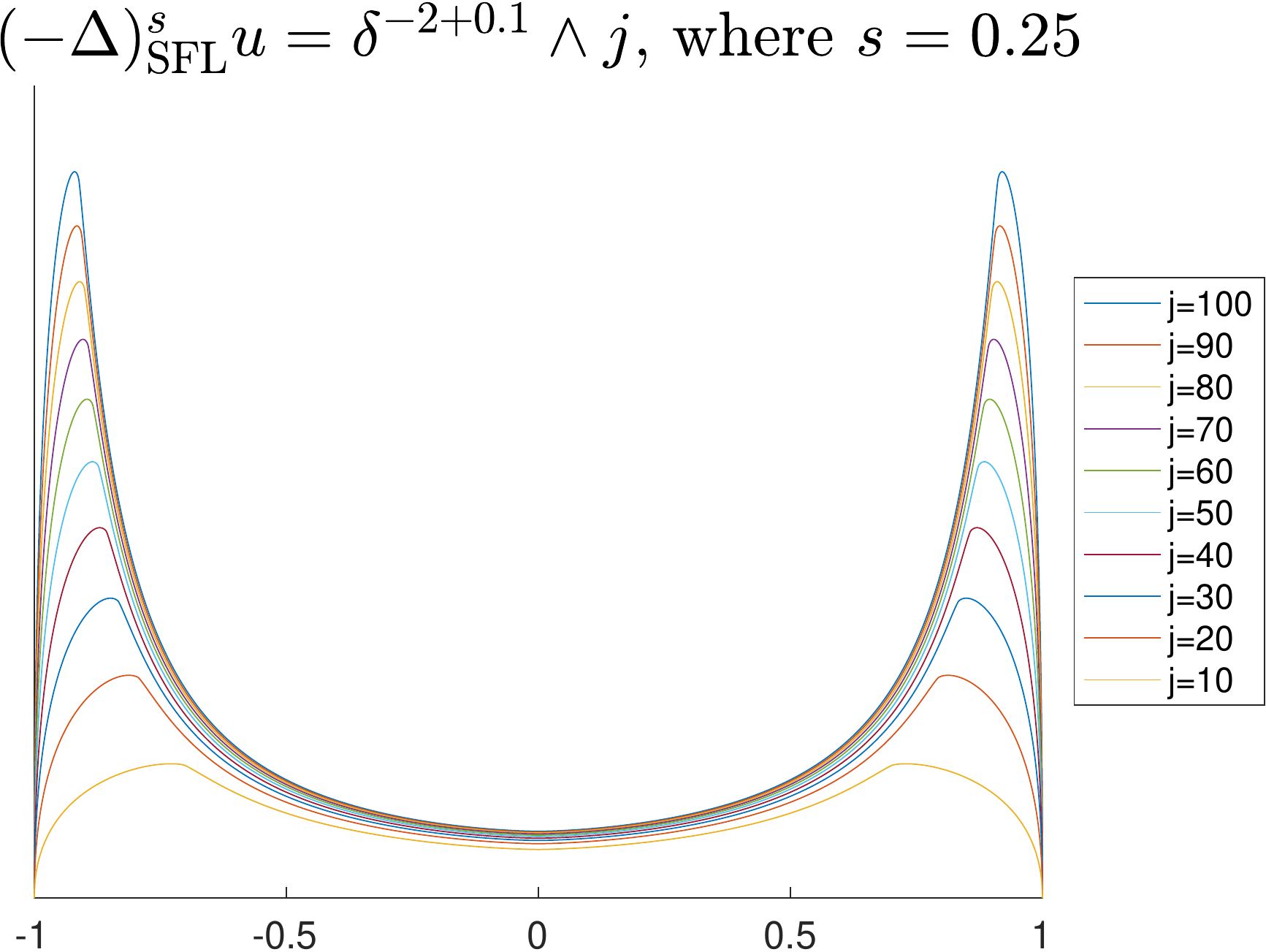}%
                    
                    \includegraphics[width=\textwidth]{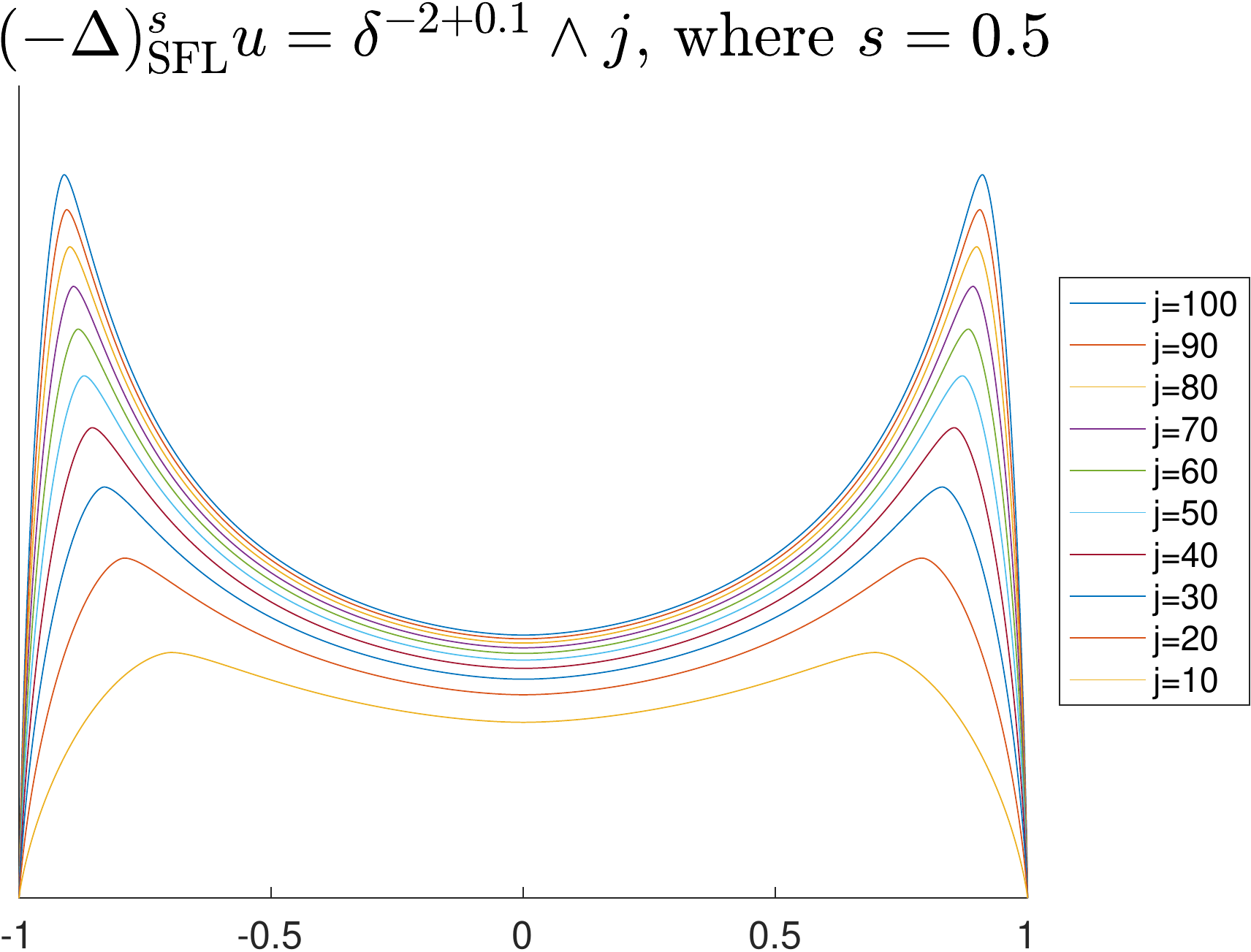}
                    
                    \includegraphics[width=\textwidth]{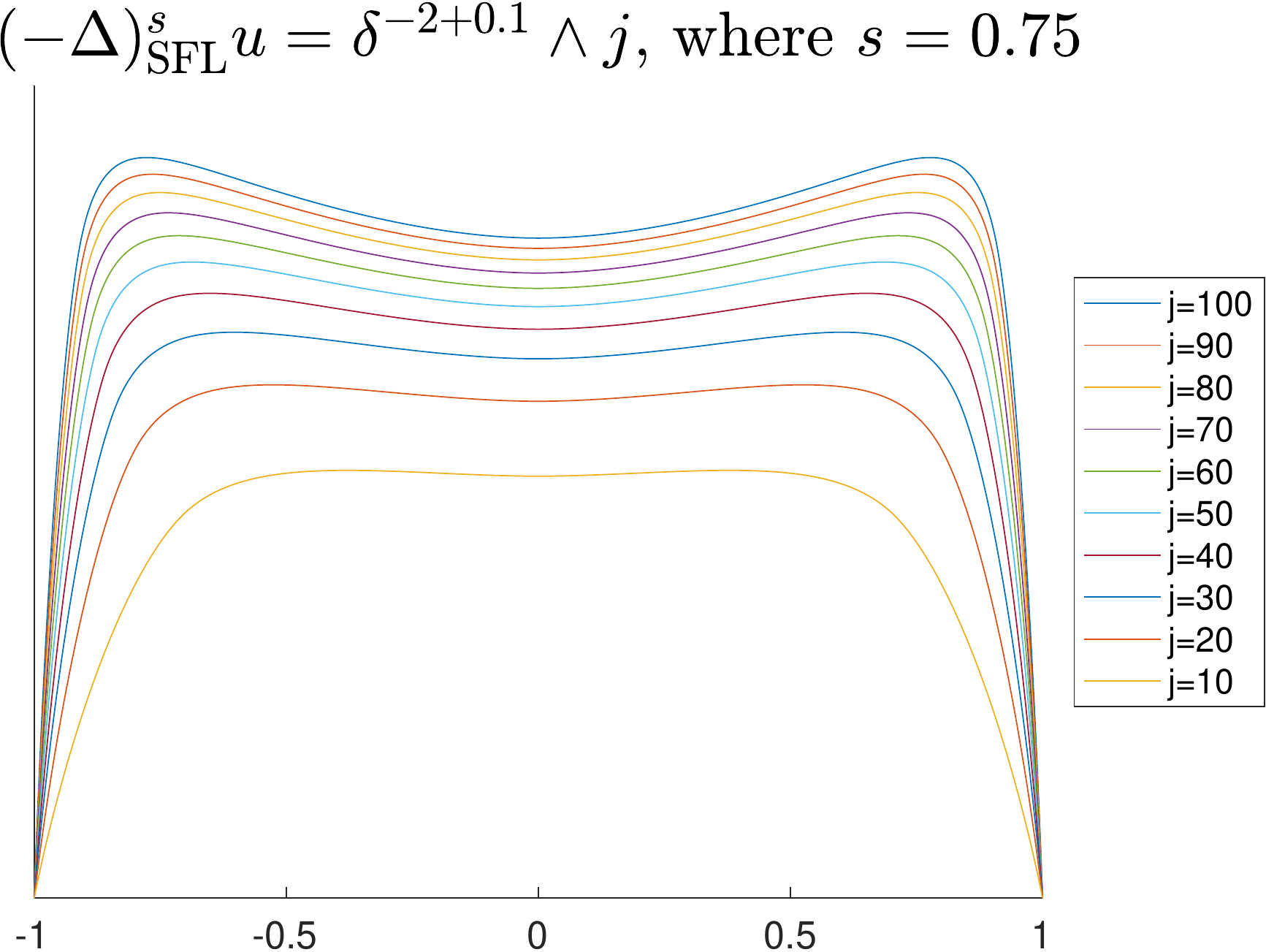}
                    
                    \includegraphics[width=\textwidth]{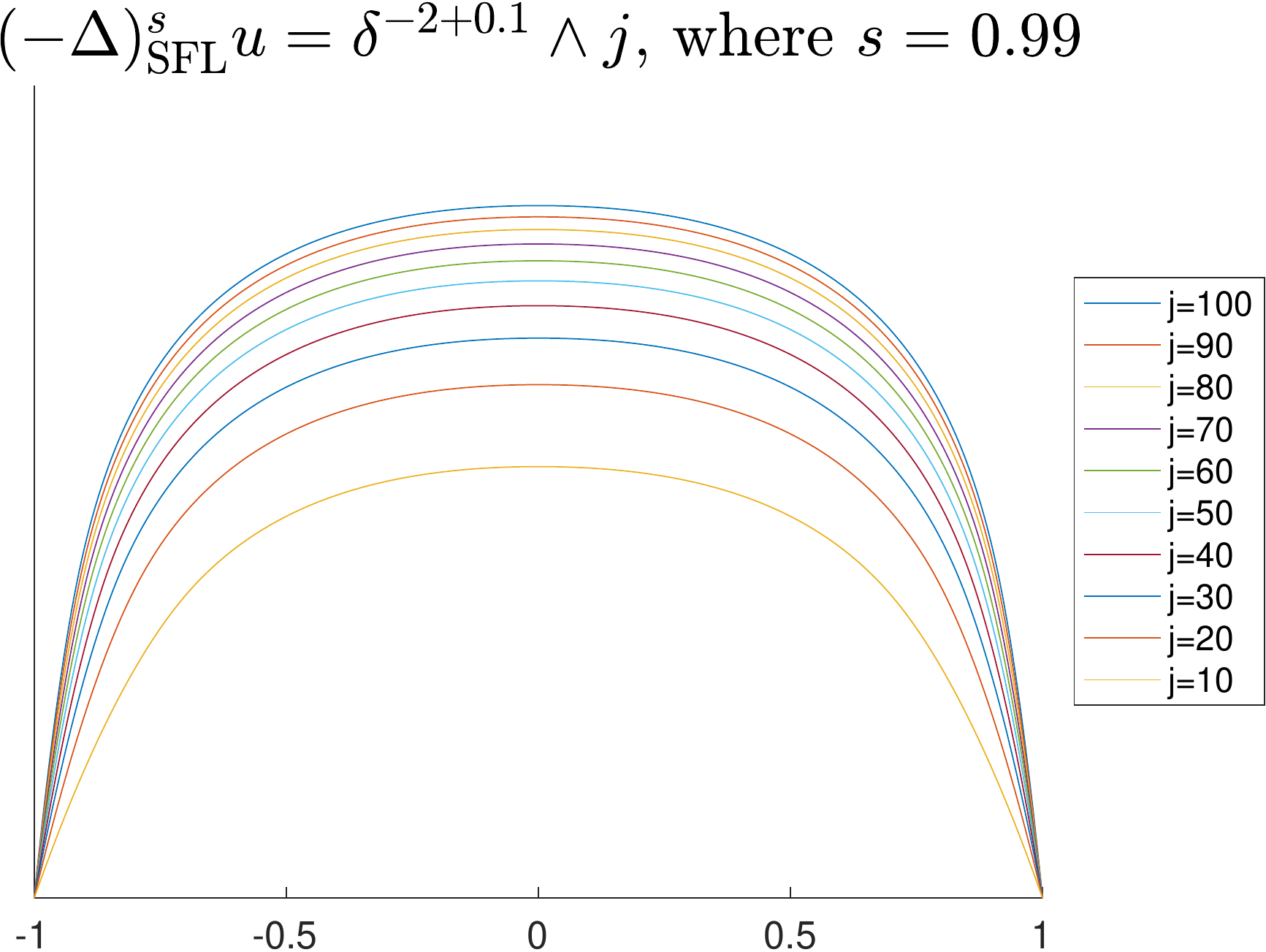}
                    
                    \subcaption{SFL}
                    
                    \label{fig:1b}
    \end{minipage}
    \begin{minipage}[t][12cm][b]{.32\textwidth}
                    \centering
                     
                    \includegraphics[width=\textwidth]{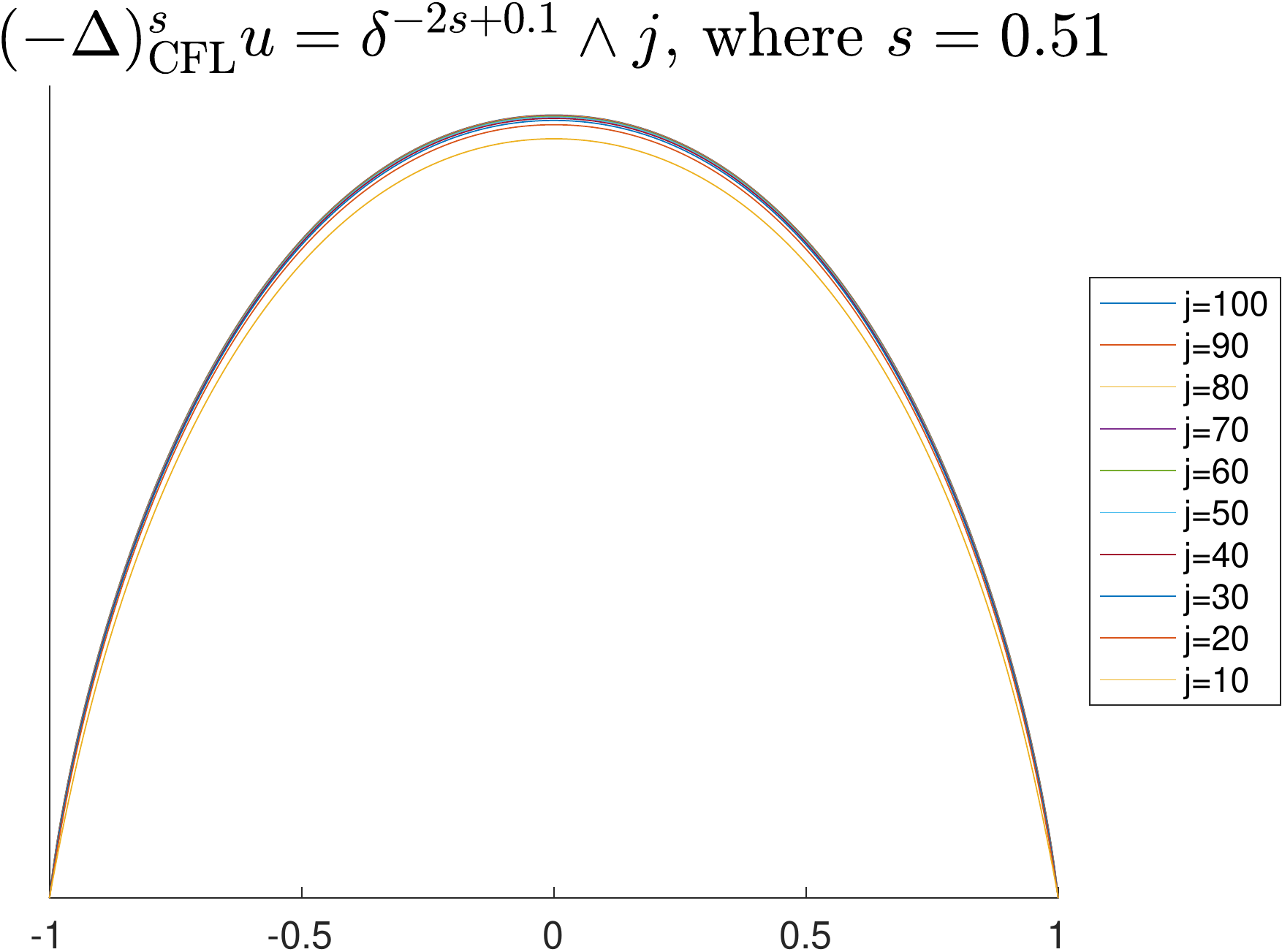}     
                                        
                    \includegraphics[width=\textwidth]{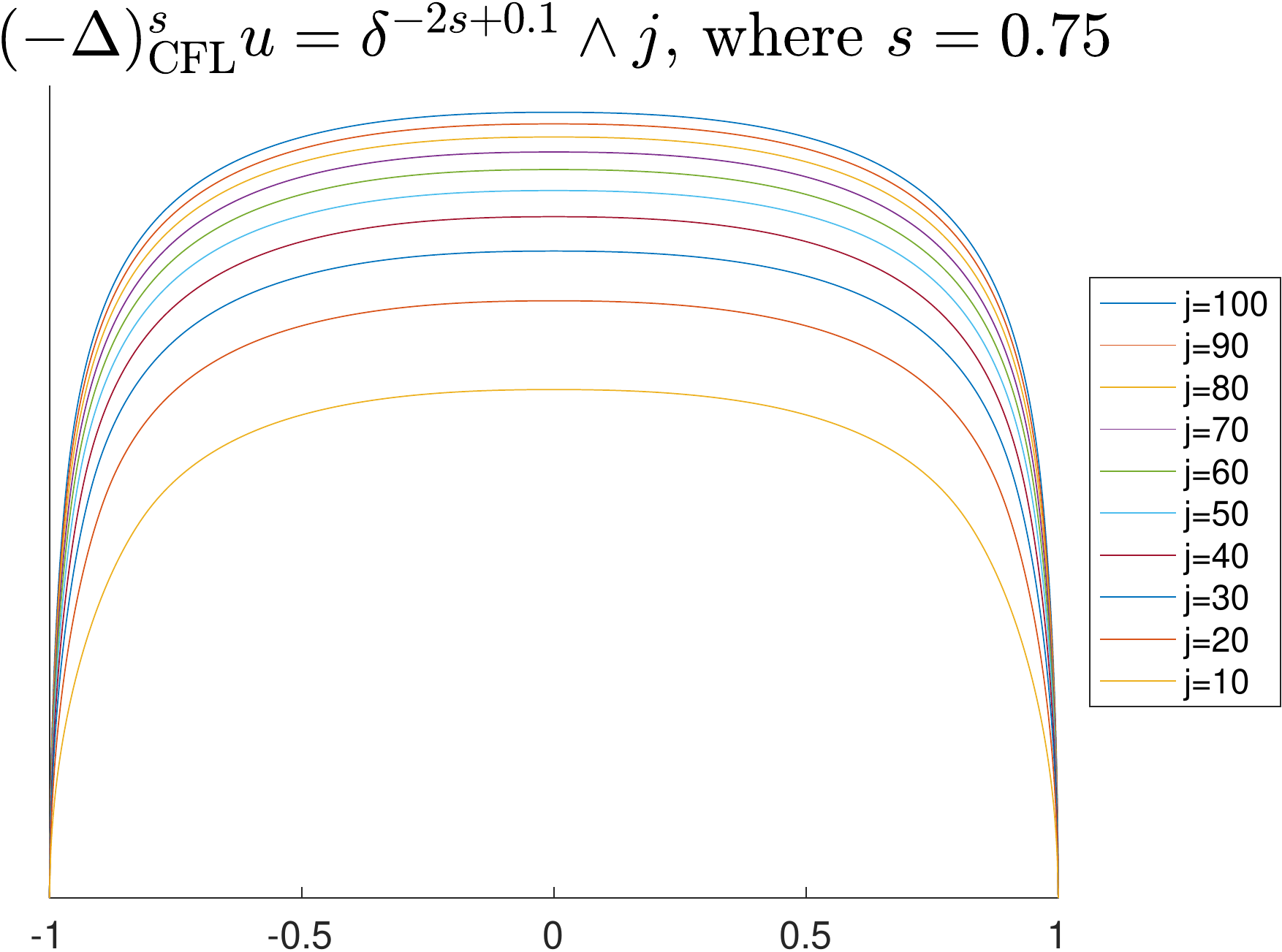}
                    
                    \includegraphics[width=\textwidth]{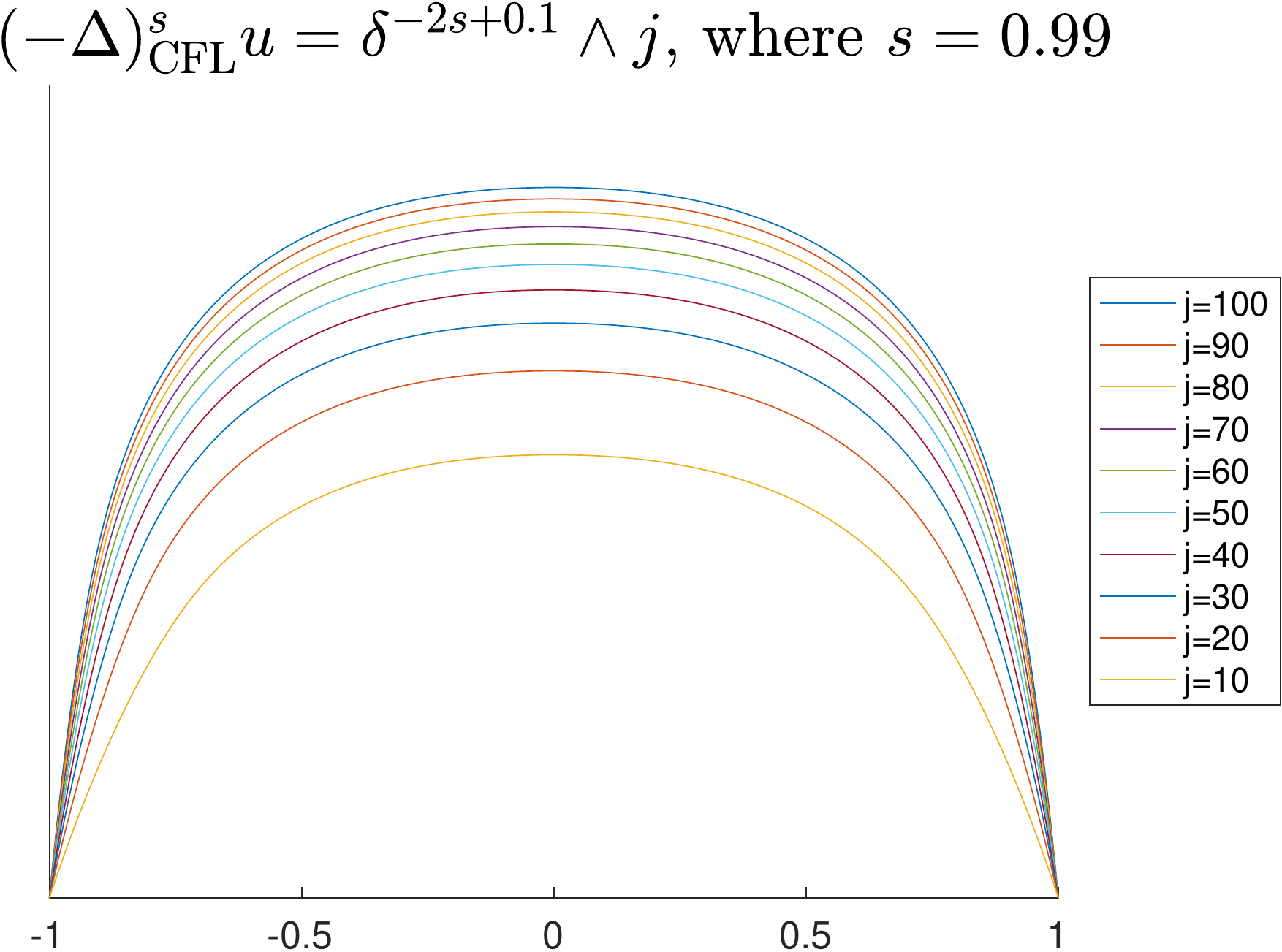}
                    
                    \subcaption{CFL}
                    
    \end{minipage}
	
    \caption{
    Numerical solutions of $ \Ls u = \delta^{-\gamma -1 + \ee } \wedge j $ in dimension $n=1$ for different operators. All computations correspond to Finite Difference numerical schemes. For the RFL we take the weights for the finite difference discretization of the fractional Laplacian in $\mathbb R^n$ in \cite{delTeso2018} (see also \cite{Ciaurri}). The discretization for smooth functions is rigorously shown to be $O(h^2)$. A previous approach by Finite Differences is given in \cite{Huang}. Experimental results in \cite{Huang} suggest that the restriction for RFL is of order $O(h^s)$. For the SFL we use the fractional power of matrix of finite differences. It is known that the eigenvalues of this matrix converge to those of the usual Laplacian, and hence its fractional power produces a convergent scheme for the SFL. A different scheme can be found in \cite{Cusimano+delTeso2018}. For the CFL we have used a novel approach, which we will describe in an upcoming paper.}
    \label{fig:subcritical}
\end{figure}

\subsection{Range of exponents}

Before stating and proving the main result of this paragraph,
we need to state a couple of technical estimates on which the
result is based. Since the proofs of these estimates is rather long and technical,
we defer them to Appendix~\ref{appendix:technicalities}.
The first one gives some interior estimates;
the second one is describing the sharp behaviour of solutions at the boundary.

\begin{remark}\label{rmk:tubular}
In what follows, we use without further notice~$\ee>0$ to denote the fixed width on which
the tubular neighbourhood theorem can be rightfully applied, \textit{i.e.}, the map
\begin{eqnarray*}
		\Phi : \partial \Omega \times (-\ee , \ee)& \longrightarrow &\mathbb R^n  \\
		(z,\delta) &\longmapsto & z + \delta \mathbf n (z)
\end{eqnarray*}
defines a diffeomorphism to its image. Here, $\mathbf n$  represents the unit interior normal.
This is well known for smooth manifolds (see~\cite{daSilva2008}), and holds also for $C^{1,1}$ open sets of $\mathbb R^n$.
The notation $\delta$ might seem like an abuse of notation, but it will lead to no confusion since, in this setting, $\mathrm{dist}(\Phi (z,\delta), \partial \Omega) = |\delta|$ for $\varepsilon$ sufficiently small.
\end{remark}

\begin{lemma}\label{lem:sharp interior}
Assume that~\eqref{eq:K0}--\eqref{eq:K3} hold.
Moreover, assume $\beta+\gamma>-1$ and let $\eta<\ee$ be fixed.
Then there exists a constant $\overline{c}(\eta)>0$ such that,
for any $x\in\{\delta>\eta/2\}$, it holds
\begin{align}\label{eq: sharp interior 1}
\Green(\delta^\beta\chi_{\{\delta<\eta\}})(x) \leq \overline{c}(\eta).
\end{align}
Moreover, if $\gamma<s-\frac12$ and then $x\in\{\delta>\eta/2\}$,
\begin{align}\label{eq: sharp interior 2}
\Green(\delta^\beta\chi_{\{\delta<\eta\}})(x) \leq \eta^{\beta+\gamma+1}\delta(x)^\gamma
\end{align}
up to constants not depending on $\eta$.
\end{lemma}

\begin{lemma}\label{lem:sharp boundary}
Assume that~\eqref{eq:K0}--\eqref{eq:K3} hold.
Moreover, assume $\beta+\gamma>-1$ and let $\eta<\ee$ be fixed.
Then for any $x\in\{\delta<\eta/2\}$ it holds
\begin{align}\label{first boundary estimates with theta}
\Green(\delta^\beta\chi_{\{\delta<\eta\}})(x) \asymp
\delta(x)^{\beta+2s}+\eta^{\beta+\gamma+1}\delta(x)^\gamma+\Theta(\eta,x),
\end{align}
where the $\Theta$ is defined as follows:
\begin{enumerate}[\rm a)]
\item If $\gamma<s-\frac12$ then
\begin{align}\label{theta gamma small}
\Theta(\eta,x):=
\delta(x)^{\beta+2\gamma+1};
\end{align}
\item If $\gamma=s-\frac12$ then
\begin{align}\label{theta gamma equal}
\Theta(\eta,x):=
\delta(x)^{\beta+2s}\,|\!\ln\delta(x)|
+
\eta^{\beta+\gamma+1}\,|\!\ln\eta|\,\delta(x)^\gamma;
\end{align}
\item If $\gamma>s-\frac12$ then
\begin{align}\label{theta gamma large}
\Theta(\eta,x):=
\left\lbrace\begin{aligned}
& 0  									& & \text{if } \beta<\gamma-2s, \\
& \delta(x)^{\beta+2s}\,|\!\ln(\delta(x)/\eta)|  	& & \text{if } \beta=\gamma-2s, \\
& \eta^{\beta+2s-\gamma}\delta(x)^\gamma 	& & \text{if } \beta>\gamma-2s.
\end{aligned}\right.
\end{align}
\end{enumerate}
\end{lemma}

We are now ready to prove the following estimate.

\begin{theorem}\label{thm:range of exponents}
Assume that~\eqref{eq:K0}--\eqref{eq:K3} hold.
Moreover, assume $\beta + \gamma > -1$.
Then $\delta^{\beta} \in L^1 (\Omega, \delta^\gamma)$
and
\begin{equation*}
	\Green(\delta^\beta) \asymp \delta^{ \alpha }
\end{equation*}
with
\begin{table}[H]
\centering
\begin{tabular}{c|c|c|c}
$\alpha$ & $\gamma<s-\frac12$ & $\gamma=s-\frac12$ & $\gamma>s-\frac12$ \\ [.25em]
\hline & & & \\ [-.75em]
$\beta<\gamma-2s$ 	& $\bullet$ 		& $\bullet$	& $\beta+2s$ \\
 & & & \\ [-.75em]
$\beta=\gamma-2s$	& $\bullet$		& $\bullet$ 	& $\gamma$ {\footnotesize (and log. weight)} \\
 & & & \\ [-.75em]
$\beta>\gamma-2s$	& $\gamma$		& $\gamma$	& $\gamma$
\end{tabular}
	\caption{Range of exponents. We indicate by $\bullet$ the cases outside the admissible range $\beta + \gamma > -1$.}\label{table:exponents}
\end{table}
\noindent where by logarithmic weight we mean that $\Green(\delta^{\gamma - 2s}) \asymp \delta^\gamma \, \big(1+|\!\ln \delta |).$
\end{theorem}
\begin{proof}
Let us first notice that conditions $\beta\leq\gamma-2s$ and $\gamma\leq s-\frac12$
are not compatible: indeed, if they both held, then it would be
$\beta+\gamma\leq 2\gamma-2s\leq -1$ contradicting our standing assumption on~$\beta$.

	We pick some fixed $\eta < \ee$ and we write
	\begin{equation*}
		\Green( \delta^\beta) = \Green (\delta^\beta \chi_{ \{\delta \le \eta \} }) + \Green (\delta^\beta \chi_{ \{ \delta > \eta \} }   ).
	\end{equation*}
	Notice that
	\begin{align*}
		\Green (\delta^\beta \chi_{ \{ \delta > \eta \} }   )\asymp\delta^\gamma
	\end{align*}
	as $\delta^\beta \chi_{ \{ \delta > \eta \} }\in L^\infty_c(\Omega)$.
	For the other term, we exploit Lemmas~\ref{lem:sharp boundary}
	and~\eqref{eq: sharp interior 1} to say
	\begin{align*}
		\Green (\delta^\beta \chi_{ \{\delta \le \eta \} })\asymp
		\delta(x)^{\beta+2s}+\delta(x)^\gamma+\Theta(1,x)
	\end{align*}
	where $\Theta$ is defined as in Lemma~\ref{lem:sharp boundary}.
	Now, the asymptotic behaviour is driven by the least exponent on $\delta$,
	yielding the situation depicted in Table~\ref{table:exponents}.
\end{proof}

\begin{remark}
	Let us look at the ranges for $\alpha$ and $\beta$ as in Theorem~\ref{thm:range of exponents}, disregarding the logarithmic cases, to better understand the possible boundary behaviours of solutions to~\eqref{eq:FDE}.
	When $\gamma > s - \frac 1 2$ the admissible range for $\beta$ is $(-1-\gamma, +\infty)$; in this case $\alpha$ runs in $(2s-\gamma-1, \gamma]$: notice that $2s-\gamma-1$ might be negative, meaning that also $\alpha$ is allowed to be negative in some cases. This translates in particular into a rebuttal of $\Green(\delta^\beta)=0$ on $\partial\Omega$, despite the fact that this would be the solution to a  homogeneous boundary (or exterior problem) value problem. For exterior problem, this shows solutions are discontinuous on the boundary for some singular data (possibly outside $L^1(\Omega)$).
For boundary value problems this is a breakdown of the boundary condition.
However, this behaviour intrinsic to the problem, since
we are only constructing the closure of the solution operator $\Green$, to its maximal domain of definition.
	
	If instead $\gamma < s - \frac 1 2$ then again $\beta$ ranges in $(-\gamma-1,+\infty)$, but this time $\alpha$ is bound to be equal to $\gamma$, meaning that there is no range for $\alpha$.
\end{remark}

\begin{example}
Let us exemplify the statement of Theorem~\ref{thm:range of exponents}.
If we consider $\beta=0$, we deduce
\begin{align*}
\Green(\chi_\Omega)\ \asymp\ \left\lbrace\begin{aligned}
& \delta^\gamma & & \text{if }\gamma<2s \\
& \delta^{\gamma}\,\big(1+|\!\ln\delta|\big) & & \text{if }\gamma=2s \\
& \delta^{2s} & & \text{if }\gamma>2s
\end{aligned}\right.
\qquad\text{in }\Omega.
\end{align*}
Setting $\beta=\gamma$ gives
\begin{align*}
\Green(\delta^\gamma)\asymp\delta^\gamma,
\qquad \text{in }\Omega.
\end{align*}
Taking $\beta=\pm s$ returns respectively
\begin{align*}
\Green(\delta^s)\ \asymp\ \left\lbrace\begin{aligned}
& \delta^\gamma & & \text{if }\gamma<3s \\
& \delta^{\gamma}\,\big(1+|\!\ln\delta|\big) & & \text{if }\gamma=3s \\
& \delta^{3s} & & \text{if }\gamma>3s
\end{aligned}\right.
\quad\text{and}\quad
\Green(\delta^{-s})\ \asymp\ \left\lbrace\begin{aligned}
& \delta^\gamma & & \text{if }\gamma<s \\
& \delta^{\gamma}\,\big(1+|\!\ln\delta|\big) & & \text{if }\gamma=s \\
& \delta^{s} & & \text{if }\gamma>s.
\end{aligned}\right.
\end{align*}

The value $\beta=-2s$ is a somewhat critical value for the boundary behaviour
(if $\gamma>2s-1$, otherwise  $\Green$ is not defined),
since
\begin{align*}
\Green(\delta^{-2s})\ \asymp\ 1.
\end{align*}
Below this value, if $\beta$ is of the form $\beta=-2s-\varepsilon,\ \varepsilon\in(0,\gamma-2s+1) \setminus\{-\gamma\}$, one has
\begin{align*}
\Green(\delta^{-2s-\varepsilon})\ \asymp\ \delta^{-\varepsilon}.
\end{align*}
\end{example}

\normalcolor

\begin{remark}
Notice that, if $\beta \in (-1/2,-2s)$ we have that $\delta^\beta\in L^2(\Omega)$ and $\Green(\delta^\beta) \notin L^\infty (\Omega)$. This is possible if $s\in(0,1/4)$. Hence, this breakdown of the boundary conditions happens \emph{inside} the variational (energy) theory. This should not be surprising since, for $s < 1/2$, $H_0^s = H^s$ (the space has no trace). This points to an essential difference between the properties of the classical Laplacian and the fractional Laplacian for small values of $s$.
\end{remark}

\subsection{Subcritical boundary behaviour in average terms}
We have an extension of the result for the classical Laplacian on averaged convergence to the boundary, see~\cite{Ponce2016a}.
\begin{lemma}\label{lem:subritical traces}
Assume that~\eqref{eq:K0}--\eqref{eq:K3} hold
and let $f \in L^1 (\Omega, \delta^\gamma)$.
\begin{enumerate}[\rm a)]
	\item If $\gamma > s - 1/2$,
	\begin{equation*}
		\frac 1\eta\int_{ \{  \delta < \eta  \} } \frac{|\Green(f)|}{\delta^{2s-\gamma-1}} \longrightarrow 0 \qquad\text{as }\eta\downarrow 0.
	\end{equation*}
	\item If $\gamma = s - 1/2$,	
	\begin{equation*}
		\frac 1{\eta\,|\!\ln \eta|} \int_{ \{  \delta < \eta  \} } \frac{|\Green(f)|}{\delta^\gamma} \longrightarrow 0 \qquad\text{as }\eta\downarrow 0.
	\end{equation*}
	\item If $\gamma < s - 1/2$,
	\begin{equation}\label{subgamma}
		\frac 1\eta\int_{ \{  \delta < \eta  \} } \frac{|\Green(f)|}{\delta^\gamma} \le C.
	\end{equation}
\end{enumerate}
\end{lemma}

\begin{proof}
	Assume that $f \ge 0$. Let us start from a).  It is clear that, by duality,
	\begin{equation*}
			\eta^{-1} \int_{ \{  \delta < \eta  \} } \frac{\Green(f)}{\delta^{2s-\gamma-1}} =\int_ \Omega f \delta^\gamma \; \frac{\Green (  \delta^{-2s+\gamma+1} \chi_{ \{ \delta < \eta \} }  )}{\eta \delta^\gamma}.
	\end{equation*}
	We decompose this last integral into two
	\begin{multline*}
	\int_ \Omega f \delta^\gamma \frac{\Green (  \delta^{-2s+\gamma+1} \chi_{ \{ \delta < \eta \} }  )}{\eta \delta^\gamma} = \int_ { \{ \delta \le \eta/2 \}  } f \delta^\gamma \frac{\Green (  \delta^{-2s+\gamma+1} \chi_{ \{ \delta < \eta \} }  )}{\eta \delta^\gamma}\ +\\
	+\ \int_ { \{ \delta > \eta/2 \} } f \delta^\gamma \; \frac{\Green (  \delta^{-2s+\gamma+1} \chi_{ \{ \delta < \eta \} }  )}{\eta \delta^\gamma}.
	\end{multline*}
Using~\eqref{first boundary estimates with theta} and~\eqref{theta gamma large},
in $\{\delta<\eta/2\}$ we get
\begin{align*}
\frac{\Green(\delta^{{-2s+\gamma+1}}\chi_{\{\delta<\eta\}})}{\eta\delta^\gamma}
\asymp\frac{\delta}\eta+\eta^{-2s+2\gamma+1}+1\leq 3
\end{align*}
and therefore
\begin{align*}
\int_ { \{ \delta \le \eta/2 \}  } f \delta^\gamma \frac{\Green (  \delta^{-2s+\gamma+1} \chi_{ \{ \delta < \eta \} }  )}{\eta \delta^\gamma} \longrightarrow 0
\qquad \text{as }\eta\downarrow0
\end{align*}
by dominated convergence.
On the other hand, in $\{\delta>\eta/2\}$ we have, for $\sigma\in(0,2\gamma-2s+1)$
\begin{align*}
\Green(\delta^{-2s+\gamma+1}\chi_{\{\delta<\eta\}})\leq\eta^{1+\sigma}\Green(\delta^{-2s+\gamma-\sigma}\chi_{\{\delta<\eta\}})\leq\eta^{1+\sigma}\Green(\delta^{-2s+\gamma-\sigma})\asymp\eta^{1+\sigma}\delta^{\gamma-\sigma}
\end{align*}
where we have used Theorem~\ref{thm:range of exponents}.
As a consequence
\begin{align*}
\int_ { \{ \delta > \eta/2 \}  } f \delta^\gamma \frac{\Green (  \delta^{-2s+\gamma+1} \chi_{ \{ \delta < \eta \} }  )}{\eta \delta^\gamma}
\leq \eta^\sigma \int_ { \{ \delta > \eta/2 \}  } f \delta^{\gamma-\sigma}
\longrightarrow 0
\end{align*}
again by dominated convergence.

The proof of b) is analogous by using~\eqref{theta gamma equal}.

Let us now consider c). As above, by duality,
\begin{multline*}
\eta^{-1} \int_{ \{  \delta < \eta  \} } \frac{\Green(f)}{\delta^\gamma}
=\int_\Omega f\delta^\gamma\;\frac{\Green(\delta^{-\gamma}\chi_{\{\delta<\eta\}})}{\eta\,\delta^\gamma}\ = \\
=\ \int_{\{\delta>\eta/2\}} f\delta^\gamma\;\frac{\Green(\delta^{-\gamma}\chi_{\{\delta<\eta\}})}{\eta\,\delta^\gamma}
+\int_{\{\delta<\eta/2\}} f\delta^\gamma\;\frac{\Green(\delta^{-\gamma}\chi_{\{\delta<\eta\}})}{\eta\,\delta^\gamma}.
\end{multline*}
For the first integral we use~\eqref{eq: sharp interior 2} with $\beta=-\gamma$ to deduce
\begin{align*}
\int_{\{\delta>\eta/2\}} f\delta^\gamma\;\frac{\Green(\delta^{-\gamma}\chi_{\{\delta<\eta\}})}{\eta\,\delta^\gamma}
\leq \int_{\{\delta>\eta/2\}} f\,\delta^\gamma
\end{align*}
up to constants not depending on $\eta$.
For the second one we use~\eqref{first boundary estimates with theta} and~\eqref{theta gamma small} which give
\begin{align*}
\frac{\Green(\delta^{-\gamma}\chi_{\{\delta<\eta\}})}{\eta\,\delta^\gamma}\asymp
\frac{\delta^{-2\gamma+2s}}{\eta}+1+\frac{\delta(x)}{\eta}\leq 3 ,
\qquad\text{in }\{\delta<\eta\},
\end{align*}
and therefore
\begin{align*}
\int_{\{\delta<\eta/2\}} f\delta^\gamma\;\frac{\Green(\delta^{-\gamma}\chi_{\{\delta<\eta\}})}{\eta\,\delta^\gamma}
\leq C\int_\Omega |f|\,\delta^\gamma.
\end{align*}
This completes the proof for $f \ge 0$. 

If $f$ changes sign, then can we apply the result have to $f_+$ and $f_-$. 
\end{proof}

\subsection{Sharp weighted spaces for the Green operator}
\label{sec:sharp functional spaces}
The computations above allow to complement the analysis carried out in Theorem~\ref{thm:functional},
and improve the estimate for the optimal data from $L^1_{loc}$ to a weighted space.
It follows the general philosophy that, due to~\eqref{eq:vwf}, for any $\mu \in \cM(\Omega, \delta^\gamma)$ we have
\begin{equation*}
	\Green: L^1 \big( \Omega, \Green( \mu ) \big) \to L^1 \big(\Omega, \mu \big).
\end{equation*}
The result is as follows.
\begin{theorem}\label{thm:precise weights}
Assume that~\eqref{eq:K0}--\eqref{eq:K3} hold and let $\alpha>(\gamma-2s)\vee(-\gamma-1)$. We have that
\begin{equation*}
	\Green : L^1 \big(\Omega, \delta^\gamma\big) \to L^1 \big(\Omega,\delta^\alpha\big)
\end{equation*}
is well-defined and continuous.
\end{theorem}
\begin{proof}
Take $f\in L^1 (\Omega, \delta^\gamma)$. Then
\begin{align*}
\int_\Omega \big|\Green(f)\big|\delta^\alpha\leq
\int_\Omega |f|\,\Green\big(\delta^\alpha\big).
\end{align*}
As $\alpha>-\gamma-1$ by assumption, we can apply Theorem~\ref{thm:range of exponents}.
Since $\alpha>\gamma-2s$, then
$
\Green\big(\delta^\alpha \big)\asymp\delta^\gamma
$
and, therefore,
\begin{align*}
\int_\Omega \big|\Green(f)\big| \, \delta^\alpha \leq
C\int_\Omega|f| \, \delta^\gamma.
\end{align*}
This completes the proof.
\end{proof}

\begin{corollary}
Under the assumptions of Theorem~\ref{thm:precise weights}, if $\gamma<2s$, then solutions for any admissible data are in $L^1 (\Omega)$.
\end{corollary}
\begin{proof}
In the notations of Theorem~\ref{thm:precise weights}, notice that, if $\gamma<2s$, then $\alpha=0$ is an admissible choice.
\end{proof}

\label{sec:sharp boundary behaviour good data}

For $f \in L^\infty_c (\Omega)$ we have shown that $\Green(f) \asymp \delta^\gamma$. In order to study the sharp boundary behaviour, we want to study $\Green(f) / \delta^\gamma$. For this reason we introduce the following definition.
\begin{definition}\label{def:gamma normal derivative}
	We denote by
	\begin{equation*}
		D_\gamma u (z) := \lim_{\substack{x \to z \\ x \in \Omega}} \frac{u(x)}{\delta(x)^\gamma}
		\qquad z\in\partial\Omega,
	\end{equation*}
	and we call it normal $\gamma$-normal derivative of $u$.
\end{definition}
In order to prove sharp boundary behaviour we assume that the Green kernel has a $\gamma$-normal derivative $\Gnu$.
\begin{align}
	&\textrm{There exists $\Gnu:\partial \Omega \times \Omega \to \mathbb R$, such that, for every sequence } \Omega\ni x_j\to z \in \partial \Omega  \nonumber \\
	\label{eq:defn G nu}
	\tag{K5}
	& \textrm{we have } \lim_j \frac{ \G(x_j, \cdot) }{\delta (x_j)^\gamma} = \Gnu (z,\cdot)  \textrm{ a.e. in } \Omega .
\end{align}
\begin{remark}
	Notice that $\Gnu (z,y) = D_\gamma ( \G (\cdot, y) ) (z)$.	
\end{remark}

\begin{remark}
As a consequence of~\eqref{eq:K2} we have, for a.e. $y\in\Omega$ and $z\in\partial\Omega$,
\begin{align}
	\label{eq:bounds for Gnu}
	\Gnu(z,y)&\asymp\lim_{\substack{x \to z \\ x \in \Omega }}\frac1{{|x-y|}^{n-2s}}\left(\frac{\delta(y)}{{|x-y|}^2}\wedge\frac1{\delta(x)}\right)^\gamma
	=\frac{{\delta(y)}^\gamma}{{|z-y|}^{n-2s+2\gamma}} .
\end{align}
\end{remark}

\begin{remark}
Assumption~\eqref{eq:defn G nu} is satisfied in our three reference examples:
\begin{itemize}
	\item For the RFL, it follows from the Boundary Harnack Principle~\cite{bogdan97bhp} and the boundary regularity of solutions on smooth domains~\cite{Ros-Oton2014}; for $y\in\Omega,\psi\in C^\infty_c(\Omega)$ fixed and $\gamma=s$, we have
	\begin{align*}
		\frac{\G(x,y)}{\delta(x)^s}=\frac{\G(x,y)}{\Green(\psi)(x)}\frac{\Green(\psi)(x)}{\delta(x)^s},
		\qquad x\in\Omega.
	\end{align*}
	Both factors lie in $C^\alpha(\overline\Omega\setminus B_r(y))$, at least for $\alpha,r>0$ small enough.
	Indeed, the first one is due to~\cite[Theorem 1]{bogdan97bhp} and is a consequence of the $s$-harmonicity of the two involved functions close to the boundary; the second factor, instead, is more related to the smoothness of the boundary and a more classical Schauder regularity, see~\cite[Theorem 1.2]{Ros-Oton2014}. The kernel $\Gnu$ has been first introduced in~\cite{Abatangelo2015}, although it is strongly related to the Martin kernel, see for example~\cite{bogdan99representation}.
	\item For the SFL, the well-definition of $\Gnu$ is contained in~\cite[Lemma 14]{Abatangelo2017a}; in this case, the proof relies on a computation on an explicit representation formula for $\G$ in terms of the classical Dirichlet heat kernel.
	\item For the CFL, a Boundary Harnack Inequality is available (see~\cite[Section 6]{bogdan03censored}), so we can repeat the argument for the RFL. To this regard see, in particular,~\cite[Remark 6.1 and equation (6.35)]{bogdan03censored}.
\end{itemize}
\end{remark}

\begin{theorem}\label{thm:gamma normal derivative}
	Assume
		\eqref{eq:K0}--\eqref{eq:K3} and~\eqref{eq:defn G nu}
	and let $f \in L^1_c (\Omega)$. For $u = \Green(f)$, $D_\gamma u$ is well-defined on $\partial\Omega$. 		
	Furthermore
	\begin{equation}\label{represesentation dgamma u}
		D_\gamma u (z) = \int_{ \Omega } \Gnu (z,y) \, f(y) \; dy
	\end{equation}
and
\begin{equation*}
	|D_\gamma u (z)| \le C \| f \|_{L^1(\Omega)} \, \mathrm{dist}\big(z,\supp(f)\big)^{2s-2\gamma-n}.
\end{equation*}
\end{theorem}

\begin{proof}
	We write
	\begin{equation*}
		\frac{u(x)}{{\delta(x)}^\gamma} = \int_\Omega \frac{\G(x,y)}{{\delta(x)}^\gamma} \, f(y) \; dy.
	\end{equation*}
	Let $z\in\partial\Omega$, ${(x_j)}_{j\in\N}\subset\Omega$
	such that $x_j\to z$ as $j\uparrow\infty$, and $K = \supp(f)\Subset \Omega$.
	Then, up to constants, for $j$ sufficiently large
	\begin{equation*}
		\frac { \G(x_j,y) } {{\delta(x_j)}^\gamma} \le \textrm{dist}(x_j, K)^{2s-n-2\gamma} \delta(y)^{\gamma} \le \mathrm{dist}(K, \partial \Omega )^{2s-n-2\gamma}
		\qquad y\in K.
	\end{equation*}
	Therefore, since convergence a.e. in $y$ is given by~\eqref{eq:defn G nu}, by dominated convergence
	\begin{equation*}
		\frac{\G (x_j,\cdot)}{\delta(x_j)^\gamma} \, f\to \Gnu(z,\cdot) f
		\qquad \textrm{ in } L^1 (\Omega), \text{ as }j\uparrow\infty.
	\end{equation*}
	Thus
	\begin{equation*}
		\frac{u(x_j)}{{\delta(x_j)}^\gamma} = \int_\Omega \frac{\G(x_j,y)}{{\delta(x_j)}^\gamma} f(y) \; dy \to \int_\Omega \Gnu (z,y) f(y) \; dy
		\qquad \text{as }j\uparrow\infty.	
		\end{equation*}
	The limit is, by definition, $D_\gamma u (z)$.
	The pointwise estimate is a consequence of~\eqref{eq:bounds for Gnu}.
\end{proof}

\begin{remark}
	One needs to be careful with pathological cases that do not satisfy~\eqref{eq:defn G nu}, including
	\begin{equation*}
		\G (x,y) = \left( 2 + \sin \frac{1}{ \delta (x) } \right) \left( 2 + \sin \frac{1}{ \delta (y) } \right) \frac{1}{|x-y|^{n-2s} } \left( \frac{ \delta(x) \delta(y) }{|x-y|^2} \wedge 1 \right)^\gamma.
	\end{equation*}
\end{remark}

\begin{remark}
	Due to the lower Hopf estimates in Theorem~\ref{thm:lower Hopf}, if $f \ge 0$ we have
	\begin{equation*}
		D_\gamma u (z) \ge c \int_{ \Omega } f (y) \, \delta(y)^\gamma \; dy,
		\qquad z\in\partial\Omega.
	\end{equation*}
\end{remark}

\begin{proposition}
	\label{prop:convergence RHS of Martin wdf}
	Assume
		\eqref{eq:K0}--\eqref{eq:K3} and~\eqref{eq:defn G nu}.
	Let $j\in\N$, $A_j = \{ 1/j < \delta < 2/j \}$,~$u = \Green(f)$ for some~$f \in L^\infty_c (\Omega)$,
	and~$h$ be continuous on a neighbourhood of~$\partial \Omega$. Then
	\begin{equation*}
		\frac{1}{|A_j|} \int_{A_j} h(x) \, \frac{u(x)}{\delta(x)^\gamma}  \; dx
		\longrightarrow \frac{1}{|\partial \Omega|} \int_{\partial \Omega} h(z) \, D_\gamma u (z) \; dz
	\end{equation*}
\end{proposition}

\begin{proof}
	We write
	\begin{align*}
		\frac{1}{|A_j|} \int_{A_j} h(x) \frac{u(x)}{\delta(x)^\gamma} \; dx
		= \frac{j}{|A_j|} \int_{\partial \Omega} \frac{1}{j} \int_{1/j}^{2/j} h(z + \rho \mathbf n (z))\frac{u(z + \rho \mathbf n (z))}{\rho^\gamma} \; d\rho \; dz.
	\end{align*}
	Therefore, for every~$z \in \partial \Omega$,
	\begin{equation*}
		I_j (z) = \frac{1}j\int_{1/j}^{2/j} h(z + \rho \mathbf n (z)) \frac{u(z + \rho \mathbf n (z))}{\rho^\gamma}d \rho \ \longrightarrow\ h(z) D_\gamma u(z),
		\qquad \text{as }j\uparrow\infty.
	\end{equation*}
	Since~$I_j$ is a bounded function in~$\partial \Omega$ and pointwise convergent, by the dominated convergence theorem as~$j\uparrow\infty$
	\begin{equation*}
		\frac{j}{|A_j|} \int_{\partial \Omega} \frac{1}{j} \int_{1/j}^{2/j} h(z + \rho \mathbf n (z)) \frac{u(z + \rho \mathbf n (z))}{\rho^\gamma} \;d\rho \; dz \ \longrightarrow\  \frac{1}{|\partial \Omega|} \int_{\partial \Omega} h(z) D_\gamma u(z) \; dz.
	\end{equation*}
	This completes the proof.
\end{proof}

 \section{Limit of the interior theory: the \texorpdfstring{$\Ls$}{L}-harmonic pro\-blem}
\label{sec:Martin problem}

\subsection{Limit of the interior theory}
A classical approach known for the usual Laplacian to recover the non-homogeneous Dirichlet boundary problem is to concentrate all mass towards the boundary.

\begin{theorem}
\label{prop:u star}
Let~$\Green$ satisfy~\eqref{eq:K0}--\eqref{eq:K3} and~\eqref{eq:defn G nu}.
Let~$A_j = \{1/j < \delta < 2/j \}$,~$j\in\mathbb{N}$,
and
\begin{align*}
		f_j = \frac { |\partial \Omega| \, \chi_{A_j}} {|A_j| \, \delta^\gamma}
\end{align*}
such that~$\| f_j \delta^\gamma \|_{L^1} = |\partial \Omega|$. Then, there exists a function in~$u^\star \in L^1_{loc} (\Omega)$ such that
\begin{equation*}
	\Green(f_j) \rightharpoonup u^\star, \qquad \textrm{ in } L^1 (K) \textrm { for every } K \Subset \Omega.
\end{equation*}
Furthermore,~$u^\star$ is a solution of
\begin{equation*}
	\int_\Omega u^\star \psi = \int_{\partial \Omega} D_\gamma \big( \Green(\psi) \big), \qquad \text{for any } \psi \in L^\infty_c (\Omega).
\end{equation*}
and is given by
\begin{equation}
	\label{eq:repr u star}
	u^\star(x) = \int_ {\partial \Omega} \Gnu (z,x) \; dz.
\end{equation}
\end{theorem}
\begin{proof}
	It is clear that~$\supp( f_j )= \overline A_j$ and~$\|f_j \delta^\gamma\|_{L^1(\Omega)} = |\partial \Omega|$. Therefore, due to Lemma~\ref{lem:uniform integrability over compacts}, a subsequence of~$\Green(f_j)$,~$\Green (f_{j}^{(1)})$, is weakly convergent in~$L^1 (\{  \delta  \ge 1  \})$ to a function~$u_1$. A further subsequence,~$\Green (f_{j}^{(2)})$, converges in~$L^1 (\{ \delta  \ge \frac 1 2 \})$ to a function~$u_2$. Iterating the process, we construct sequences~$f_j^{(m)}$ and functions~$u_m$ defined on~$\{\delta>1/m\}$, for every~$m\in\mathbb N$.
	Applying Proposition~\ref{prop:convergence RHS of Martin wdf} we have that
	\begin{align*}
		\int_ \Omega \Green(f_j)  \psi &= \int_{ \Omega } \Green (\psi)  \frac {|\partial \Omega| \, \chi_{ A_j }} {|A_j| \, \delta^\gamma}
		=  \frac{|\partial \Omega|}{|A_j|} \int_{ A_j }  \frac{ \Green (\psi ) }{\delta^\gamma} \to \int_{ \partial \Omega } D_\gamma [\Green (\psi)]
	\end{align*}
	for any~$\psi \in L^\infty_c (\Omega)$. Therefore,
	\begin{equation*}
		\int_\Omega u_m \psi = \int_{ \partial \Omega } D_\gamma [\Green (\psi)], \qquad \text{for any } \psi \in L^\infty_c (\Omega), \textrm{ such that } \supp \psi \subseteq \{ \delta \ge 1 / m \}.
	\end{equation*}
	For~$m > k$, using~$\psi = \sign(u_m-u_k) \chi_{\{\delta \ge 1/k \}}$ as a test function, we check that
	$u_m|_{\delta \ge 1/k} = u_k$.
	We define~$u^\star (x) = u_m (x)$ for any~$m > 1/ \delta(x)$.
	Given~$\psi \in L^\infty_c (\Omega)$, we~$u^\star \psi = u_m \psi~$ for any~$m>1/\textrm{dist}(\supp \psi,\partial \Omega)$. Therefore
	\begin{align*}
		\int_ \Omega u^\star \psi &= \int_{ \partial \Omega } D_\gamma [\Green (\psi)]\,  \qquad \text{for any } \psi \in L^\infty_c (\Omega).
	\end{align*}
	If we now consider the Green representation, we get
	\begin{multline*}
		\int_{ \partial \Omega } D_\gamma [\Green (\psi)] \; dz =  \int_{\partial \Omega} \left( \int_\Omega \Gnu(z,x) \, \psi(x) \; dx \right) dz \ = \\
		=\ \int_\Omega \psi (x) \left( \int_{\partial \Omega} \Gnu(z,x) \; dz \right) dx.
	\end{multline*}
	With this representation formula we show that all convergent subsequences share a  limit, and therefore the whole sequence converges.
	\end{proof}
	
	\begin{remark}
	In Figure~\ref{fig:towards Martin} we show a numerical simulation of the behaviour of the approximating sequence for the case of the SFL, under different values of~$s$.
	\end{remark}
	
\begin{figure}[p]
	\centering
	\begin{minipage}[t][12cm][b]{.32\textwidth}
                    \centering
                    \includegraphics[width=\textwidth]{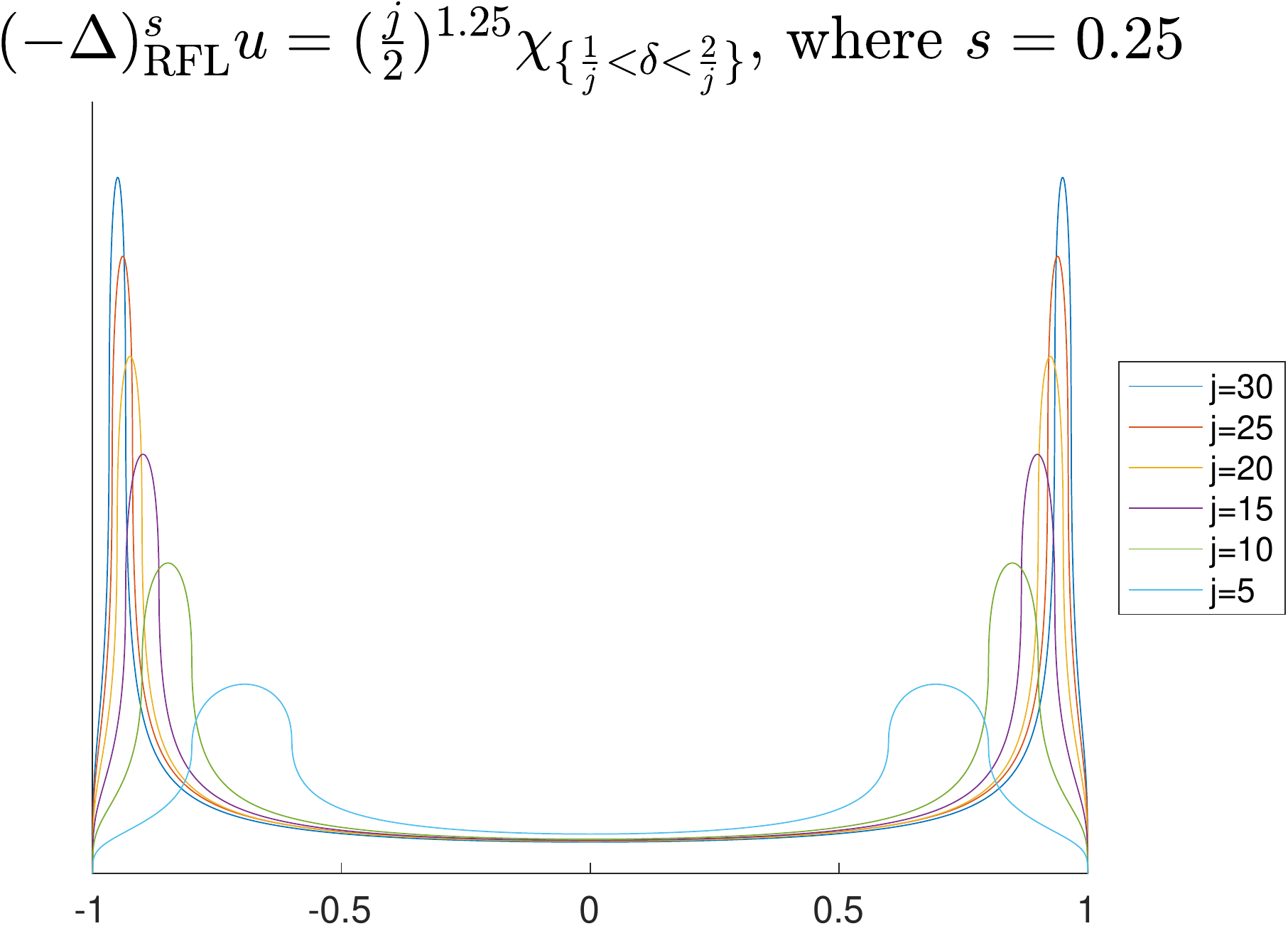}%
                    
                    \includegraphics[width=\textwidth]{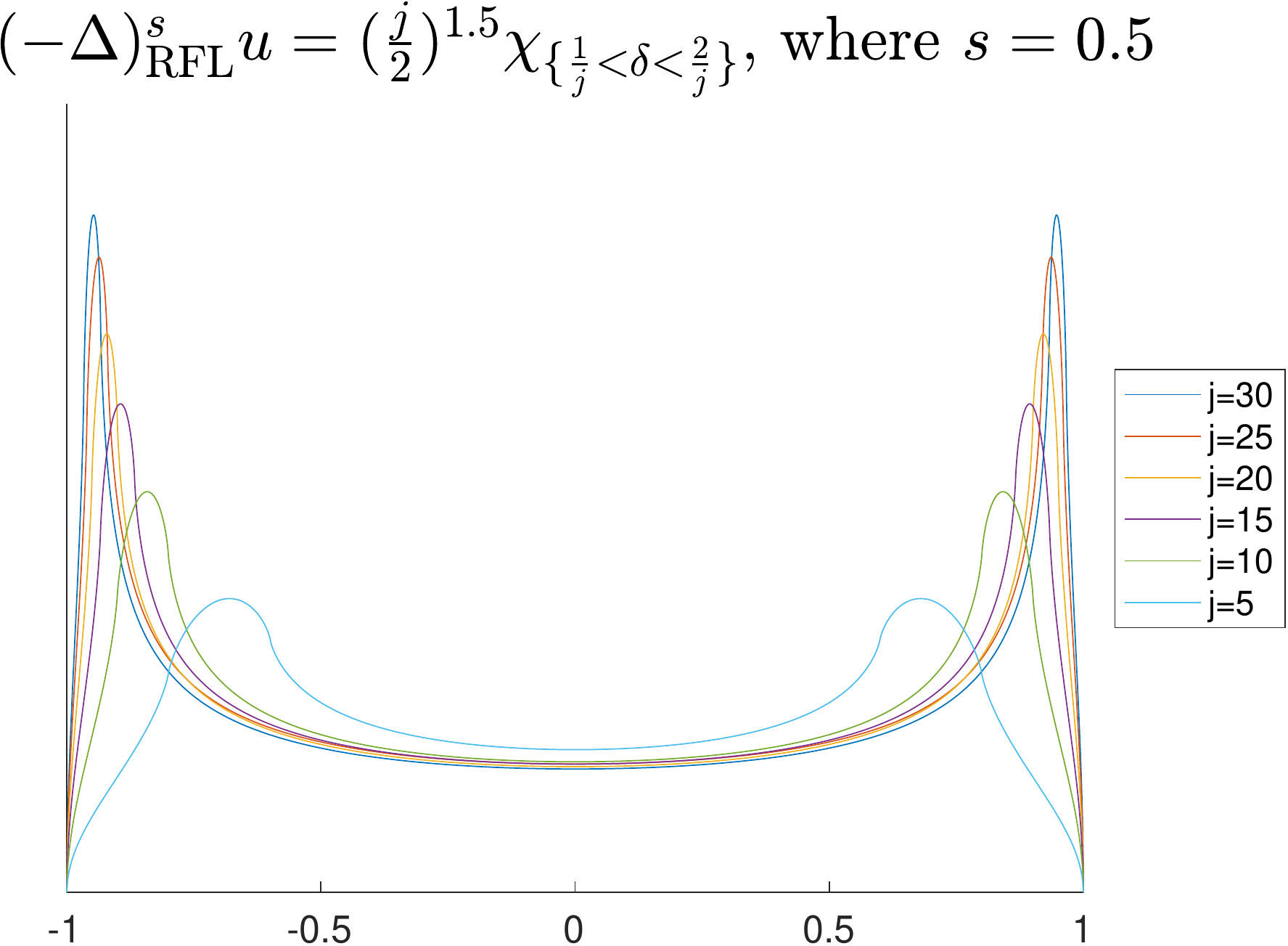}
                    
                    \includegraphics[width=\textwidth]{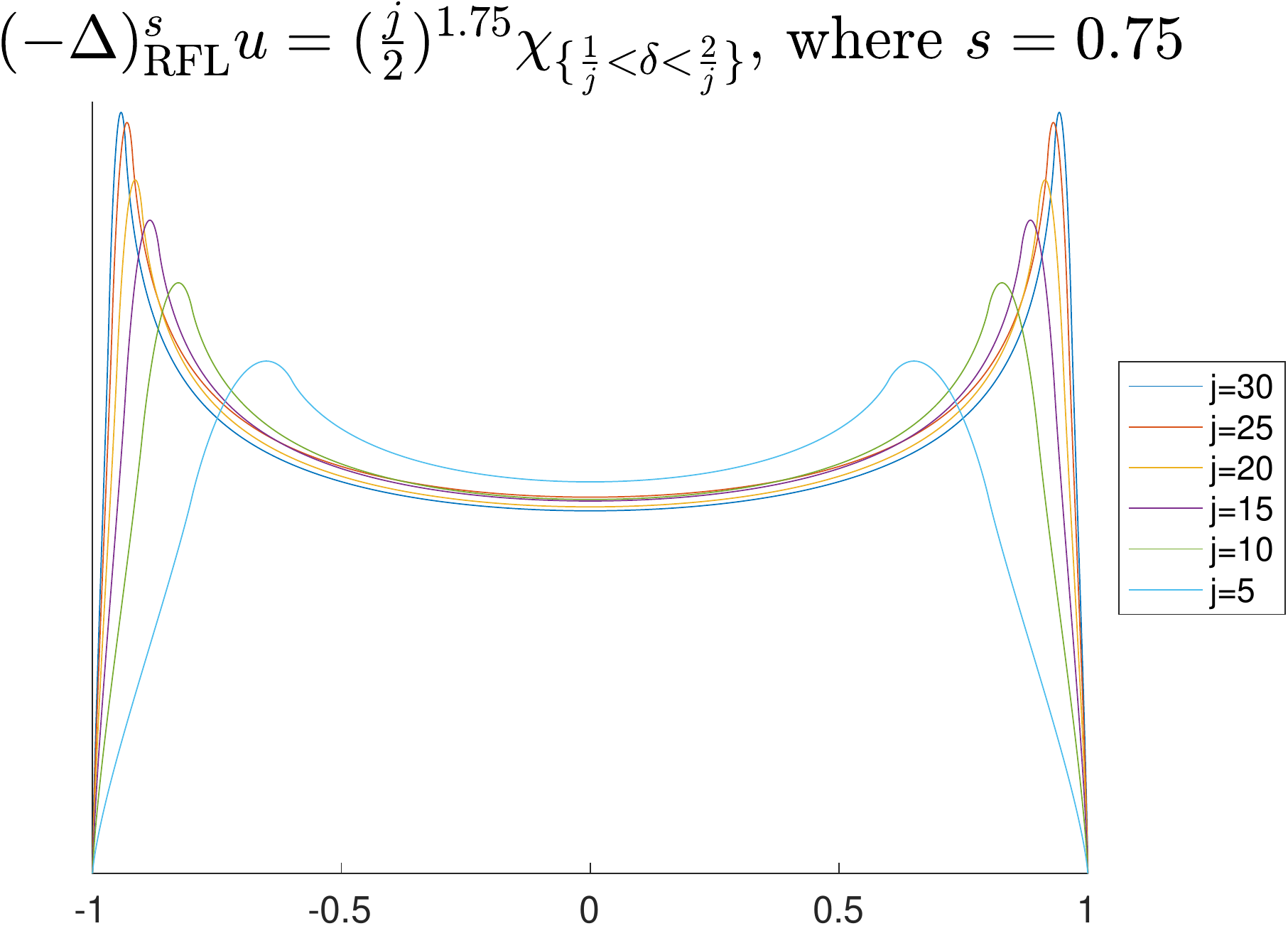}
                    
                    \includegraphics[width=\textwidth]{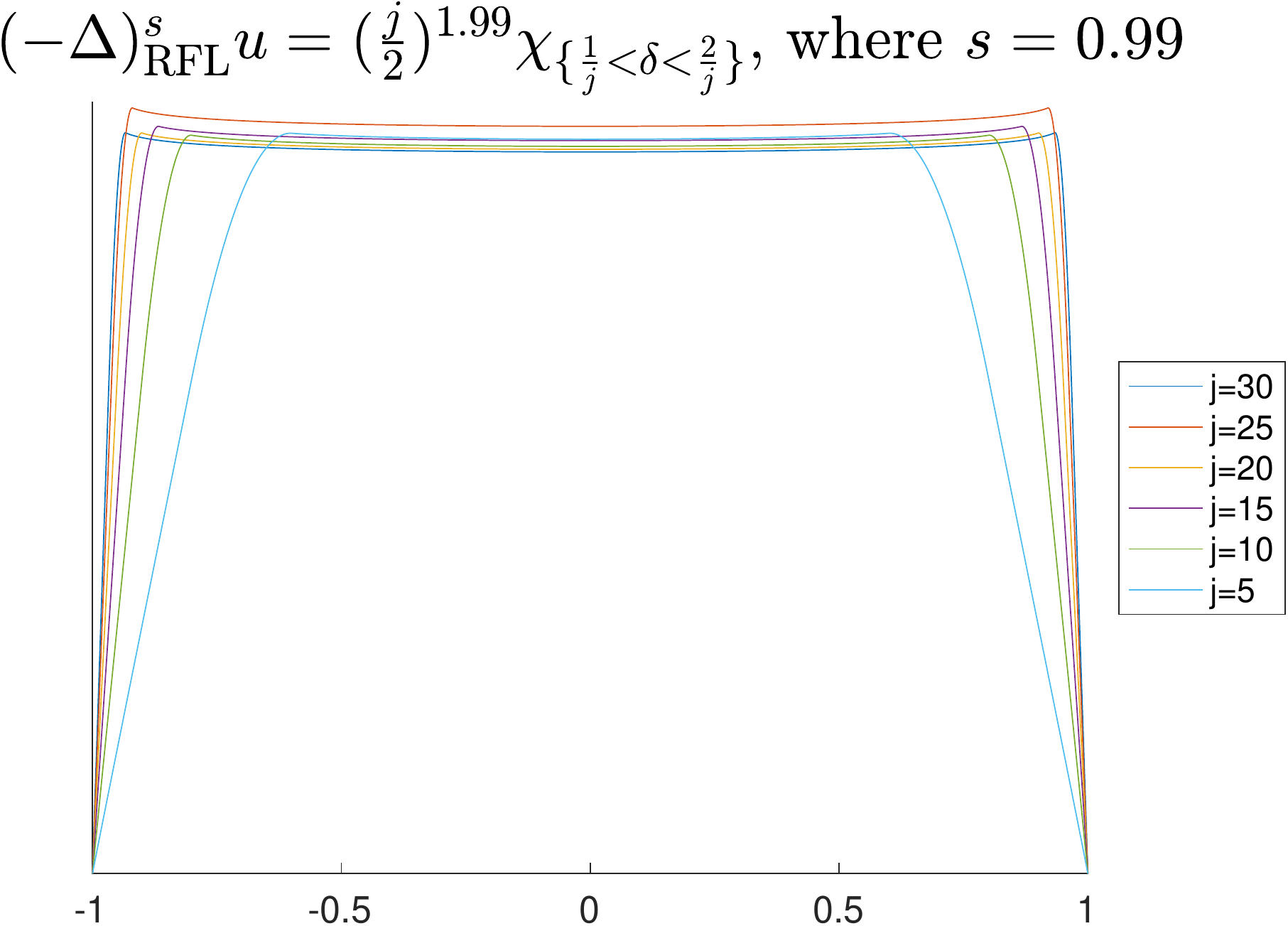}
                    
                    \subcaption{RFL}
                    
                    \label{fig:1b}
    \end{minipage}
    \begin{minipage}[t][12cm][b]{.32\textwidth}
                    \centering
                    \includegraphics[width=\textwidth]{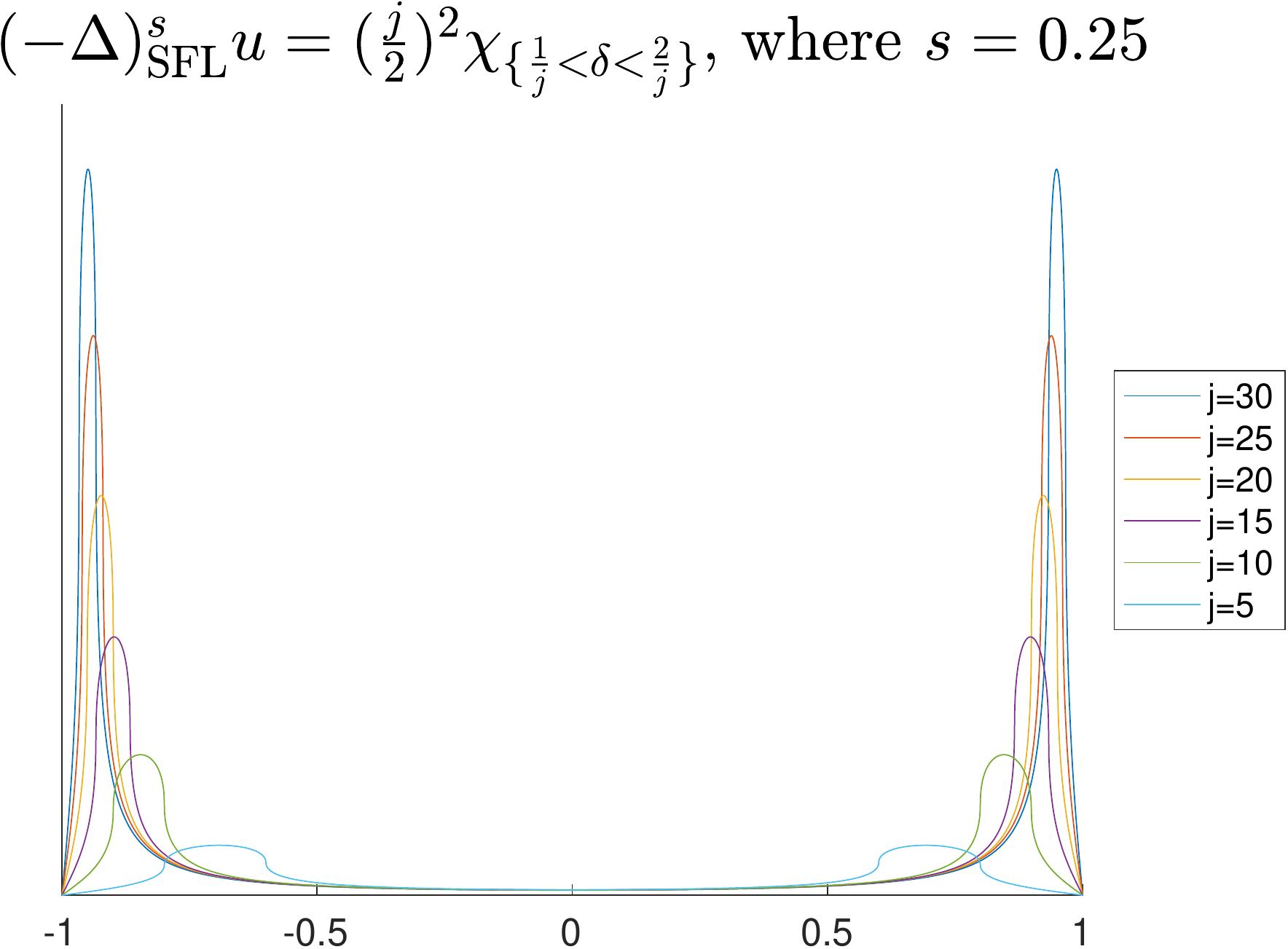}%
                    
                    \includegraphics[width=\textwidth]{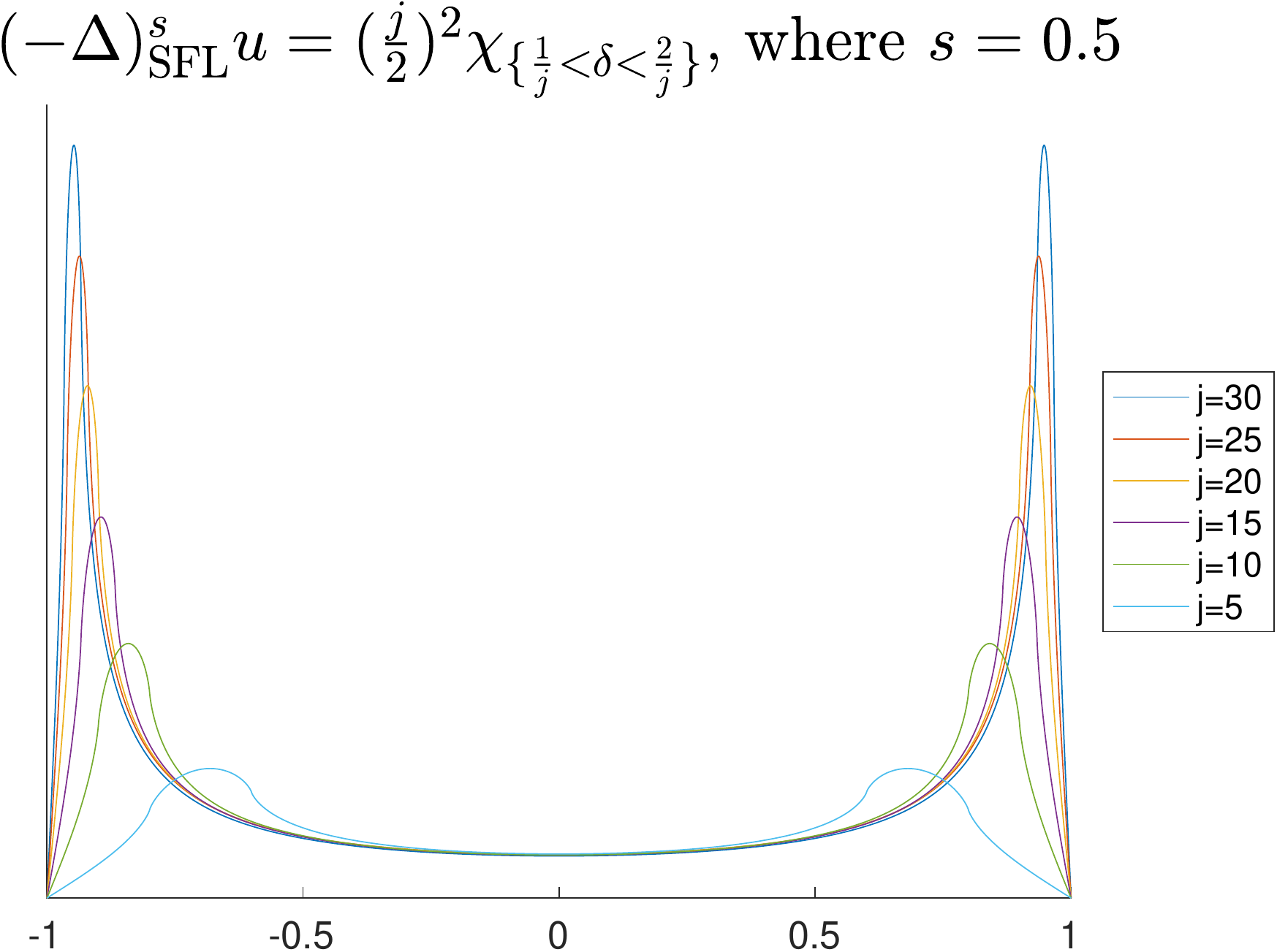}
                    
                    \includegraphics[width=\textwidth]{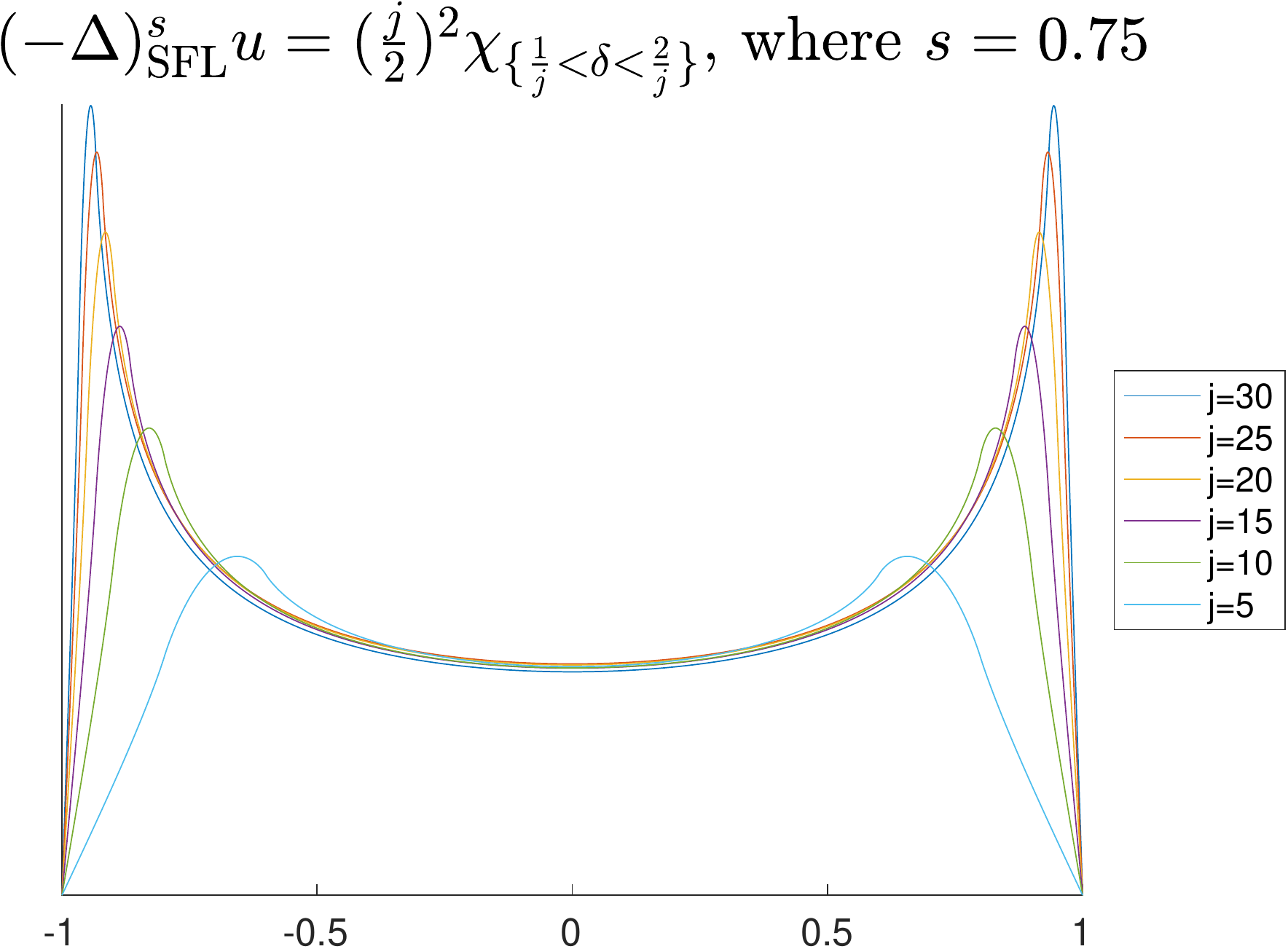}
                    
                    \includegraphics[width=\textwidth]{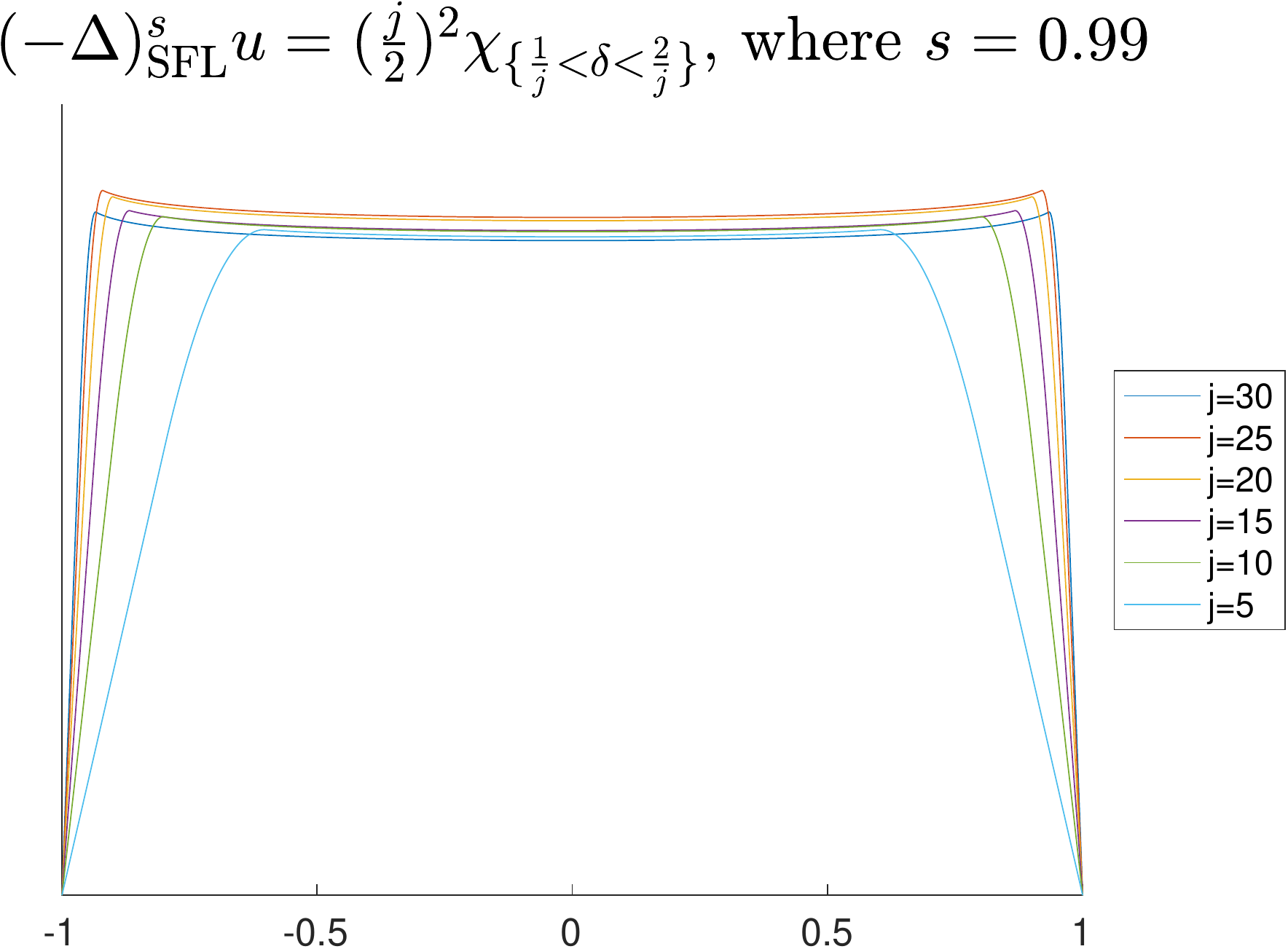}
                    
                    \subcaption{SFL}
                    
                    \label{fig:1b}
    \end{minipage}
    \begin{minipage}[t][12cm][b]{.32\textwidth}
                    \centering
                    
                    \includegraphics[width=\textwidth]{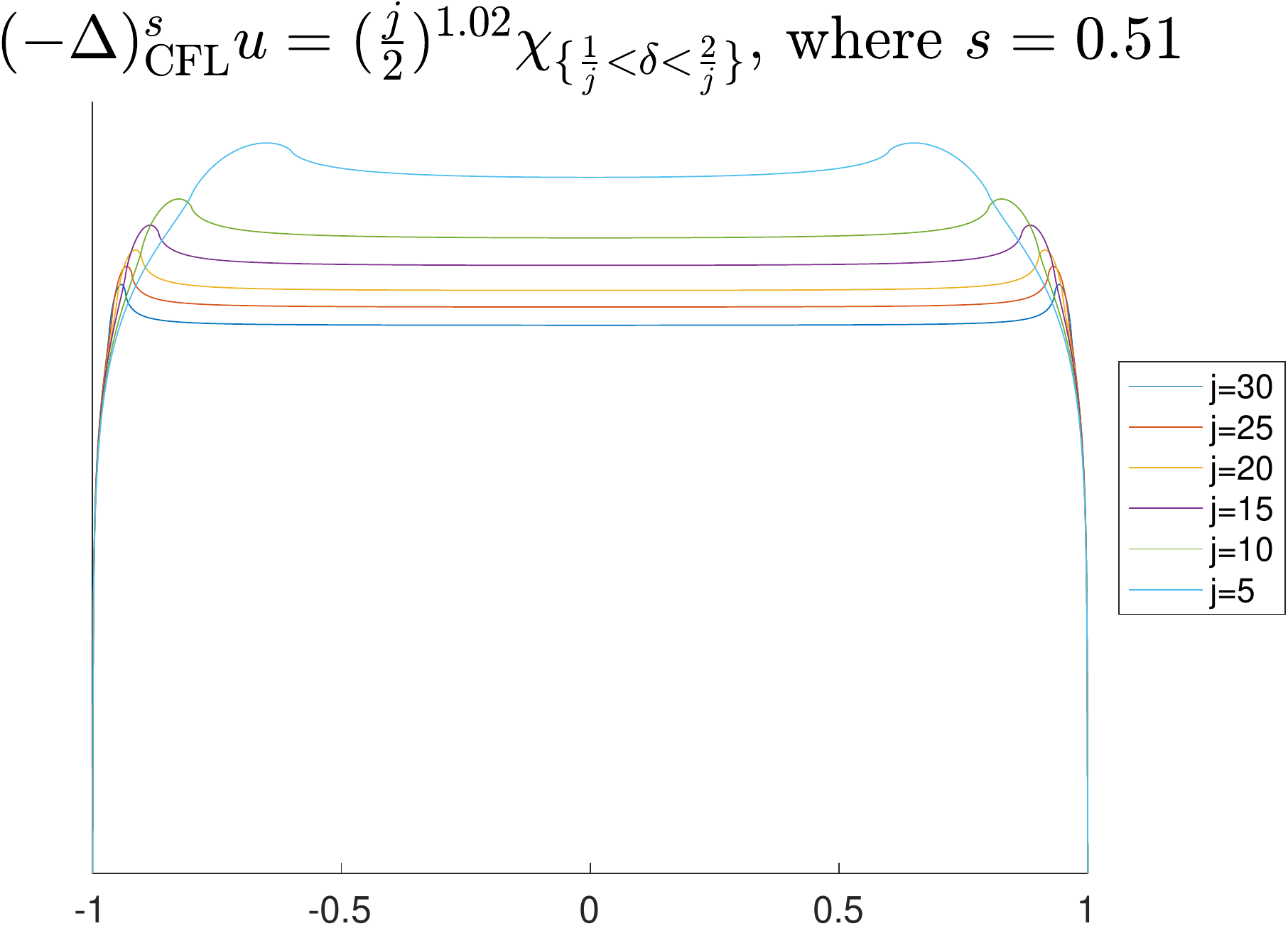}     
                                        
                    \includegraphics[width=\textwidth]{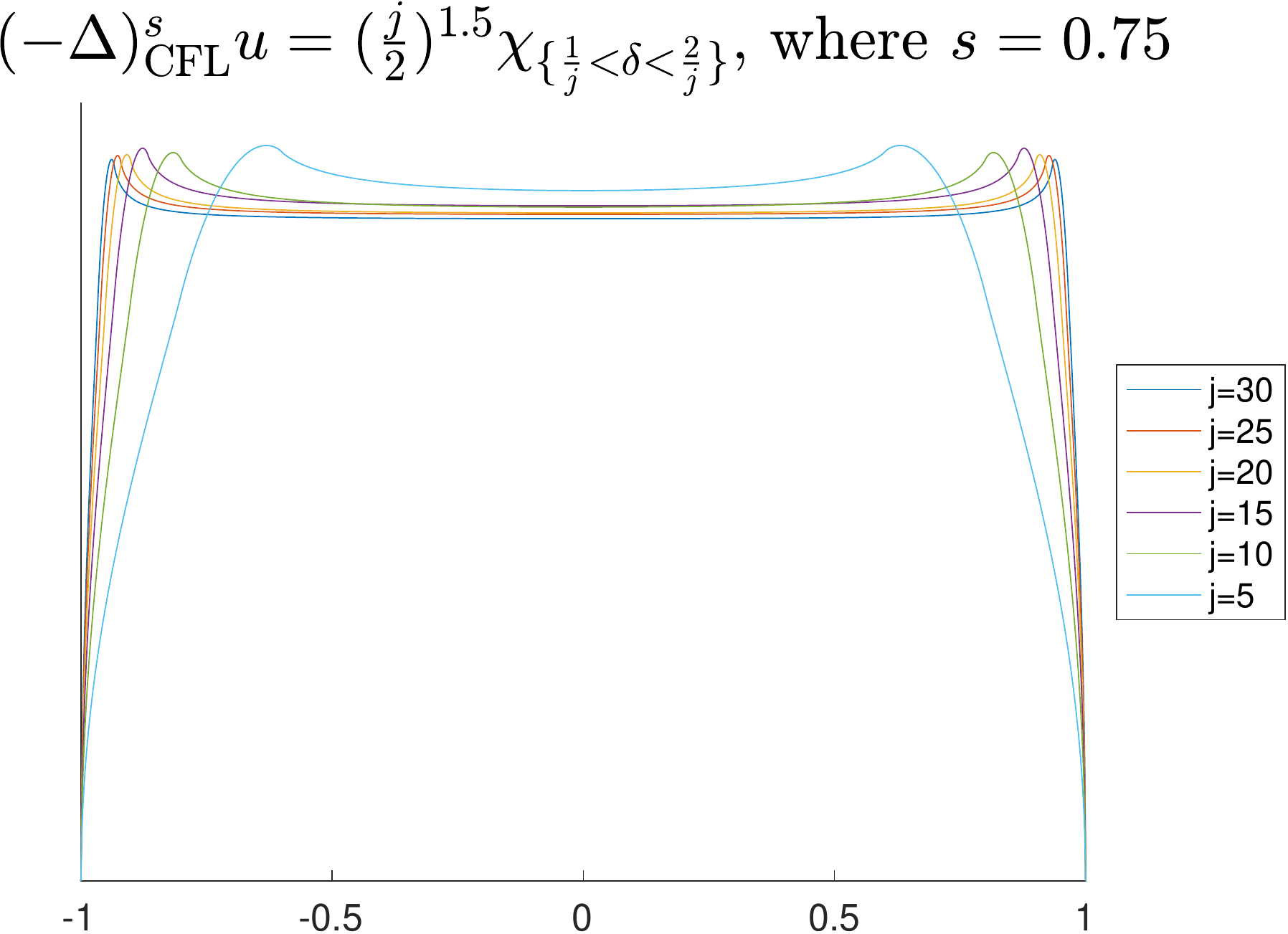}
                    
                    \includegraphics[width=\textwidth]{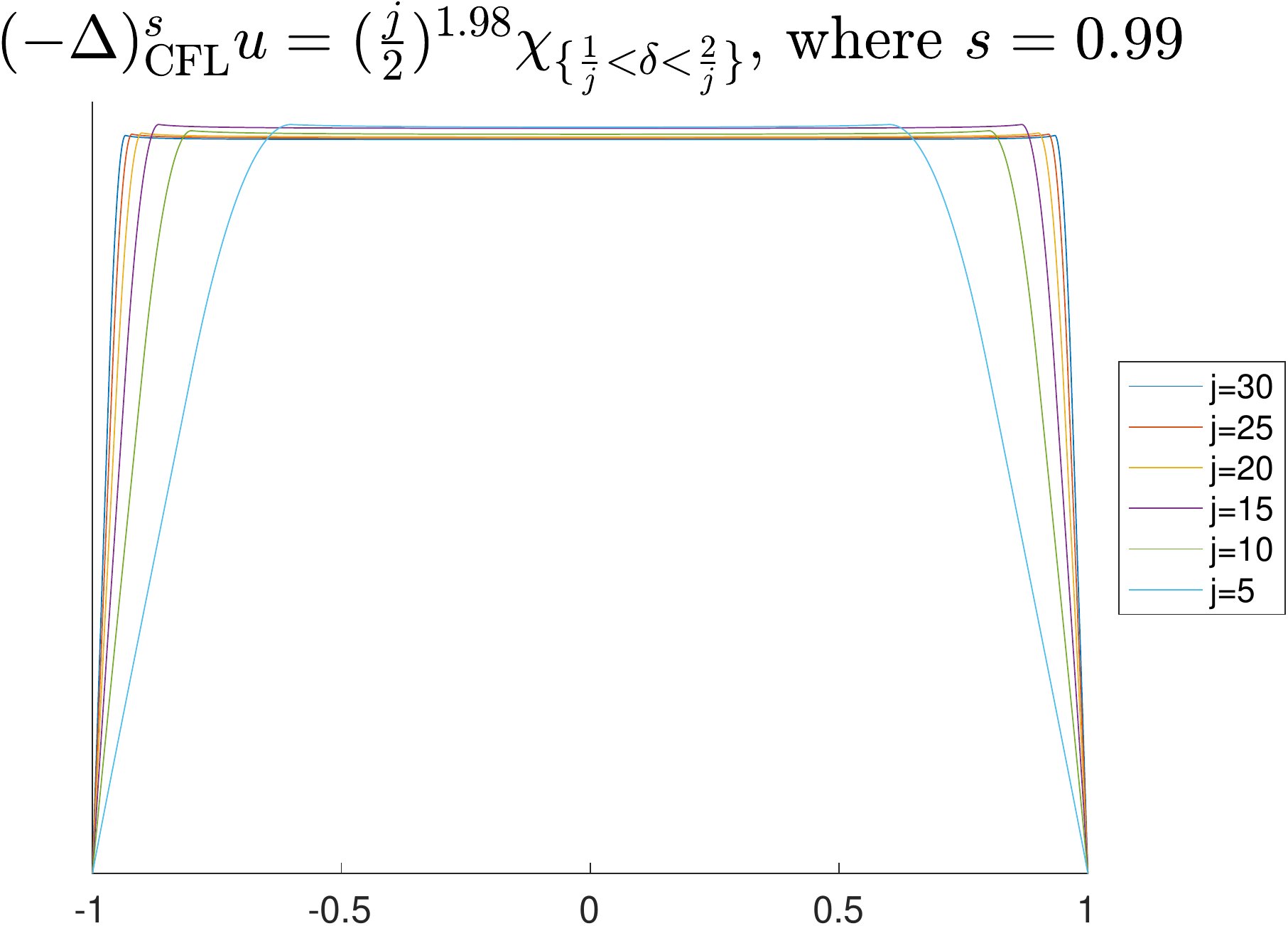}
                    
                    \subcaption{CFL}
                    
    \end{minipage}    
    \caption{
    Numerical solutions of~$ \Ls u = (\frac j 2)^{1+\gamma} \chi_{\frac 1 j < \delta < \frac 2 j}$ in dimension~$n = 1$. In the limit as $j \to +\infty$, we recover the profile of solutions to the $\Ls$-harmonic problem. We implement the schemes introduced in Figure~\ref{fig:subcritical}.}
    \label{fig:towards Martin}
\end{figure}
	
	\begin{corollary}
	 Under the assumptions and notations of Theorem~\ref{prop:u star}, it holds
		\begin{equation}
			\label{eq:estimates for u star}
			u^\star \asymp
			\begin{dcases}
				\delta^{2s-\gamma -1 } & \gamma > s - \frac 1 2 \\
				\delta^\gamma \, \big(1+|\!\ln \delta|\big) & \gamma = s - \frac 1 2 \\
				\delta^\gamma & \gamma < s - \frac 1 2
			\end{dcases}
			\qquad\text{in }\Omega.
		\end{equation}
	\end{corollary}
	\begin{proof}
	This follows by plugging~\eqref{eq:bounds for Gnu} into~\eqref{eq:repr u star}.
	Indeed,
	\begin{align}\label{eq:estimates for u star 2}
	u^\star(x)\asymp {\delta(x)}^\gamma \int_{\partial\Omega} \frac{dz}{{|z-x|}^{n-2s+2\gamma}},
	\qquad x\in\Omega,
	\end{align}
	where
	\begin{align*}
	\int_{\partial\Omega} \frac{dz}{{|z-x|}^{n-2s+2\gamma}}
	\asymp\begin{dcases}
				{\delta(x)}^{2s-2\gamma-1} & 2s-2\gamma-1 < 0 \\
				1+|\!\ln \delta(x)| & 2s-2\gamma-1 = 0 \\
				1 & 2s-2\gamma-1 > 0
	\end{dcases}
	\qquad x\in\Omega,
	\end{align*}
	which completes the proof.
\end{proof}

\begin{remark}
	Notice that, for~$\gamma > s - \frac 1 2$,~$u^\star$ has the limit rate~$\delta^{2s-\gamma-1}$ which is not accessible to solutions of the interior problem.
\end{remark}


\begin{remark}
	The function~$u^\star$ is a large solution
	(\textit{i.e.},~$u^\star (x) \to +\infty$ as~$\delta(x) \to 0$)  if
		and only if~$2s-\gamma - 1 <  0$.
	 We have:
	\begin{enumerate}
		\item In the RFL case~$\gamma = s$, so~$2s-\gamma-1 = s-1 <  0$.
		\item In the SFL case~$\gamma = 1$, so~$2s-\gamma-1 = 2(s-1) <  0$.
		\item In the CFL case~$\gamma = s - \frac 1 2$, so~$2s-\gamma - 1 =0$ for~$\frac 1 2 < s < 1$;
			in this case~$u^\star$ is not singular.
	\end{enumerate}
\end{remark}

\subsection{The \texorpdfstring{$\Ls$}{L}-harmonic problem}

For a self-adjoint operator in our class of study it makes sense to consider the following boundary problem
\begin{equation}
\label{eq:nonhomogeneous Martin with differential operator}
\int_ \Omega u \, \Ls \varphi = \int_{ \partial \Omega } h \, D_\gamma \varphi
\end{equation}
for some suitable test functions~$\varphi$. In the case of the usual Laplacian, this is the non-homogenenous Dirichlet problem with data~$h$.
This very weak formulation was first studied in~\cite{Brezis1971}.

Passing to our weak-dual formulation,~\eqref{eq:nonhomogeneous Martin with differential operator} is written
\begin{equation}
	\label{eq:nonhomogeneous Martin}
		\int_ \Omega u \psi = \int_{ \partial \Omega } h \, D_\gamma [\Green (\psi)] \qquad \text{for any } \psi \in L^\infty_c (\Omega).
	\end{equation}
	
	Heuristically,~\eqref{eq:nonhomogeneous Martin} can be read as an~$\Ls$-harmonicity condition for~$u$ in~$\Omega$, \textit{i.e.},~$\Ls u=0$ in~$\Omega$.
	
		In order to understand this weak-dual problem, we proceed informally. If one takes~$\psi = \mu_x$, the Dirac delta, we obtain
		\begin{equation*}
			u(x) = \int_{\partial \Omega} h(z) \, \Gnu(z,x) \; dz.
		\end{equation*}
		We will see in Theorem~\ref{thm:Martin classical traces}
		that
		\begin{equation*}
			\frac{u(x)}{u^\star(x)} = \int_{\partial \Omega} h(z) \frac{\Gnu(z,x)}{u^\star(x)} \; dz \longrightarrow h(\theta).
		\end{equation*}
		
Hence, in some sense~\eqref{eq:nonhomogeneous Martin} is a formulation of the problem~\eqref{eq:Martin}. We will devote this section to rigorously proving these intuitions.

\subsection{Existence, uniqueness and kernel representation}

We have the following theorem of well-posedness.

\begin{theorem}
	\label{thm:well posedness of Martin problem}
	Let~$\Omega$ be a smooth domain, assume~\eqref{eq:K0}--\eqref{eq:K3} and~\eqref{eq:defn G nu} and~$h \in L^1 (\partial \Omega)$. Then, there exists a unique~$u \in L_{loc}^1 (\Omega)$ satisfying~\eqref{eq:nonhomogeneous Martin}.
	Furthermore,
	\begin{enumerate}
	\item This unique solution can be represented by
	\begin{equation}
		\label{eq:solution Martin kernel G nu}
		u(x) = \int_{\partial \Omega} \Gnu (z,x) \, h(z) \; dz,
		\qquad \text{for }x\in\Omega.
	\end{equation}
	\item We have the estimate
	\begin{align}
		\| u \|_{L^\infty (K)} &\le C\, \mathrm{dist}(K, \partial \Omega)^{2s-n-\gamma} \| h \|_{L^1(\partial \Omega)}. \label{martin estimate 2}
	\end{align}
	\item
		If~$h \in C (\partial \Omega)$ then there exists a sequence~${(f_j)}_{j\in\mathbb{N}}\subset L^1 ( \Omega, \delta^\gamma )$ such that
	\begin{equation}
	\Green (f_j) \rightharpoonup u \qquad \textrm{ in } L^1_{loc} (\Omega), \text{ as }j\uparrow\infty. \label{martin limiting}
	\end{equation}
	\end{enumerate}
\end{theorem}

\begin{proof}
	The uniqueness is immediate to prove. Let~$u_1, u_2$ be two solutions, then~$u = u_1 - u_2$ satisfies
	\begin{equation*}
		\int_\Omega u \psi = 0 \qquad \text{for any } \psi \in L^\infty_c (\Omega).
	\end{equation*}
	In particular, let~$K\Subset \Omega$ and take~$\psi = \sign (u) \chi_K$. Then
	\begin{equation*}
		\int_K |u| = 0.
	\end{equation*}
	Since this is true for all~$K$, we have that~$u = 0~$ a.e. in~$\Omega$. Hence~$u_1 = u_2$.\smallskip
	
	The kernel representation~\eqref{eq:solution Martin kernel G nu} follows as in the proof of Theorem~\ref{prop:u star}, by exchanging the order of integration in~\eqref{eq:nonhomogeneous Martin}. Notice that this kernel representation can be rigurously justified on its own and therefore grant uniqueness. Nevertheless, since we will construct it as a limit of the interior theory, this is not needed.\smallskip

	We prove existence,~\eqref{martin estimate 2}, and~\eqref{martin limiting} simultaneously.
	We split the proof into different steps.\smallskip
	
	Let us first assume that~$0 \le h \in C (\partial \Omega)$. Using the notations defined in Remark~\eqref{rmk:tubular}, we extend the definition of~$h$ to the interior by setting
	\begin{equation*}
		H(x) = h (z(x)).
	\end{equation*}
	Recall that~$\delta(x)=|x-z(x)|,\ z(x)\in\partial\Omega$, and notice that~$H\in C(\{ \delta < \ee\})$.
	
	Let us define, for~$j > \frac 1 \varepsilon$, the sequence
	\begin{equation*}
		f_j = H \, \frac {|\partial \Omega| \chi_{A_j} }{|A_j| \delta^\gamma} \qquad \in L^1 (\Omega).
	\end{equation*}
	We check that this sequence is bounded in~$L^1 (\Omega, \delta^\gamma)~$ by estimating
	\begin{equation*}
		\int_ \Omega f_j \delta^\gamma  = \int_{A_j} \frac{ |\partial \Omega|}{|A_j|} H = \frac{|\partial \Omega|}{|A_j|} \int_{A_j} H \le |\partial \Omega| \| H \|_{L^\infty(\Omega)}.
	\end{equation*}
	We define
	$
		u_j = \Green(f_j).
	$
	
	We now show local~$L^1$-weak convergence.
	Let~$K \Subset \Omega$. For any~$A \subset K$ we have that
	\begin{equation*}
		\int_A u_j \le C_{K,\beta} |A|^\beta \| f_j \delta^\gamma \|_{L^1(\Omega)} \le C_{K,\beta} \| H \|_{L^\infty(\Omega)} |A|^\beta,
	\end{equation*}
	for some~$\beta > 0$, by Lemma~\ref{lem:uniform integrability over compacts}. Therefore the sequence
	$u_j$ is equi-integrable in~$K$ and it admits
	a subsequence~$u_{j_k}$ weakly convergent to some
	$u_K\in L^1(K)$.
	That is, if we consider~$\psi \in L^\infty (\Omega)$, with~$\supp \psi \subseteq K$, we have that
	\begin{equation*}
		\int_{ \Omega } u_{j_k} \psi \to \int_\Omega u_K \psi,
		\qquad\text{as }k\uparrow\infty.
	\end{equation*}
	On the other hand, by Proposition~\ref{prop:convergence RHS of Martin wdf}, we have that
	\begin{align*}
		\int_\Omega \Green(\psi) f_j &= \frac{|\partial\Omega|}{|A_j|} \int_{A_j} \frac{\Green (\psi)}{\delta^\gamma} \, H   \longrightarrow \int_{\partial \Omega} D_\gamma [\Green(\psi)] \, h.
	\end{align*}
	Therefore
	\begin{equation*}
		\int_\Omega u_K \psi  =  \int_{\partial \Omega} D_\gamma [\Green(\psi)] h , \qquad \text{for any } \psi \in L^\infty, \text{ with } \supp \psi \subseteq K.
	\end{equation*}
	
	For two compacts~$K,K'\Subset\Omega$ and the corresponding~$u_K,u_{K'}$ built as above, we actually have	$u_K = u_{K'}$ in~$K \cap K'$. Indeed,
	let us consider the test function
	\begin{equation*}
		\psi = \begin{dcases}
 					\sign(u_K - u_{K'}) & \text{in } K \cap K', \\
 					0 & \text{in }\Omega\setminus(K\cap K').
 				\end{dcases}
	\end{equation*}
	It is an admissible test function for both~$u_K$ and~$u_{K'}$. Therefore
	\begin{equation*}
		\int_{K \cap K'} |u_K - u_{K'}| = 0.
	\end{equation*}
	
	We define now
	\begin{equation*}
		u (x) = u_{K(x)} (x) \qquad x\in\Omega,\textrm{ where } K(x) = \{ y\in\Omega : \delta(y) \ge \delta(x)/2\}.
	\end{equation*}
	We have shown above that any converging subsequence of~$u_j$ converges weakly to~$u$ over compacts. In particular,~$u_j \rightharpoonup u$ in~$L^1_{loc}$. By construction~$u$ solves~\eqref{eq:nonhomogeneous Martin}.
	
	Passing to the limit the estimate
	in Theorem~\ref{thm:L 1 to L 1 loc estimate}
	\begin{equation*}
		\int_K u_j \le C_K \| f_j \delta^\gamma \|_{L^1},
	\end{equation*}
	we deduce that, as~$j \uparrow\infty$,
	\begin{equation*}
		\int_K u \le C_K \| h \|_{L^1 (\partial \Omega)}.
	\end{equation*}
	
	Moreover, in view of~\eqref{eq:L 1 to L inf outside support}
	we have that
	\begin{equation*}
		\| u_j \|_{L^\infty(K)} \le C \mathrm{dist}(K, A_j)^{2s-n-\gamma} \| f_j \delta^\gamma \|_{L^1(\Omega)}.
	\end{equation*}
	We deduce that the sequence~$u_j$ converges weak-$\star$ in~$L^\infty (K)$ to~$u$, and that
	\begin{equation*}
		\| u \|_{L^\infty (K)} \le C \mathrm{dist}(K, \partial \Omega)^{2s-n-\gamma} \| h \|_{L^1 (\partial \Omega)}.
	\end{equation*}
	
	We now consider
	$ 0 \le h \in L^1 (\partial \Omega)$.
	We take an approximation sequence
	$0\le h_k \in C (\partial \Omega)$
	converging to~$h$ in~$L^1(\partial \Omega)$.
	The sequence~$u_k$ of solutions corresponding to~$h_k$ can be
	constructed through the previous step. Due to the estimates, we can pass to the limit over compacts and apply the uniqueness reasoning above
	to recover a function~$u \in L^\infty_{loc}$ solution of~\eqref{eq:nonhomogeneous Martin} with data~$h$.
	
	For~$h \in L^1 (\partial \Omega)$, we can decompose it as~$h = h_+ - h_-$, construct solutions~$u_1$ and~$u_2$ corresponding to~$h_+$ and~$h_-$ and recover~$u = u_1 - u_2$ satisfying all properties.
	
	This completes the proof.
	\end{proof}

\begin{corollary}
	\label{cor: u / u star in L1}
	In the assumptions of Theorem~\ref{thm:well posedness of Martin problem} and for~$u$ defined as in~\eqref{eq:solution Martin kernel G nu},
	we have that
		\begin{equation*}
			\left\|  \frac{u}{u^\star} \right\|_{L^1(\Omega)} \le C \| h \|_{L^1 (\partial \Omega)}.
		\end{equation*}
\end{corollary}
\begin{proof}
	It holds
	\begin{align*}
	\int_\Omega \frac{|u(x)|}{u^\star(x)}\;dx \leq
	\int_{\partial\Omega} |h(z)| \int_\Omega \frac{ D_\gamma \mathbb G(z,x) }{ u^\star(x) } \; dx \; dz.
	\end{align*}
	In view of~\eqref{eq:bounds for Gnu} and~\eqref{eq:estimates for u star}, for~$z\in\partial\Omega$,
	\begin{align*}
	\int_\Omega \frac{ D_\gamma \mathbb G(z,x) }{ u^\star(x) } \; dx \asymp
	\left\lbrace\begin{aligned}
		& \int_\Omega \frac{ \delta(x)^{2\gamma-2s+1} }{ {|x-z|}^{n-2s+2\gamma} } \; dx  & & \text{if }\gamma>s-\frac12, \\
		& \int_\Omega \frac{ dx }{ {|x-z|}^{n-1} \, \big(1+|\!\ln\delta(x)|\big) }   & & \text{if }\gamma=s-\frac12, \\
		& \int_\Omega \frac{ dx }{ {|x-z|}^{n-2s+2\gamma} }   & & \text{if }\gamma<s-\frac12.
	\end{aligned}\right.
	\end{align*}
	
	When~$\gamma<s-1/2$ then~$2\gamma-2s<-1$, which implies
	\begin{align*}
	\int_\Omega \frac{ dx }{ {|x-z|}^{n-2s+2\gamma} } \asymp 1,
	\qquad \text{on }\partial\Omega.
	\end{align*}
	When~$\gamma\geq s-1/2$, it suffices to use relation~$\delta(x)\leq|x-z|$ for any~$x\in\Omega,z\in\partial\Omega$
	in order to deduce
	\begin{align*}
	\int_\Omega \frac{ \delta(x)^{2\gamma-2s+1} }{ {|x-z|}^{n-2s+2\gamma} } \; dx
	\leq \int_\Omega \frac{ dx }{ {|x-z|}^{n-1} } \asymp 1, \qquad \text{on }\partial\Omega,
	\end{align*}
	and
	\begin{align*}
	\int_\Omega \frac{ dx }{ {|x-z|}^{n-1} \, \big(1+|\!\ln\delta(x)|\big) } \leq
	\int_\Omega \frac{ dx }{ {|x-z|}^{n-1} \, \big(1+\big|\!\ln|x-z|\big|\big) } \asymp 1, \qquad \text{on }\partial\Omega.
	\end{align*}
\end{proof}

\begin{corollary}
Under the assumptions of Theorem~\ref{thm:well posedness of Martin problem},
\normalcolor
the solution operator to problem~\eqref{eq:nonhomogeneous Martin}
	\begin{equation}
		\label{eq:Martin functional spaces}
		\Martin: L^1 (\partial \Omega) \to L^\infty_{loc} (\Omega)
	\end{equation}
	is linear, continuous, and it admits the kernel representation
	\begin{equation}
		\label{eq:Martin kernel representation}
		\Martin (h ) (x) = \int_ {\partial \Omega} \M (x,z) \, h(z) \; dz,
		\qquad x\in\Omega,
	\end{equation}
	where~$\M$ is given by
	\begin{equation}
		\label{eq:Martin kernel from Gnu}
		\M(x,z) = \Gnu (z,x).
	\end{equation}
	Furthermore,
	for any~$\alpha>(\gamma-2s)\vee(-\gamma-1)$
	\begin{equation}
		\label{eq:Martin functional spaces L1}
		\Martin: L^1 (\partial \Omega) \to L^1 (\Omega, \delta^\alpha) \textrm{ is continuous}.
	\end{equation}
	with operator norm
	\begin{equation*}
		\| \Martin \|_{L^1 (\partial \Omega ) ; L^1 (\Omega,\delta^\alpha)} \le \| \Green \|_{L^1 (\Omega, \delta^\gamma ) ; L^1 (\Omega,\delta^\alpha)}.
	\end{equation*}
\end{corollary}

\begin{proof}
	The results~\eqref{eq:Martin functional spaces},~\eqref{eq:Martin kernel representation}, and~\eqref{eq:Martin kernel from Gnu} follow immediately from Theorem~\ref{thm:well posedness of Martin problem}.

	Now, let us prove~\eqref{eq:Martin functional spaces L1}. First, let~$0 \le h \in C(\partial \Omega)$. By recalling the construction in Theorem~\ref{thm:well posedness of Martin problem}, there is a sequence~$f_j \ge 0$ such that~$0 \le u_j = \Green(f_j)$ with~$\| f_j \delta^ \gamma \|_{L^1 (\Omega)} = \| h \|_{L^1 (\partial \Omega)}$ such that~$u_j \rightharpoonup u$ in~$L^1_{loc} (\Omega)$. Going back to the proof of Theorem~\ref{thm:precise weights} we observe that,
	\begin{equation*}
		\int_{ \Omega } u_j \delta^\alpha  \le
		C_\Green \int_\Omega f_j  \, \delta^\gamma = C_\Green \int_{\partial \Omega} h,
	\end{equation*}
	where~$C_\Green := \| \Green \|_{L^1 (\Omega, \delta^\gamma ) ; L^1 (\Omega,\delta^\alpha)}$.
	Since we only have convergence over compact sets, we assure that, for any~$K \Subset \Omega$,
	\begin{equation*}
		\int_{ K  } u_j \delta^\alpha \le  C_\Green \int_{\partial \Omega} h.
	\end{equation*}
	Since~$\chi_{K} \delta^\alpha \in L^\infty_c(\Omega)$ and~$u_j$ converges weakly in~$L^1_{loc}(\Omega)$
	\begin{equation*}
		\int_K u\delta^\alpha = \lim_{j \to +\infty} \int_{ K  } u_j \delta^\alpha \le  C_\Green \int_{\partial \Omega} h.
	\end{equation*}
	Since this holds for any~$K \Subset \Omega$ and~$C_\Green$ does not depend on~$K$, then~$u \delta^\alpha \in L^1 (\Omega)$ and
	\begin{equation}
		\label{eq:Martin alpha L1 continuity constant}
		\int_\Omega u\delta^\alpha \le  C_\Green  \int_{\partial \Omega} h.
	\end{equation}
	If~$0 \le h \in L^1 (\partial \Omega)$, we can construct an approximating sequence~$0 \le h_j \in C(\partial \Omega)$ and we recover~\eqref{eq:Martin alpha L1 continuity constant} by passing to the limit.
	
	If~$h \in L^1 (\partial \Omega)$ is sign-changing, we repeat the argument for~$h_+$ and~$h_-$ and apply~\eqref{eq:Martin alpha L1 continuity constant} to deduce
	\begin{equation*}
		\int_\Omega |u|\delta^\alpha \le  C_\Green \int_{\partial \Omega} |h|.
	\end{equation*}
	This completes the proof.
\end{proof}

\begin{remark}
	Let~$u = \Martin (h)$. Notice that, since~$2s - \gamma - 1 < 0$, ~\eqref{eq:Martin functional spaces L1} shows that~$u \delta^{\ee + \gamma - 2s} \in L^1 (\Omega)$ for any~$\ee > 0$. This is sharper than Corollary~\ref{cor: u / u star in L1}, which only guarantees that~$u  \delta^{1 + \gamma - 2s} \asymp u /u^* \in L^1 (\Omega)$.
\end{remark}

\begin{remark}
	\label{lem:boundary behaviour of martin}

	Due to the estimates for~$\Gnu$, we know that
	\begin{equation*}
	\M(x,z) \asymp \frac{\delta(x)^\gamma}{|x-z|^{n+\gamma-(2s-\gamma)}} \qquad x\in\Omega,z\in\partial\Omega.
	\end{equation*}

\end{remark}

\begin{remark}
	In the classical case, this corresponds to the usual Poisson kernel. For~$\Ls = (-\Delta)^s_\RFL$, this somehow corresponds to the existing notion of Martin kernel (see~\cite{Abatangelo2015,bogdan99representation}).
\end{remark}

\begin{remark}
	Notice that
	\begin{equation*}
	\left| \dfrac{\M(x,z)}{\int_{ \partial \Omega }  \M (x,z') dz'} \right| \le C \frac{ \delta(x)^{2\gamma+1-2s} }{ {|x-z|}^{n-2s+2\gamma} }
	\qquad x\in\Omega,z\in\partial\Omega.
	\end{equation*}
\end{remark}

\subsection{Boundary behaviour of solutions of the \texorpdfstring{$\Ls$}{L}-harmonic problem}

\subsubsection{Bounded data}
\begin{theorem}
	\label{thm:Martin classical traces}	Let us assume~\eqref{eq:K0}--\eqref{eq:K3} and~\eqref{eq:defn G nu}. Let~$h \in C(\partial\Omega)$ and~$\gamma>s-\frac12$.
	Then, the unique solution
		$u\in L^1_{loc} (\Omega)$
	of~\eqref{eq:nonhomogeneous Martin} satisfies
	\begin{align*}
	\lim_{ \substack{ x \to \theta \\ x \in \Omega} } \frac{ u (x) }{u^\star(x) } = h(\theta)
	\end{align*}
	uniformly in~$\theta\in\partial\Omega$.
\end{theorem}

\begin{proof}
	We estimate, for~$x \in \Omega$ and up to multiplicative constants,
	\begin{multline*}
	\left|\frac{u(x)}{u^\star(x)} -h(\theta)\right| = \left| \frac{\int_ {\partial \Omega} \M(x,z) h(z) \, dz}{\int_{ \partial \Omega}  \M(x,z') \, dz' } - \frac{\int_ {\partial \Omega} \M(x,z) h(\theta)\, dz}{\int_ {\partial \Omega} \M(x,z') \, dz'} \right| = \\
	= \left|  \int_{ \partial \Omega } \frac{\M(x,z)}{\int_{ \partial \Omega }  \M (x,z') dz'} (h(z) - h(\theta)) \; dz    \right|
	\le \delta(x)^{-2s+2\gamma+1}  \int_{\partial\Omega} \frac{|h(z)-h(\theta)|}{|x-z|^{n+2\gamma-2s}} \; dz,
	\end{multline*}
	where we have used Lemma~\ref{lem:boundary behaviour of martin}.
	Fix now~$\ee>0$ arbitrarily small and let~$\eta>0$ small enough in order to have~$|h(z)-h(\theta)|<\ee$ for any~$z\in\partial\Omega\cap B(\theta,\eta)$. Note that, since~$\partial\Omega$ is compact and~$h$ is continuous, then~$\eta$ is independent of~$\theta$ by uniform continuity. Then we have
	\begin{multline*}
	\delta(x)^{-2s+2\gamma+1}  \int_{\partial\Omega\cap B(\theta,\eta)} \frac{|h(z)-h(\theta)|}{|x-z|^{n+2\gamma-2s}} \; dz \ \leq \\
	\leq\ \delta(x)^{-2s+2\gamma+1}  \int_{\partial\Omega} \frac{\ee}{|x-z|^{n+2\gamma-2s}} \; dz \leq \ee
	\end{multline*}
	and
	\begin{multline*}
	\delta(x)^{-2s+2\gamma+1}  \int_{\partial\Omega\setminus B(\theta,\eta)} \frac{|h(z)-h(\theta)|}{|x-z|^{n+2\gamma-2s}} \; dz 	\ \leq \\
	\leq\ \ee^{-n-2\gamma+2s}\|h\|_{L^\infty(\partial\Omega)}\delta(x)^{-2s+2\gamma+1}\ \longrightarrow\ 0
	\end{multline*}
	as~$x\to\theta$. The above yields that
	\begin{align*}
	\limsup_{x\to\theta}\left|\frac{u(x)}{u^\star(x)} -h(\theta)\right|\leq\ee,
	\end{align*}
	but, since~$\ee$ is arbitrary, our claim is proved.
\end{proof}

\begin{remark}
	The reasoning above also holds for~$\gamma=s-\frac12$ under suitable modifications.
	For~$\gamma<s-\frac12$ the integration
	\[
	\int_{\partial\Omega} |\theta-z|^{-n-2\gamma+2s} h(z) dz
	\]
	makes perfect sense for~$\theta\in\partial\Omega$, so no normalization will ever be able to improve
	\[
	\lim_{x\to\theta} \delta(x)^{-\gamma} u(x) \asymp \int_{\partial\Omega} |\theta-z|^{-n-2\gamma+2s} h(z) dz.
	\]
	Heuristically, it seems like that the Martin kernel is \textit{not} singular enough to
	select only the values of~$h$ around~$\theta$ when passing to the limit.
	The kernel seems to be too ``spread around''.
\end{remark}

\subsubsection{Integrable data}

\begin{theorem}
	\label{thm:Martin boundary behaviour integrable data}
	Let~$h \in L^1 (\partial \Omega)$ and~$\gamma\geq s-\frac12$. Then, for any~$\phi\in C(\overline{\Omega})$, it holds
	\begin{equation*}
		\frac1\eta\int_{ \{  \delta < \eta  \} } \frac{\Martin(h)}{u^\star}\,\phi
		\ \longrightarrow\ \int_{\partial\Omega} h\,\phi
		\qquad\text{as }\eta\downarrow 0.
	\end{equation*}
\end{theorem}
\begin{proof}
	Notice that the claim holds for~$h\in C(\partial\Omega)$ by Theorem~\ref{thm:Martin classical traces}.
	For a general~$h\in L^1(\partial\Omega)$, let us consider a sequence~${(h_k)}_{k\in\mathbb{N}}\subset C(\partial\Omega)$ such that
	$\|h_k-h\|_{L^1(\partial\Omega)}\downarrow 0$ as~$k\uparrow\infty$. Then split
	\begin{multline*}
	\left|\eta^{-1} \int_{ \{  \delta < \eta  \} } \frac{\Martin(h)}{u^\star}\,\phi - \int_{\partial\Omega} h\,\phi \right| \leq
		\left| \eta^{-1} \int_{ \{  \delta < \eta  \} } \frac{\Martin(h)-\Martin(h_k)}{u^\star}\,\phi \right| + \\
		+\left| \eta^{-1} \int_{ \{  \delta < \eta  \} } \frac{\Martin(h_k)}{u^\star}\,\phi - \int_{\partial\Omega} h_k\,\phi \right|
		+\left| \int_{\partial\Omega} (h_k-h)\,\phi \right|.
	\end{multline*}
	Fix~$\ee>0$ arbitrarily small and let~$k\in\mathbb{N}$ large enough to have~$\|h_k-h\|_{L^1(\partial\Omega)}<\ee$.
	The above inequality and Theorem~\ref{thm:Martin classical traces} entail
	\begin{align*}
	\limsup_{\eta\downarrow0}\left|\eta^{-1} \int_{ \{  \delta < \eta  \} } \frac{\Martin(h)}{u^\star}\,\phi - \int_{\partial\Omega} h\,\phi \right| \leq
		\limsup_{\eta\downarrow0}\left| \eta^{-1} \int_{ \{  \delta < \eta  \} } \frac{\Martin(h)-\Martin(h_k)}{u^\star}\,\phi \right| + \ee,
	\end{align*}
	for any~$k\in\mathbb{N}$ large enough. Write
	\begin{align*}
	& \eta^{-1} \int_{ \{  \delta < \eta  \} } \frac{\Martin(h)-\Martin(h_k)}{u^\star}\,\phi \ = \\
	& =\ \eta^{-1} \int_{ \{  \delta < \eta  \} } \frac{\phi(x)}{u^\star(x)}\int_{\partial\Omega}\M(x,z)\,\big(h_k(z)-h(z)\big)\;dz\;dx\ = \\
	& =\ \int_{\partial\Omega} \big(h_k(z)-h(z)\big) \eta^{-1} \int_{ \{  \delta < \eta  \} } \M(x,z)\,\frac{\phi(x)}{u^\star(x)}\;dx\;dz ,
	\end{align*}
	in order to deduce that, up to constants, it holds
	\begin{align*}
	&	\left|\eta^{-1} \int_{ \{  \delta < \eta  \} } \frac{\Martin(h)-\Martin(h_k)}{u^\star}\,\phi\right| \ \leq \\
	& 	\leq \|\phi\|_{L^\infty(\Omega)}\|h_k-h\|_{L^1(\partial\Omega)}\sup_{z\in\partial\Omega}
	\eta^{-1} \int_{ \{  \delta < \eta  \} } \frac{\M(x,z)}{u^\star(x)}\;dx \\
	&	\leq \|\phi\|_{L^\infty(\Omega)}\|h_k-h\|_{L^1(\partial\Omega)}\sup_{z\in\partial\Omega}
	\eta^{-1} \int_{ \{  \delta < \eta  \} } \frac{\delta(x)^{-2s+2\gamma+1}}{|x-z|^{n+2\gamma-2s}}\;dx .
	\end{align*}
	By the coarea formula it holds
	\begin{multline*}
	\int_{ \{  \delta < \eta  \} } \frac{\delta(x)^{-2s+2\gamma+1}}{|x-z|^{n+2\gamma-2s}}\;dx
	=\int_0^\eta t^{-2s+2\gamma+1} \int_{\{x:\delta(x)=t\}} |x-z|^{-n-2\gamma+2s} \; dx\ = \\
	=\ \int_0^\eta t^{-2s+2\gamma+1} \, t^{-1-2\gamma+2s} \; dt = \eta
	\end{multline*}
	and therefore
	\begin{align*}
	\left|\eta^{-1} \int_{ \{  \delta < \eta  \} } \frac{\Martin(h)-\Martin(h_k)}{u^\star}\,\phi\right|\leq
	\ee\|\phi\|_{L^\infty(\Omega)}.
	\end{align*}
	The case~$\gamma=s-\frac 1 2$ follows by a similar argument.
\end{proof}


\section{Comments and open problems}
\label{sec:comments}

\begin{enumerate}

	\item The case {$2s-\gamma -1 > \gamma$} (\textit{i.e.},~$\gamma < s - \frac 1 2$) seems to pose problems to uniqueness.
Indeed in this case
\begin{equation*}
	M(1) \asymp \delta^{(2s-\gamma -1) \wedge \gamma} \longrightarrow 0\
	\qquad \text{as }x\to\partial\Omega.
\end{equation*}
It seems that problem
\begin{equation*}
	\begin{dcases}
		\Ls u = f & \text{in }\Omega \\
		u = 0 & \text{on }\partial \Omega
	\end{dcases}
\end{equation*}
does not have a unique solution, as~$\Green(f) + \Martin(h)$ is also a solution for any~$h \in C(\partial \Omega)$. Therefore, the construction of the Green operator assumed at the beginning (which chooses a single solution), seems to be made by applying some additional \emph{selection criteria}.
	This phenomenon should be studied.

	\item\label{sec:system RFLs} In trying to construct an example satisfying~$\gamma < s - \frac 1 2$ relation, we have considered the following example:
		let~$f \in L^\infty_c (\Omega)$ and consider the system
		\begin{equation*}
		\left\lbrace\begin{aligned}
			(-\Delta)^{\frac s k}_\RFL v_1 &= f & & \text{in }\Omega \\
			(-\Delta)^{\frac s k}_\RFL v_1 &= v_2 & & \text{in }\Omega \\
			& \ \vdots \\
			(-\Delta)^{\frac s k}_\RFL v_{k-1} &= v_{k-2} & & \text{in }\Omega \\
			(-\Delta)^{\frac s k}_\RFL u &= v_{k-1} & & \text{in } \Omega \\
			v_1 = \ldots = v_{k-1}= u &= 0 & & \text{in }\mathbb R^n\setminus\overline\Omega.
		\end{aligned}\right.
		\end{equation*}
		Then
		\begin{equation*}
		u = \Green (f) = \Green_{{\frac s k}} \circ \ldots \circ \Green_{\frac s k} (f)
		\end{equation*}
		where~$\Green_{\frac s k}$ is the Green operator of~$(-\Delta)_\RFL^{\frac s k}$.
		It seems that~$\Green$ is self-adjoint and, for~$1/2 < s < 1$, we expect its kernel to be of the form
		\begin{equation*}
		\G (x,y) \asymp |x-y|^{2s-n} \left( \frac{\delta(x) \delta(y)}{|x-y|^2} \wedge 1 \right)^{\frac s k}.
		\end{equation*}
		
	\item The operators that admit exterior data~$u = g$ in~$\mathbb R^n\setminus\overline\Omega$ (\textit{e.g.}, the RFL) have an exterior kernel, that is sometimes denoted by~$\mathbb P(x,y)$,~$x \in \Omega, y \in \mathbb R^n\setminus\overline\Omega$ (in the case of the RFL it holds that~$\mathbb P(x,y) = -(-\Delta)^s_y \G(x,y)$). It seems reasonable that the singular solutions of type~$u^\star$ can also be detected from the outside, as it has been done for instance in~\cite[Lemma 7]{bogdan99representation} and~\cite[Lemma 3.6]{Abatangelo2015}.
	
	\item
	Notice that, so far, we have given all our estimates in terms of~$u / u^\star$. However, it would nice to give an operator~$\widehat\Martin$ such that
	\begin{equation*}
		\lim_{x \to z} \frac{\widehat \Martin (h) (x)}{\delta(x)^{2s-\gamma - 1}} = h(z) , \qquad z\in\partial\Omega.
	\end{equation*}
	Nevertheless, the boundary behaviour of~$u^\star$ is only known in terms of rate. An interesting question is if the following limit is defined
	\begin{equation*}
		K(z) = \lim_{x \to z} \frac{u^\star (x)}{\delta(x)^{2s-\gamma - 1}} , \qquad z\in\partial\Omega.
	\end{equation*}
	This seems to be a further assumption on the kernel.
	If it is, then~$K \asymp 1$, we can set
	\begin{equation*}
		\widehat \Martin (h) = \Martin \left( \frac{h}{K}  \right)
	\end{equation*}
	so that
	\begin{equation*}
		\lim_{x \to z} \frac{\widehat \Martin (h) (x)}{\delta(x)^{2s-\gamma - 1}} = \lim_{x \to z} \frac{ \Martin (h / K) (x)}{u^\star(x)} \frac{u^\star(x)}{\delta(x)^{2s-\gamma - 1}} = \frac{h(z)}{K(z)} K(z) = h(z)
	\end{equation*}

	\item \label{it:singular solution from differentiation}
	In the case of the RFL, the existence of solutions which are singular at boundary  can be obtained by taking the derivative of regular solutions. In Appendix \ref{appendix:Ros-Oton} we include an account of how positive singular solutions can be obtained, which was explained to us by Ros-Oton. It is based on a very interesting formula of chain-rule type.
	
	However, this argument does not seem to apply in general. In particular, it could fail in those examples where the commutation with the derivative does not hold, so that the singular rates cannot be predicted by such means. For the SFL it is easy to see that we cannot repeat the reasoning: in one dimension (say~$\Omega=(-1,1)$), we take the first eigenfunction for the SFL~$u(x) = \cos \frac {\pi x} 2$. Then, all derivatives are bounded functions and no singularity appears. However, we have shown that the blow-up rate of the critical solution is~$\delta^{2(s-1)}$.

	\item It is interesting to point out that, when~$f \asymp \delta^{-2s}$ (and it is admissible in the sense of~\eqref{eq:admissible class introduction}, \textit{i.e.},~$2s < \gamma +1$), then our main result Theorem~\ref{thm:range of exponents} says that~$\Green(f) \asymp 1$.
	
	Given~$g \in L^\infty( \partial \Omega ) $ it is therefore natural to ask whether there exists a function~$f$ such that~$\Green(f)(x) \to g(z)$ as~$x \to z \in \partial\Omega$.
	
	This would amount to studying whether the non-homogeneous Dirichlet problem
	\begin{equation*}
		\begin{dcases}
			\Ls u = f & \text{in } \Omega \\
			u = g & \text{on } \partial \Omega \\
			u = 0 & \text{in } \mathbb R^n\setminus\overline\Omega \text{ (if applicable) }
		\end{dcases}
	\end{equation*}
	has solutions for~$g$ bounded. When~$2s - \gamma - 1 < 0$ the solution of such problem will satisfy~$\lim_{x \to \partial \Omega} u / u^\star = 0$ on $\partial \Omega$. This would indicate that~$u = \Green (f)$. If such~$u$ exists, it will never be unique since, taking~$f_2 \in C^\infty_c(\Omega)$, $\hat u = \Green (f + f_2) = u + G(f_2)$ will also go to~$g$ at the boundary.
	Hence, it seems that, for  operators~$\Ls$ with a Green kernel satisfying~$2s - \gamma - 1 < 0$, $f$ and~$g$ cannot be chosen independently. In fact, for compactly supported~$f$ the only possible bounded~$g$ is zero, in complete contrast to the problem for the classical Laplacian.

 	In many cases this inverse task of finding one or several~$f$ given~$g$ turns out to be simple. If, for instance, the direct operator~$\Ls$ is given by a singular integral, then given some bounded smooth boundary data~$g$ we can extend them to the interior of~$\Omega$ as a smooth function~$\widetilde g$, and by zero outside. Then, we can take~$u=\widetilde g$ and compute~$f=\Ls u$ in~$\Omega$ explicitly, for~$\Omega$ of class~$C^{1,1}$. This construction is particularly enlightening when~$g = 1$. The natural extension to the inside is
 		\begin{equation*}
 			u (x) = \widetilde g (x)
 			=
 			\begin{dcases}
 				1 & \text{in } x\in \overline \Omega, \\
 				0 & \text{in } x \in \mathbb R^n\setminus\overline\Omega.
 			\end{dcases}
 		\end{equation*}
 		When $\Ls$ is the RFL, the computation yields, for $x \in \Omega$,
		\begin{multline*}
			f_\RFL(x)
			= (-\Delta)_\RFL^s u(x) \ =\\
			= c_{n,s} \, \pv \int_\Omega \frac{1-1}{|x-y|^{n+2s}} \; dy + c_{n,s}  \int_{\mathbb R^n\setminus\overline\Omega}  \frac{1-0}{|x-y|^{n+2s}} \; dy
			\asymp 	\delta(x)^{-2s}.
		\end{multline*}
		The last computation is a simple although technical exercise.
		Notice that~$f$ is in~$L^1 (\Omega)$ if~$s<1/2$ and in~$L^2(\Omega)$ if~ $s<1/4$.
		
		For the SFL we use the kernel representation and deduce, for $x \in \Omega$,
		\begin{align*}
			f_\SFL(x) = (-\Delta)_\SFL^s u(x) =
			\pv\int_\Omega (1 - 1) \, J(x,y) \; dy + \kappa (x)
			=\kappa (x) \asymp \delta(x)^{-2s},
		\end{align*}
		\textit{cf.}~\eqref{spectral estimates}.
		Curiously, in the CFL (which satisfies~$2s-\gamma - 1 = 0$) case, the~$\Ls$-harmonic problem has~$u^\star \asymp 1$ and we get the non-homogeneous Dirichlet problem.  Thus for~$u=1$ one trivially has for all~$x \in \Omega$,
		\begin{equation*}
			f_\CFL (x) = (-\Delta)_\CFL^s u(x) = c_{n,s} \, \pv \int_\Omega \frac{(1-1)}{|x-y|^{n+2s}} dy = 0.
		\end{equation*}
		This shows, in particular, that~$u  = 1$ is CFL-harmonic.
		\end{enumerate}




\appendix

\section{Derivation of an explicit singular solution by X. Ros-Oton}
\label{appendix:Ros-Oton}
This section is the result of conversations  with Prof.\ X.\ Ros-Oton who had indicated to the authors that, at least in the case of the restricted fractional Laplacian (RFL), some singular solutions could be obtained by differentiation of the continuous solutions of the standard theory as developed in \cite{Ros-Oton2014b}, 
see also comment number~\ref{it:singular solution from differentiation} in Section~\ref{sec:comments} above.
The objection we made that the solutions obtained by plain differentiation may change sign  was taken into account. In this section, for the sake of clarity we drop the subindex~$\RFL$: $(-\Delta)^s = (-\Delta)^s_\RFL$.

This is the way his argument proceeds in four  steps. Working in arbitrary dimension, we consider 
functions $u\in W^{1,1}(\mathbb R^N)$. 
First, we need the interesting identity
\begin{equation}
(-\Delta)^s (x\cdot \nabla u)= x\cdot \nabla (-\Delta)^s u + 2s (-\Delta)^s u, \quad \text{in } \ \mathbb R^N.
\end{equation}
This identity is proved in \cite[Proof of Lemma 5.1]{Ros-Oton2014b} by calculations with integrals, but we suggest the reader to do it as an exercise, by taking Fourier transforms and manipulating the resulting formula.

Then we need the result by Getoor (see \cite[Section 3]{getoor} or \cite{Biler}) that applies to a very particular continuous solution
$$
(-\Delta)^s (1-|x|^2)_+^s= C>0 \qquad \text{in }  B_1\subset \mathbb R^N.
$$
with $C=C(N,s)>0.$ The next step is new and goes as follows. If we call $U(x)=(1-|x|^2)_+^s$ and put $V=x\cdot \nabla U$
we conclude that
$$
(-\Delta)^s  V= (-\Delta)^s (x\cdot \nabla U)= x\cdot \nabla  C + 2s  C=2sC
$$
holds in $B_1$. Finally, we consider $W=2sU-V$ and get $(-\Delta)^s W=0$ in $B_1.$ Moreover, inside $B_1$ we get
$$
W=  2sU -x\cdot \nabla U= \frac{2s(1-|x|^2)}{(1-|x|^2)^{1-s}}+\frac{2s|x|^2}{(1-|x|^2)^{1-s}}=
\frac{2s}{(1-|x|^2)^{1-s}}\,,
$$
and this is a singular solution up to a harmless constant. This is precisely the singular solution provided by Hmissi \cite{MR1295711} and Bogdan \cite{bogdan99representation}.

\section{Proofs of Lemmas~\ref{lem:sharp interior} and~\ref{lem:sharp boundary}}
\label{appendix:technicalities}

\begin{proof}[Proof of Lemma~\ref{lem:sharp interior}]
We estimate
\begin{multline*}
\Green (\delta^\beta \chi_{\{ \delta < \eta \}})(x) \asymp
\int_{\{\delta(y)<\eta\}}\frac{\delta(y)^\beta}{|x-y|^{n-2s}}
\left( \frac{\delta(x) \delta(y)}{|x-y|^2} \wedge 1  \right)^\gamma\;dy\ \asymp \\
\asymp\ \delta(x)^\gamma\int_{\{\delta(y)<\eta/4\}}\frac{\delta(y)^{\beta+\gamma}}{|x-y|^{n-2s+2\gamma}} \;dy
+\int_{\{\eta/4<\delta(y)<\eta\}}\frac{\delta(y)^{\beta}}{|x-y|^{n-2s}} \;dy.
\end{multline*}
Using the co-area formula (see, \textit{e.g.},~\cite{federer1969geometric+measure+theory}), we have that
\begin{align}
\int_{\{\delta(y)<\eta/4\}}\frac{\delta(y)^{\beta+\gamma}}{|x-y|^{n-2s+2\gamma}} \;dy
&=
\int_0^{\eta/4} t^{\beta+\gamma}\int_{\{\delta(y)=t\}}\frac{dy}{|x-y|^{n-2s+2\gamma}}\;dt,
\label{outside-tubular-1}\\
\int_{\{\eta/4<\delta(y)<\eta\}}\frac{\delta(y)^{\beta}}{|x-y|^{n-2s}} \;dy
&=
\int_{\eta/4}^\eta t^{\beta}\int_{\{\delta(y)=t\}}\frac{dy}{|x-y|^{n-2s}}\;dt.
\label{outside-tubular-2}
\end{align}
Let us first deal with~\eqref{outside-tubular-1}.
The inner~$(n-1)$-dimensional integral is uniformly bounded in~$\eta$ whenever~$-2s+2\gamma<-1$, which is~$\gamma<s-\frac12$.
The integration in the~$t$ variable concludes then the claimed estimate.
If instead~$\gamma>s-\frac12$, then we have
\begin{align*}
\int_{\{\delta(y)<\eta/4\}}\frac{\delta(y)^{\beta+\gamma}}{|x-y|^{n-2s+2\gamma}} \;dy
& \asymp\ \int_0^{\eta/4} t^{\beta+\gamma} \big(\delta(x)-t\big)^{2s-2\gamma-1}\;dt \\
& \asymp\ \delta(x)^{\beta-\gamma+2s}\int_0^{\eta/\delta(x)} t^{\beta+\gamma} (4-t)^{2s-2\gamma-1}dt \\
& \asymp\ \delta(x)^{\beta-\gamma+2s}\left(\frac{\eta}{\delta(x)}\right)^{\beta+\gamma+1}
=\eta^{\beta+\gamma+1}\delta(x)^{2s-2\gamma-1};
\end{align*}
as in this case it is~$2s-2\gamma-1<0$, then we conclude
\begin{align*}
\int_{\{\delta(y)<\eta/4\}}\frac{\delta(y)^{\beta+\gamma}}{|x-y|^{n-2s+2\gamma}} \;dy
\leq\eta^{\beta+2s-\gamma}
\end{align*}
up to universal constants.
Needless to say, the above also holds when~$\gamma=s-\frac12$ with the suitable modifications
and we get
\begin{align*}
\int_{\{\delta(y)<\eta/4\}}\frac{\delta(y)^{\beta+\gamma}}{|x-y|^{n-2s+2\gamma}} \;dy\leq
\eta^{\beta+\gamma+1}\,|\!\ln\eta|.
\end{align*}

Let us now give a close look at~\eqref{outside-tubular-2}.
The inner~$(n-1)$-dimensional integral is uniformly bounded in~$\eta$ whenever~$s>\frac12$,
then the integration in~$t$ is elementary.
If~$s<1/2$, reasoning as above we get
\begin{align*}
\int_{\{\eta/4<\delta(y)<\eta\}}\frac{\delta(y)^{\beta}}{|x-y|^{n-2s}} \;dy
\asymp\int_{\eta/4}^\eta t^{\beta} \big|\delta(x)-t\big|^{2s-1}\;dt
\asymp\eta^{\beta}\delta(x)^{2s},
\end{align*}
and, in case~$s=1/2$,
\begin{align*}
\int_{\{\eta/4<\delta(y)<\eta\}}\frac{\delta(y)^{\beta}}{|x-y|^{n-2s}} \;dy
\asymp\eta^{\beta}\delta(x)\,|\!\ln\delta(x)|.
\end{align*}

The above proves the first claim in the statement. Mind now that in the case~$\gamma<s-\frac12$
we could simply estimate
\begin{align*}
\Green (\delta^\beta \chi_{\{ \delta < \eta \}})(x) \leq
\delta(x)^\gamma\int_{\{\delta(y)<\eta\}}\frac{\delta(y)^{\beta+\gamma}}{|x-y|^{n-2s+2\gamma}} \;dy
\end{align*}
and the above analysis would then bear
\begin{align*}
\Green (\delta^\beta \chi_{\{ \delta < \eta \}})(x) \leq \eta^{\beta+\gamma+1}\delta(x)^\gamma,
\qquad\text{for }\delta(x)>\eta/2.
\end{align*}
\end{proof}

\begin{proof}[Proof of Lemma~\ref{lem:sharp boundary}]
From now on let~$x\in\{\delta<\eta/2\}$ be fixed. Call~$\Phi:B(x,1)\to B(0,1)$
a diffeomorphism satisfying
\begin{align}\label{boundary diff}
\begin{split}
& \Phi(\Omega\cap B(x,1))=B(0,1)\cap\{y\in\mathbb R^n:y\cdot e_n>0\} \\
& \Phi(y)\cdot e_n=\delta(y)\text{ for any }y\in B(x,1),\qquad
\Phi(x)=\delta(x)e_n.
\end{split}
\end{align}
Also, we are going to intensively use~\eqref{eq:K0},~\eqref{eq:K1}, and~\eqref{eq:K2}.
We split the estimate into the five regions
\begin{align*}
\Omega_1:=B\big(x,\delta(x)/2\big),\qquad &
\Omega_2:=\{y:\delta(y)<\eta\}\setminus B(x,1),\\
\Omega_3:=\{y:\delta(y)<\delta(x)/2\}\cap B(x,1),\qquad &
\Omega_4:=\{y:3\delta(x)/2<\delta(y)<\eta\}\cap B(x,1), \\
\Omega_5:=\{y:\delta(x)/2<\delta(y)<3\delta(x)/2\} & \cap \big(B(x,1) \setminus B(x,\delta(x)/2)\big).
\end{align*}
\begin{itemize}

\item For~$y\in\Omega_1$ we use that
\begin{align*}
\left(\frac{\delta(x)\delta(y)}{|x-y|^2}\wedge 1\right)\asymp 1
\end{align*}
so that
\begin{align}\label{omega1}
\begin{split}
\int_{\Omega_1}\G(x,y)\,\delta(y)^\beta \; dy & \asymp
\int_{\Omega_1}\frac{\delta(y)^\beta}{|x-y|^{n-2s}}\;dy \\
&\asymp\delta(x)^\beta\int_{B(x,\delta(x)/2)}\frac{dy}{|x-y|^{n-2s}}\asymp
\delta(x)^{\beta+2s}.
\end{split}
\end{align}

\item For~$y\in\Omega_2$ we use that
\begin{align*}
\left(\frac{\delta(x)\delta(y)}{|x-y|^2}\wedge 1\right)\asymp \frac{\delta(x)\delta(y)}{|x-y|^2}
\end{align*}
so that
\begin{align}\label{omega2}
\begin{split}
\int_{\Omega_2}\G(x,y)\delta(y)^\beta dy & \asymp
\delta(x)^\gamma\int_{\Omega_2}\frac{\delta(y)^{\beta+\gamma}}{|x-y|^{n-2s+2\gamma}}dy \\
& \asymp\ \delta(x)^\gamma\int_{\Omega_2}\delta(y)^{\beta+\gamma}dy
\ \asymp\ \eta^{\beta+\gamma+1}\delta(x)^\gamma.
\end{split}
\end{align}

\item For~$y\in\Omega_3$ we use that
\begin{align*}
\left(\frac{\delta(x)\delta(y)}{|x-y|^2}\wedge 1\right)\asymp \frac{\delta(x)\delta(y)}{|x-y|^2}
\end{align*}
so that
\begin{align*}
\int_{\Omega_3}\G(x,y)\,\delta(y)^\beta \; dy\asymp
\delta(x)^\gamma\int_{\Omega_3}\frac{\delta(y)^{\beta+\gamma}}{|x-y|^{n-2s+2\gamma}}\;dy.
\end{align*}
Applying the change of variable entailed by~$\Phi$ ---as defined in~\eqref{boundary diff}---,
we get
\begin{align*}
& \int_{\Omega_3}\G(x,y)\delta(y)^\beta dy \asymp
\delta(x)^\gamma\int_{\{0<z_n<\delta(x)/2\}\cap B(0,1)}\frac{z_n^{\beta+\gamma}}{\big(|\delta(x)-z_n|+|z'|\big)^{n-2s+2\gamma}}\,dz \\
& \asymp\ \delta(x)^\gamma\int_{\{|z'|<1\}}\int_0^{\delta(x)/2}\frac{z_n^{\beta+\gamma}}{\big(|\delta(x)-z_n|+|z'|\big)^{n-2s+2\gamma}}\,dz_n\,dz' \\
& \asymp\ \delta(x)^\gamma\int_0^1 t^{n-2} \int_0^{\delta(x)/2}\frac{z_n^{\beta+\gamma}}{\big((\delta(x)-z_n)+t\big)^{n-2s+2\gamma}}\,dz_n\,dt \\
& \asymp\ \delta(x)^{\beta+2s}\int_0^{1/\delta(x)} t^{n-2} \int_0^{1/2}\frac{h^{\beta+\gamma}}{\big((1-h)+t\big)^{n-2s+2\gamma}}\,dh\,dt \\
& \asymp\ \delta(x)^{\beta+2s}\int_0^{1/\delta(x)}  \frac{t^{n-2}}{\big(1+t\big)^{n-2s+2\gamma}}\,dt
  \asymp\ \delta(x)^{\beta+2s}\int_0^{1/\delta(x)}  \frac{dt}{\big(1+t\big)^{2-2s+2\gamma}} \\
& \asymp\ \delta(x)^{\beta+2s}\,
\left\lbrace\begin{aligned}
& 1 & & \text{if }1-2s+2\gamma>0 \\
& |\!\ln(\delta(x))| & & \text{if }1-2s+2\gamma=0 \\
& \delta(x)^{1-2s+2\gamma} & & \text{if }1-2s+2\gamma<0
\end{aligned}\right.
\end{align*}
which, rephrased, means
\begin{align}\label{omega3}
\int_{\Omega_3}\G(x,y) \, \delta(y)^\beta \; dy\asymp
\left\lbrace\begin{aligned}
& \delta(x)^{\beta+2s} & & \text{if }\gamma>s-\frac12, \\
& \delta(x)^{\beta+2s}\,|\!\ln(\delta(x))| & & \text{if }\gamma=s-\frac12, \\
& \delta(x)^{\beta+2\gamma+1} & & \text{if }\gamma<s-\frac12 .
\end{aligned}\right.
\end{align}

\item For the integration in~$\Omega_4$, we use that, for~$y\in\Omega_4$,
\begin{align*}
\left(\frac{\delta(x)\delta(y)}{|x-y|^2}\wedge 1\right)\asymp \frac{\delta(x)\delta(y)}{|x-y|^2}
\end{align*}
so that
\begin{align*}
& \int_{\Omega_4}\G(x,y)\,\delta(y)^\beta \; dy\asymp
\delta(x)^\gamma\int_{\{3\delta(x)/2<\delta(y)<\eta\}\cap B(x,1)}\frac{\delta(y)^{\beta+\gamma}}{\big|x-y\big|^{n-2s+2\gamma}}\,dy
\end{align*}
and, changing variables at the aid of~$\Phi$ as above,
\begin{align*}
& \int_{\Omega_3}\G(x,y)\,\delta(y)^\beta \; dy \ \asymp \\
& \asymp\ \delta(x)^\gamma\int_{\{3\delta(x)/2<z_n<\eta\}\cap B(0,1)}\frac{z_n^{\beta+\gamma}}{\big(|\delta(x)-z_n|+|z'|\big)^{n-2s+2\gamma}}\,dz \\
& \asymp\ \delta(x)^{\beta+2s}\int_{3/2}^{\eta/\delta(x)}\int_0^{1/\delta(x)}
\frac{h^{\beta+\gamma}\,t^{n-2}}{\big((h-1)+t\big)^{n-2s+2\gamma}}\;dt\;dh \\
& \asymp\ \delta(x)^{\beta+2s}\int_{3/2}^{\eta/\delta(x)}\int_0^{1/((h-1)\delta(x))}
\frac{h^{\beta+\gamma}\,r^{n-2}}{(h-1)^{1-2s+2\gamma}(1+r)^{n-2s+2\gamma}}\;dr\;dh \\
& \asymp\ \delta(x)^{\beta+2s}\int_{3/2}^{\eta/\delta(x)}
\frac{h^{\beta+\gamma}}{(h-1)^{1-2s+2\gamma}}
\int_1^{1/((h-1)\delta(x))}
\frac{dr}{(1+r)^{2-2s+2\gamma}}\;dh.
\end{align*}
At this point we need to separate into different cases, since
\begin{multline*}
\int_1^{1/((h-1)\delta(x))}\frac{dr}{(1+r)^{2-2s+2\gamma}} \ \asymp \\
\asymp\ \left\lbrace\begin{aligned}
& 1					 							& & \text{if }1-2s+2\gamma>0 \\
& |\!\ln((h-1)\delta(x))| 						& & \text{if }1-2s+2\gamma=0 \\
& (h-1)^{1-2s+2\gamma}\delta(x)^{1-2s+2\gamma} 	& & \text{if }1-2s+2\gamma<0
\end{aligned}\right.
\end{multline*}
which, rephrased, gives
\begin{align*}
& \int_{\Omega_3}\G(x,y) \, \delta(y)^\beta \; dy \ \asymp \\
& \asymp\ \delta(x)^{\beta+2s}\,
\left\lbrace\begin{aligned}
& \int_{3/2}^{\eta/\delta(x)}\frac{h^{\beta+\gamma}}{(h-1)^{1-2s+2\gamma}}\,dh	
& & \text{if }\gamma>s-\frac12 \\
& \int_{3/2}^{\eta/\delta(x)}h^{\beta+\gamma}|\!\ln((h-1)\delta(x))|\,dh \;\; 			
& & \text{if }\gamma=s-\frac12 \\
& \delta(x)^{1-2s+2\gamma}\int_{3/2}^{\eta/\delta(x)}h^{\beta+\gamma}\,dh 	
& & \text{if }\gamma<s-\frac12
\end{aligned}\right. \\
& \asymp\ \left\lbrace\begin{aligned}
& \delta(x)^{\beta+2s}\int_{3/2}^{\eta/\delta(x)}h^{\beta-\gamma+2s-1}dh	
& & \text{if }\gamma>s-\frac12 \\
& \delta(x)^{\beta+2s}\int_{3/2}^{\eta/\delta(x)}h^{\beta+\gamma}|\!\ln((h-1)\delta(x))|\,dh 	\quad	
& & \text{if }\gamma=s-\frac12 \\
& \delta(x)^{\beta+2\gamma+1}\int_{3/2}^{\eta/\delta(x)}h^{\beta+\gamma}\,dh 	
& & \text{if }\gamma<s-\frac12 .
\end{aligned}\right.
\end{align*}
Now, in the case~$\gamma>s-\frac12$ we have
\begin{align*}
& \int_{\Omega_3}\G(x,y)\delta(y)^\beta dy\asymp
\delta(x)^{\beta+2s}\left\lbrace\begin{aligned}
& \eta^{\beta-2\gamma+2s}\delta(x)^{-\beta+\gamma-2s} 	& & \text{if }\beta-\gamma+2s>0 \\
& |\!\ln(\eta/\delta(x))|								& & \text{if }\beta-\gamma+2s=0 \\
& 1													& & \text{if }\beta-\gamma+2s<0
\end{aligned}\right. \\
& =\ \left\lbrace\begin{aligned}
& \eta^{\beta-2\gamma+2s}\delta(x)^\gamma 			& & \text{if }\beta-\gamma+2s>0 \\
& \delta(x)^{\beta+2s}|\!\ln(\eta/\delta(x))|		& & \text{if }\beta-\gamma+2s=0 \\
& \delta(x)^{\beta+2s}							& & \text{if }\beta-\gamma+2s<0 ,
\end{aligned}\right.
\end{align*}
when~$\gamma=s-\frac12$
\begin{align*}
& \int_{\Omega_3}\G(x,y) \, \delta(y)^\beta \; dy\asymp
\delta(x)^{\beta+2s}\int_{3/2}^{\eta/\delta(x)}h^{\beta+\gamma}|\!\ln((h-1)\delta(x))|\,dh \\		
& \asymp\
\delta(x)^{2s-\gamma-1}\eta^{\beta+\gamma+1}\big|\!\ln((\eta/\delta(x)-1)\delta(x))\big|
\ \asymp\ \eta^{\beta+\gamma+1}\,|\!\ln\eta|\,\delta(x)^\gamma,
\end{align*}
whereas for~$\gamma<s-\frac12$ we have
\begin{align*}
\int_{\Omega_3}\G(x,y)\,\delta(y)^\beta \; dy\asymp
\delta(x)^{\beta+2\gamma+1}\left(\frac\eta{\delta(x)}\right)^{\beta+\gamma+1}=
\eta^{\beta+\gamma+1}\delta(x)^{\gamma}.
\end{align*}
Resuming the information collected about the integral over~$\Omega_4$,
we have the behaviour described in Table~\ref{omega4}.
\begin{table}[th]
	\centering
	\begin{tabular}{c|c|c|c}
	&~$\gamma<s-\frac12$ &~$\gamma=s-\frac12$ &~$\gamma>s-\frac12$ \\ [.2em]
	\hline & & & \\ [-.75em]
	$\beta<\gamma-2s$ &~$\eta^{\beta+\gamma+1}\delta(x)^{\gamma}$ &
	$\eta^{\beta+\gamma+1}\,|\!\ln\eta|\,\delta(x)^\gamma$ &~$\delta(x)^{\beta+2s}$ \\
	 & & & \\ [-.75em]
	$\beta=\gamma-2s$ &~$\eta^{\beta+\gamma+1}\delta(x)^{\gamma}$ &
	$\eta^{\beta+\gamma+1}\,|\!\ln\eta|\,\delta(x)^\gamma$ &~$\delta(x)^{\beta+2s}|\!\ln(\eta/\delta(x))|$ \\
	 & & & \\ [-.75em]
	$\beta>\gamma-2s$ &~$\eta^{\beta+\gamma+1}\delta(x)^{\gamma}$ &
	$\eta^{\beta+\gamma+1}\,|\!\ln\eta|\,\delta(x)^\gamma$ &~$\eta^{\beta-2\gamma+2s}\delta(x)^\gamma$
	\end{tabular}
	\caption{The integration over~$\Omega_4$.}\label{omega4}
\end{table}

\item For~$y\in\Omega_5$ we use that
\begin{align*}
\left(\frac{\delta(x)\delta(y)}{|x-y|^2}\wedge 1\right)\asymp \frac{\delta(x)\delta(y)}{|x-y|^2},
\end{align*}
so that
\begin{align*}
& \int_{\Omega_5}\G(x,y)\,\delta(y)^\beta \; dy\ \asymp \\
& \asymp\ \delta(x)^\gamma\int_{\{\delta(x)/2<\delta(y)<3\delta(x)/2\}\cap (B(x,1)\setminus B(x,\delta(x)/2))}\frac{\delta(y)^{\beta+\gamma}}{\big|x-y\big|^{n-2s+2\gamma}}\,dy \\
& \asymp\ \delta(x)^{\beta+2\gamma}\int_{\{\delta(x)/2<\delta(y)<3\delta(x)/2\}\cap (B(x,1)\setminus B(x,\delta(x)/2))}\big|x-y\big|^{-n+2s-2\gamma}\,dy
\end{align*}
and, applying the change of variable induced by the~$\Phi$ defined in~\eqref{boundary diff},
\begin{align*}
& \int_{\Omega_5}\G(x,y)\,\delta(y)^\beta \; dy \ \asymp \\
& \asymp\ \delta(x)^{\beta+2\gamma}\int_{\delta(x)/2}^{3\delta(x)/2}\int_{\delta(x)/2}^1
r^{n-2}\big(|\delta(x)-h|+r\big)^{-n+2s-2\gamma}\,dr\,dh \\
& \asymp\ \delta(x)^{\beta+2\gamma}
\int_{\delta(x)/2}^1 r^{2s-2\gamma-1} \int_{-\delta(x)/(2r)}^{\delta(x)/(2r)}
\big(|t|+1\big)^{-n+2s-2\gamma}\,dt\,dr \\
& \asymp\ \delta(x)^{\beta+2s}
\int_{1/2}^{1/\delta(x)} \rho^{2s-2\gamma-1} \int_0^{1/\rho}
\big(t+1\big)^{-n+2s-2\gamma}\,dt\,d\rho \\
& \asymp\ \delta(x)^{\beta+2s}
\int_{1/2}^{1/\delta(x)} \rho^{2s-2\gamma-2}\,d\rho \\
& \asymp \delta(x)^{\beta+2s}
\left\lbrace\begin{aligned}
& \delta(x)^{-2s+2\gamma+1} 	& & \text{if }2s-2\gamma-1>0 \\
& |\!\ln\delta(x)| 			& & \text{if }2s-2\gamma-1=0 \\
& 1 							& & \text{if }2s-2\gamma-1<0
\end{aligned}\right.
\end{align*}
meaning
\begin{align}\label{omega5}
\int_{\Omega_5}\G(x,y)\,\delta(y)^\beta \; dy\ \asymp\ \left\lbrace\begin{aligned}
& \delta(x)^{\beta+2\gamma+1}					& & \text{if }\gamma<s-\frac12 \\
& \delta(x)^{\beta+2s}|\!\ln\delta(x)| 		& & \text{if }\gamma=s-\frac12 \\
& \delta(x)^{\beta+2s}						& & \text{if }\gamma>s-\frac12 .
\end{aligned}\right.
\end{align}

\end{itemize}
\end{proof}


\section*{Acknowledgments}
The research of DGC and JLV was partially supported by grant PGC2018-098440-B-I00 from the Ministerio de Ciencia, Innovación y Universidades of the Spanish Government. NA was partially supported by the Alexander von Humboldt Foundation.

The authors would like to thank the Instituto of Matemática Interdisciplinar at UCM, for supporting a visit of NA to Madrid. 
	JLV is grateful for the hospitality of Prof.~Figalli during his visit to ETH 
	and Prof.~Ros-Oton for fruitful conversations. 
	The authors are grateful to Prof.~Vondra\v cek for 
	comments that improved the paper. 
	Finally, we would like to thank Dr.\ del Teso for his advice on the numerics.

\end{document}